\documentclass{article}
\usepackage{preamble} 

\newcommand{\sd}{\operatorname{sd}}
\newcommand{\dd}{\operatorname{d}}
\newcommand{\phic}{{\phi^\circ}}

\newcommand{\Et}{E^{\tau}}
\newcommand{\Etn}{E^{\tau_n}}

\newcommand{\lt}{\lambda^{\tau}}
\newcommand{\kp}{\kappa^\phi}
\newcommand{\okp}{\ol{\kappa}^\phi}
\newcommand{\sdp}{\sd^\phic}
\newcommand{\nup}{\nu^\phi}
\newcommand{\tI}{\tilde{I}}
\newcommand{\tN}{\tilde{\cN}}
\newcommand{\co}{\operatorname{co}}
\newcommand{\tGamma}{\tilde{\Gamma}}

\title{Area-preserving crystalline curvature flow in two dimensions}
\author{Eric Kim}

\begin{document}

\maketitle

\begin{abstract}
    We study the dynamics of planar sets under area-preserving crystalline curvature flow. We prove under mild assumptions on $\phi$ that the flat flow solution from regular initial data coincides with a classical ODE evolution, extending the results of \cite{Almgren1995} to the area-preserving setting. We also show that for arbitrary initial data, the flat flow converges exponentially in time to a disjoint union of Wulff shapes, and under a non-bubbling assumption, the flow eventually becomes regular. Both of these results are novel for area-preserving crystalline flow of general sets, \emph{i.e.} without assuming geometric properties such as convexity or star-shapedness. A key ingredient of independent interest is that planar almost-minimizers are Lipschitz $\phi$-regular, which we prove by exploiting a sharp minimality estimate for distinguished line segments, as opposed to the excess decay argument given in \cite{Ambrosio2002regularity}. The novelty of our approach lies in the application of $\phi$-minimal barriers for energy competition arguments, both for the geometric rigidity of the discretized flat flow and for the regularity of almost-minimizers.
\end{abstract}

\section{Introduction}
We study the dynamics of \emph{volume-preserving $\phi$-mean curvature flow}, where $\phi$ is a convex, positively 1-homogeneous function giving rise to an anisotropic perimeter functional
\begin{equation*}
    P_\phi(E) := \int_{\partial^* E} \phi(\nu_E) d\cH^{n-1}.
\end{equation*}
Formally, we can express the flow as an evolution of sets $E_t\sb\R^n$ given by
\begin{equation}
    \label{eq:cr flow}
    V_t := \phi(\nu_{E_t}) (-\kappa^\phi_{E_t} + \okp_{E_t})
\end{equation} 
where
\begin{itemize}
    \item $V_t$ denotes the normal velocity along $\partial E_t$, and $\nu_{E_t}$ is the outer normal vector
    \item $\kappa^\phi_{E_t}$ is the \emph{$\phi$-mean curvature} of $\partial E_t$
    \item $\phi(\nu_{E_t})$ encodes that the mobility of interface is proportional to the surface energy density $\phi$
    \item $\okp_{E_t} := \inv{P_\phi(E_t)} \int_{\partial^* E_t} \kp_{E_t} dP_\phi$ is a Lagrange multiplier enforcing that the areas $|E_t|$ remain constant
\end{itemize}
Our work solely treats the case where $\phi$ is \emph{crystalline} (\emph{i.e.} piecewise linear) in the ambient dimension $n=2$ (in which case we call the flow \emph{area}-preserving), although we consider general $\phi$ and $n$ for the sake of discussing prior works. For technical reasons, and for all the results in our paper, we assume 
\begin{align}
    \label{eq:phi assumption}
    \tag{A1}
    \text{adjacent vectors in $\cN_\phi$ differ by less than 90 degrees}
\end{align}
where $\cN_\phi$ is the set of outer normals to the Wulff shape $W_\phi$ (see \eqref{eq:wulff}). Notably $\phi$ does not need to be even.

\begin{figure}
    \centering
    \includegraphics[width=0.6\linewidth]{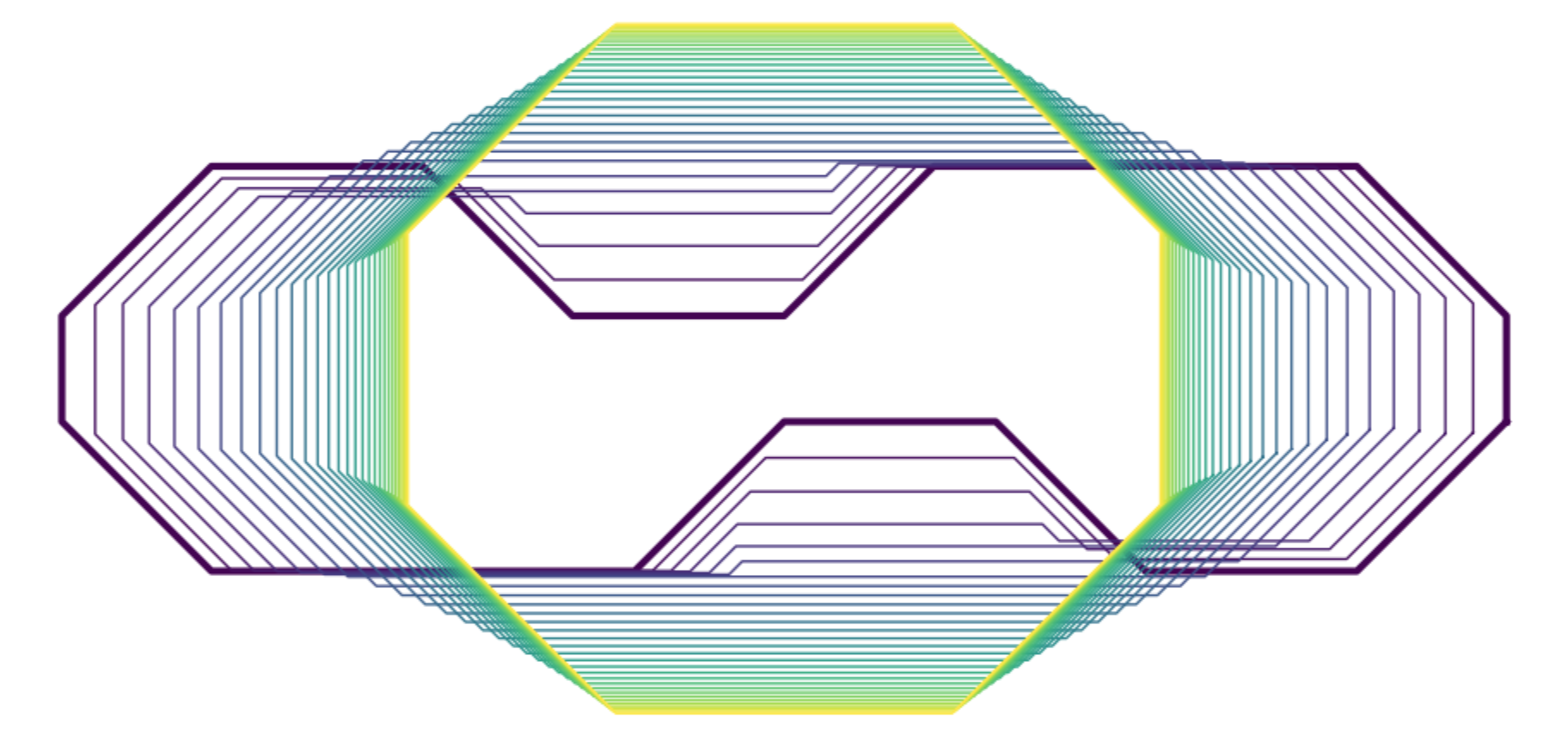}
    \caption{Area-preserving crystalline curvature flow converging to an octagonal Wulff shape}
    \label{fig:example flow}
\end{figure}

The mean curvature $\kp_E$ represents the first variation of $P_\phi$ with respect to volume perturbations. Thus the flow \eqref{eq:cr flow} can be interpreted as a (formal) $L^2$-gradient flow of the energy functional $P_\phi$ within sets of prescribed volume. As such, it is natural to expect that the flow eventually converges to a critical point of $P_\phi$. Volume-constrained minimizers of $P_\phi$ are uniquely determined (up to scaling and translation) by the Wulff shape 
\begin{equation}
    \label{eq:wulff}
    W_\phi := \{x\in\R^n: x\cdot\nu\leq \phi(\nu)\ \forall \nu\in\R^n\}
\end{equation}
while critical points are disjoint unions of equally sized Wulff shapes. 

It should be noted that the definition of $\phi$-mean curvature is much more complicated in the crystalline setting. For smooth $\phi$, one can define $\kp_E = \div_\tau \nabla \phi(\nu_E)$ in the distributional sense. For crystalline $\phi$, the first variation of $P_\phi$ is no longer a linear functional on vector field perturbations due to $\nabla\phi(\nu_E)$ being replaced by the subdifferential $\partial\phi(\nu_E)$, and thus $\kp_E$ is not defined in the distributional sense. Moreover, when defined, $\kp_E$ is a nonlocal quantity that only bears a loose interpretation as a first variation of $P_\phi$. Nonetheless, we will ignore this subtlety for the introduction.

Solutions to \eqref{eq:cr flow}, even with regular initial data, are known to encounter topological singularities such as vanishing, pinching, and collision of boundary components. Thus, in order to study long-time behavior, one requires a suitable notion of weak solution that exists globally in time. There is a long history of works establishing the well-posedness of anisotropic and crystalline mean curvature flows with prescribed forcing, such as level-set viscosity solutions \cite{GigaGiga1996, GigaGiga1998, GigaGiga2001, Giga2009facetbending, GigaPozar2020nonuniformforcing}; distance function formulations using minimizing movement schemes \cite{ATW1993, ChambolleNovaga2015planar, Chambolle2017crystalline, Chambolle2019anisotropic, Chambolle2019, Chambolle2020meanconvex}; and distributional BV solutions via diffuse-interface approximation \cite{Laux2024}. We also refer the interested reader to the survey \cite{GigaPozar2021survey} on crystalline flows. Most of these results rely heavily on the comparison principle. However, the volume-preserving flow \eqref{eq:cr flow} does not satisfy a comparison principle due to the nonlocality of the forcing term $\okp_E$. In the isotropic case ($\phi=|\cdot|$), the works \cite{Takasao2023} and \cite{Poiatti2026} established the existence of weak varifold solutions for the volume-preserving flow by using a diffuse-interface approximation.

The notion of solution we consider for \eqref{eq:cr flow} is the \emph{flat flow}, which naturally incorporates the gradient flow structure and yields evolutions which persist globally in time (in particular, past topological singularities). The scheme was first introduced by Almgren-Taylor-Wang \cite{ATW1993} and Luckhaus-Sturzenhecker \cite{Luckhaus1995} for the unforced mean curvature flow, then adapted to the volume-preserving curvature flow in the isotropic case by \cite{Mugnai2016} and in the anisotropic case with smooth elliptic $\phi$ by \cite{KimKwon2024}. In short, the flat flow is defined as a limit of minimizing movement schemes which approximate \eqref{eq:cr flow} by iterating the minimization problem $E \in \argmin \cF_\tau(\cdot,F)$, where 
\begin{equation}
    \label{eq:flat flow energy}
    \cF_\tau(E,F) := P_\phi(E) + \inv{\tau} \underbrace{\int_{E\Delta F} \dd_F^\phic(x) dx}_{\cD(E,F)} + \inv{\sqrt{\tau}}||E|-|E_0||.
\end{equation}
Here $\dd_F^\phic(x)$ is the distance of $x$ from $\partial F$ with respect to the dual anisotropy $\phic$, which encodes the $\phi$ mobility factor. The dissipation term $\cD(E,F)$ roughly represents a squared $\phi$-weighted $L^2$ distance between $E$ and $F$, while the last term penalizes deviations in volume.

\subsection{Weak-strong uniqueness}

A natural question is whether the flat flow coincides with a classical solution to \eqref{eq:cr flow} for regular initial data. In the isotropic case, \cite{JulinNiinikoski2023consistency} showed the consistency of the flat flow with the classical evolution for $C^{1,1}$ initial data. A similar result is shown for smooth elliptic $\phi$ in $n=2$ by \cite{Kubin2025} for a slight variant of the minimizing movement scheme considered in \eqref{eq:flat flow energy}. 

First let us consider what is a suitable notion of ``regularity" in the crystalline setting. Since the Wulff shape has corners, we cannot expect any better than Lipschitz regularity. However, \cite{Bellettini2001} introduced the notion of Lipschitz $\phi$-regularity, which is analogous to $C^{1,1}$ regularity in the isotropic setting, and defined a suitable analogue of $\kp_E$ for such sets. A Lipschitz set $E$ is \emph{Lipschitz $\phi$-regular} if there exists a Lipschitz vector field $X$ on $\partial E$ satisfying $X \in \partial\phi(\nu_E)$ almost everywhere. For $n=2$, the boundaries of Lipschitz $\phi$-regular sets decompose into curves along which $\kp_E$ is constant, which we call \emph{edges}. One can express the flow \eqref{eq:cr flow} as a system of ODEs tracking the evolution of each edge $I$, which we roughly state as
\begin{equation}
    \label{eq:intro ODE flow}
    \phic(N_I) = - \kp_E(I) + \okp_E
\end{equation} 
where $N_I$ is the normal velocity of $I$. Similar ODE evolutions have been considered for polygons in 
\cite{Taylor1993, Almgren1995, Yazaki2007, Giga2022nonadmissible}. We also note that \eqref{eq:intro ODE flow}, with a prescribed forcing term rather than $\okp_E$, is essentially equivalent to the evolution developed in \cite{ChambolleNovaga2015planar}.

We now informally state our first main result, whose precise statement is given in \cref{thm:weak strong uniqueness}.

\begin{theorem}[Weak-strong uniqueness]
    \label{thm:weak strong informal}
    Let $E_0\sb\R^2$ be a Lipschitz $\phi$-regular set. Then any flat flow solution to \eqref{eq:cr flow} starting from $E_0$ generically coincides with the ODE evolution \eqref{eq:intro ODE flow} for as long as the latter exists. 
\end{theorem}

By the two solution notions coinciding for ``generic" initial data, we mean that inflection-type edges (\emph{type II fuzzy edges} defined in \cref{sec:lipreg}) which may vanish in the ODE evolution must have distinct vanishing times. A similar assumption is made in \cite{Almgren1995}.

Existing results in the literature that are most similar to \cref{thm:weak strong informal} are \cite{Almgren1995}, which proves the consistency of the flat flow and ODE evolution for the unforced curvature flow of admissible polygons, and \cite{Cicalese2024}, which shows this consistency for the area-preserving flat flow of rectangular initial data with $\phi=|\cdot|_1$. For initial data satisfying an interior/exterior $rW_\phi$-property (which is equivalent to Lipschitz $\phi$-regularity in $n=2$), it has also been shown that this property is preserved for a short amount of time by the area-preserving flow of convex sets \cite{Bellettini2006} and the planar flow with prescribed forcing \cite{ChambolleNovaga2015planar}. Our result is the first such consistency result for the area-preserving flow with general initial data. 

Our proof strategy for \cref{thm:weak strong informal} is similar to that of \cite{Almgren1995}, wherein the authors use corner barriers to show that the minimizing movement scheme coincides with translates of the initial data. We adopt a similar approach by extending the class of barriers to those with \emph{$\phi$-minimal} boundaries. The key difficulty in our setting is that we are limited to area-preserving energy competitors, complicating the setup for barrier arguments: rather than deleting discrepancies between $E$ and $F$, we must shift mass from regions further from $F$ to regions closer to $F$.  

A further difficulty introduced by considering Lipschitz $\phi$-regular sets rather than admissible polygons is that the vector field associated to \eqref{eq:intro ODE flow} is only of BV regularity rather than Lipschitz. For instance, a necessary ingredient for \cref{thm:weak strong informal} is the stability of the ODE evolution with respect to Hausdorff distance (\cref{prop:ODE stability hausdorff}). From the perspective of DiPerna-Lions theory, this amounts to showing $L^\infty$ stability of flow maps with respect to $L^1$ proximity of BV vector fields, which of course is not true in general. To show the stability, we rely on the fact that nonzero curvature edges have relative velocity (i.e. modulo $\okp_E$) bounded away from zero.

\subsection{Long-time behavior}

Prior works have determined for essentially arbitrary bounded initial data that flat flow solutions converge exponentially to a disjoint union of Wulff shapes for the isotropic case $\phi=|\cdot|$ in dimensions $n=2,3$ \cite{Julin2022, JulinNiinikoski2023qat, Julin2025} and for smooth elliptic $\phi$ in $n=2$ \cite{KimKwon2024}. 

In the crystalline setting, convergence results for \eqref{eq:cr flow} are known only in special cases, for instance when the initial set is convex \cite{Bellettini2009, Yazaki2007} or satisfies certain reflection symmetries \cite{KimKwonPozar2022}.
In these cases, the geometric property is preserved over time, prohibits the onset of topological singularities, and is an important ingredient for controlling the Lagrange multiplier to yield existence for \eqref{eq:cr flow}: the authors in \cite{Bellettini2009, KimKwonPozar2022} initially consider the flow with prescribed forcing, and then bootstrap the forcing to preserve volume. It is unclear whether one can apply this bootstrapping approach for general initial data.

Our second main result establishes convergence of the flat flow for any bounded initial data for crystalline $\phi$ in $n=2$:
\begin{theorem}[Long-time behavior]
    \label{thm:convergence flat flow cr}
    Let $E_0\sb\R^2$ be a set of finite perimeter and $(E(t))_{t\in[0,\infty)}$ a flat flow solution to \eqref{eq:cr flow} with initial data $E_0$. There exists a disjoint union of equally sized Wulff shapes $E_\infty := \cup_{j=1}^d (x_j + rW_\phi)$ satisfying $|E_\infty| = |E_0|$ and some constant $C>0$ such that
    \begin{equation}
        \label{eq:exp L1 convergence}
        |E(t) \Delta E_\infty| \leq Ce^{-t/C}.
    \end{equation}
    Moreover, if there is a positive distance between the Wulff shapes in $E_\infty$ (non-bubbling), then there exists $T>0$ such that $E(t)$ satisfies the ODE flow \eqref{eq:intro ODE flow} for $t\geq T$, and one has the stronger convergence
    \begin{equation}
        \label{eq:exp hausdorff convergence}
        d_H(E(t), E_\infty) \leq Ce^{-t/C}.
    \end{equation}
\end{theorem}

A key difference between \cref{thm:convergence flat flow cr} and the statement in \cite{KimKwon2024} is the eventual regularity of the flat flow under a non-bubbling assumption, which remains open in the case of smooth elliptic $\phi$. In the isotropic setting, this was recently proven by \cite{Julin2026} for $n=2,3$, also under a non-bubbling assumption, using an adaptation of Brakke's regularity theorem to the flat flow. Moreover, without the non-bubbling assumption, we only establish $L^1$ convergence, whereas the results in \cite{KimKwon2024, Julin2022, Julin2025} guarantee Hausdorff convergence by proving an equicontinuity result when the flat flow is close to a critical point. We faced an obstruction in adapting this equicontinuity result to the crystalline setting because the Euler-Lagrange equation for minimizers of \eqref{eq:flat flow energy} holds only in an averaged sense along edges and thus does not convey pointwise information about the disretized velocity.

The main dynamical ingredient for \cref{thm:convergence flat flow cr}, as in the prior works \cite{Julin2022, Julin2025, KimKwon2024}, is a \emph{Quantitative Alexandrov Theorem} (QAT) which establishes that sets with almost constant $\phi$-mean curvature are quantifiably close to a disjoint union of Wulff shapes. Our QAT roughly states that if $E\sb\R^2$ is a bounded Lipschitz $\phi$-regular set with $\kp_E$ close in $L^2$ to a constant, then $E$ is a small perturbation of a disjoint union of Wulff shapes $F$ satisfying
    \begin{equation}
        \label{eq:Lojasiewicz}
        \abs{P_\phi(E) - P_\phi(F)} \lesssim \|\kappa_E^\phi - \ol{\kappa}_E^\phi\|_{L^2(\partial E)}^2.
    \end{equation}
The inequality \eqref{eq:Lojasiewicz} plays the role of a {\L}ojasiewicz inequality, which is classically used to quantify convergence rates of gradient flows to a critical point. The proof of the crystalline QAT follows the same ideas as in \cite{Julin2022, KimKwon2024} and is in fact simpler due to a small $L^2$ deviation of $\kp_E$ prohibiting the existence of negatively curved edges. We also mention the recent work \cite{Bellettini2025elastic} which shows the convergence of the crystalline elastic flow also using a \L ojaziewicz inequality.

\subsection{Almost-minimizers of $P_\phi$}
\label{sec:almost min}

A key ingredient in our analysis is that almost-minimizers of $P_\phi$ are Lipschitz $\phi$-regular (\cref{thm:lipreg}). We recall that $E\sb\R^2$ is an \emph{almost-minimizer} of $P_\phi$ if there exist constants $\Lambda, r_0 > 0$ such that 
\begin{equation}
    \label{eq:almost minimal}
    P_\phi(E) \leq P_\phi(G) + \Lambda |E\Delta G| \qquad \text{whenever } E\Delta G \Subset B(x,r_0) \text{ for some }x\in\partial E.
\end{equation}
In particular, minimizers of \eqref{eq:flat flow energy} are almost-minimizers and thus are Lipschitz $\phi$-regular. This regularity is crucial to our analysis: it implies the discretized flat flow satisfies an Euler-Lagrange equation and is necessary for employing the crystalline QAT to prove \cref{thm:convergence flat flow cr}.

It has been known from \cite{Ambrosio2002regularity} that whenever $|\cN_\phi| > 4$, minimizers of the energy
\begin{equation}
    P_\phi(E) + \int_E g(x) dx
\end{equation}
are Lipschitz $\phi$-regular for any bounded potential $g$, and their arguments can be adapted to minimizers of \eqref{eq:flat flow energy}. Nonetheless, we provide an independent proof relying on a different set of ingredients. The authors in \cite{Ambrosio2002regularity} consider an anisotropic notion of excess, a quantity measuring the local ``flatness" of the boundary, and proceed via a conventional excess decay argument. We instead exploit a sharp minimality result (\cref{prop:sharp minimality}) for line segments with outer normal in $\cN_\phi$, which states that any deviation from such a line segment incurs an increase in surface energy that is linear in the height of the deviation. The proof of \cref{prop:sharp minimality} is a short calibration argument and yields a simpler proof of Lipschitz $\phi$-regularity than the argument in \cite{Ambrosio2002regularity}; in particular, there is no blow-up analysis. 

\subsection{Open questions}

One advantage of the simplicity of the ODE evolution is that one has a concrete picture on how to restart the flow past singular events. For instance, in the event of a collision between boundary components, the resulting set will no longer be Lipschitz $\phi$-regular due to the formation of sharp (non-admissible) corners. For the unforced curvature flow, \cite{Giga2022nonadmissible} show that one can define the evolution from non-admissible corners by inserting zero-length edges to enforce admissibility. We expect the consistency between the flat flow and ODE flow to hold past such collision events, but we open this line of inquiry to future work.

For general mobilities, our methods still imply the exponential $L^1$ convergence of the flat flow to a critical point. However, the weak-strong uniqueness and eventual regularity statements are likely false without a compatibility condition between the mobility and $\phi$ (see $\phi$-regular mobilities in \cite{Chambolle2019anisotropic}) ensuring that corners move at least as fast as their neighboring facets.

In dimensions $N\geq 3$, crystalline curvature flow is known to encounter facet bending and breaking \cite{Bellettini2004facetbreaking}. Moreover, the regularity of almost-minimizers is not well-understood, which is a significant obstacle to adapting the minimizing movement scheme to approximate the evolution. Nonetheless, we anticipate that $\phi$-minimal barriers may still be used to obtain some form of weak-strong uniqueness for regular initial data.

\subsection{Organization}
In \cref{sec:crystalline geometry} we introduce the notions of Lipschitz $\phi$-regular sets and $\phi$-minimal paths, and prove that almost-minimizers of $P_\phi$ are Lipschitz $\phi$-regular. \cref{sec:flow solutions} develops existence and essential estimates for the ODE flow and flat flow solutions. Theorems \ref{thm:weak strong informal} and \ref{thm:convergence flat flow cr} are proven in Sections \ref{sec:weak strong uniqueness} and \ref{sec:long time}, respectively.

\subsection{Acknowledgements}
The author thanks Inwon Kim for many helpful discussions and Dohyun Kwon for suggesting this problem. The author was partially supported by the National Science Foundation grant DMS-2452649.

\section{Notation/Preliminaries}
\textbf{Sets of finite perimeter.} We recall that a Lebesgue measurable set $E\sb\R^2$ is a \emph{set of finite perimeter} if its distributional gradient $D\chi_E$ is a Radon measure. We refer to the reduced boundary of $E$ as $\partial^*E$ and the surface measure $\mu_E := -D\chi_E = \nu_E\,\cH^1|_{\partial^*E}$. For $\theta\in[0,1]$, we let $E^{(\theta)}$ denote the points of Lebesgue density $\theta$, and recall the theorem of Federer that $\R^2\setminus(E^{(0)}\cup E^{(1)})$, $E^{(1/2)}$, $ \partial^*E$ are all equivalent up to a $\cH^1$-null set. We will always identify a set of finite perimeter with its set of Lebesgue points. For further background on sets of finite perimeter, we refer to \cite{Maggi2012}. 

Given an anisotropy $\phi$, a set of finite perimeter $E$, and a Borel set $A\sb\R^2$, we define the anisotropic surface energy of $E$ relative to $A$ by
\[ P_\phi(E;A) := \int_{\partial^*E\cap A} \phi(\nu_E)\,d\cH^1. \]
We abbreviate $P_\phi(E) := P_\phi(E;\R^N)$.

We let $R:\R^2\to\R^2$ denote counterclockwise rotation by 90 degrees. For a simple oriented rectifiable curve $\Gamma\sb\R^2$, we define its outer normal vector to be $\nu_\Gamma(\gamma(t)) = R\gamma'(t)$ where $\gamma$ is an arclength parametrization of $\Gamma$, and let $\Gamma^*$ denote the set of points in $\Gamma$ at which $\nu_\Gamma$ is defined. We thus define 
\begin{equation}
    P_\phi(\Gamma) := \int_{\Gamma^*} \phi(\nu_\Gamma)d\cH^1.
\end{equation}
All curves considered in this paper are \emph{oriented}. In particular, for points $x,y\in\R^2$, $[x,y]$ will denote the oriented line segment from $x$ to $y$.

We use $|E|$ to denote Lebesgue measure in $\R^2$. However, when there is no risk of confusion, we also denote $|I| := \cH^1(I)$ when $I$ is a line segment.

When $E,F\sb\R^2$ are two sets of finite perimeter, we adopt the shorthand \[ \{\nu_E = \pm \nu_F\} := \set{x\in\partial^*E\cap \partial^*F: \nu_E(x) = \pm \nu_F(x)}. \]
    We recall the surface measures for basic set operations \cite[Theorem 16.3]{Maggi2012} between two sets of finite perimeter $E,F$:
    \begin{align}
        \label{eq:intersection}
        \mu_{E\cap F} &= \mu_E|_{F^{(1)}} + \mu_F|_{E^{(1)}} + \nu_E \cH^{N-1} |_{\{\nu_E=\nu_F\}}\\
        \label{eq:setminus}
        \mu_{E\setminus F} &= \mu_E|_{F^{(0)}} - \mu_F|_{E^{(1)}} + \nu_E \cH^{N-1} |_{\{\nu_E=-\nu_F\}}
    \end{align} 

\medskip 

\noindent \textbf{Crystalline anisotropy.}
We fix $\phi$ to be a \emph{crystalline} anisotropy on $\R^2$, i.e. for some finite set of nonzero vectors $\tN_\phi := \{p_1, \dots, p_N\} \sb \R^2$, 
\begin{equation}
    \label{eq:phi def}
    \phi(\nu) := \max_{1\leq i\leq N} \nu\cdot p_i. 
\end{equation}
We fix $\{p_i\}$ to be arranged in counterclockwise order. For any $\nu\in\S^1$, we denote $\nu^\phi := \frac{\nu}{\phi(\nu)}$.

The \emph{Wulff shape} associated to $\phi$ is defined as $W_\phi := \set{\phic < 1}$, where 
\[ \phic(x) := \sup \{x\cdot \nu: \phi(\nu)\leq 1\} \]
is the dual anisotropy to $\phi$. Moreover, $W_\phi$ is a polygon with facets given by $[p_{i-1}, p_i]$. We define $\cN_\phi := \{\nu_i\}_{i=1}^N$ to be the corresponding set of outer normal vectors to $W_\phi$, such that $\nu_i$ is orthogonal to $[p_{i-1}, p_i]$. Then one can re-express
\[ \phic(x) = \max_{1\leq i\leq N} x\cdot \nu_i^\phi. \]
Throughout the paper, we make the assumption \eqref{eq:phi assumption} that adjacent vectors in $\cN_\phi$ differ in angle by less than 90 degrees.

We recall that $W_\phi$ solves the isoperimetric problem for $P_\phi$, in that any set of finite perimeter $E\sb\R^2$ satisfies the \emph{Wulff inequality} 
\begin{equation}
    \label{eq:wulff ineq}
    P_\phi(E) \geq 2|W_\phi|^{1/2} |E|^{1/2}.
\end{equation}
We also denote the scaled and translated Wulff shape \[ W_\phi(x,r) := x + rW_\phi = \set{y\in\R^N: \phic(y-x) < r}. \]

For a set $E\sb\R^2$, we define the anisotropic \emph{signed distance function} 
    \[ \sdp_E(x) := \begin{cases}
        \inf_{y\in E} \phic(x-y) &x\not\in E\\
        -\inf_{y\not\in E} \phic(y-x) &x\in E
    \end{cases} \]
and set $\dd^\phic_E(x) := |\sdp_E(x)|$. We also denote $d_H(E,F)$ as the Hausdorff distance between sets $E,F\sb\R^2$. When $E,F$ are sets of finite perimeter, $\sdp_E$ and $d_H(E,F)$ are still well-defined by the convention that $E$ is identified with its Lebesgue representative.

We close with some miscellaneous notations:
\begin{itemize}
    \item We define $L_\phi$ to be the smallest constant such that \[ L_\phi^{-1}|\nu| \leq \phi(\nu) \leq L_\phi|\nu| \qquad\forall \nu\in\R^2\setminus\{0\}. \]
    
    \item For an oriented line segment $I\sb\R^2$, we let $\nu_I$ denote its outer normal vector. We also define $s_I$ to be the constant such that $I\sb \{x\cdot\nup_I = s_I\}$.  
    
    \item Given vectors $\nu, \nu' \in\mathbb S^1$, we denote $A[\nu,\nu']\sb\mathbb S^1$ to be the closed arc running counterclockwise from $\nu$ to $\nu'$. We similarly let $A(\nu,\nu')$ denote the open arc.

    \item For a set $E\sb\R^2$ with Lipschitz boundary, we let $\chi(E)$ denote its Euler characteristic.
    
    \item $x\lesssim y$ means $x\leq Cy$ for some universal constant $C$, i.e. a value which does not depend on any of the parameters at hand, and otherwise denote additional dependencies using subscripts. $x\sim y$ means $x\lesssim y$ and $x\gtrsim y$. 

    \item Given a measure $\mu$, we denote $\dashint_X fd\mu := \inv{\mu(X)} \int_X fd\mu$.
\end{itemize}

\section{Crystalline geometry in $\R^2$}
\label{sec:crystalline geometry}

\subsection{Lipschitz $\phi$-regularity}
\label{sec:lipreg}
Here we discuss Lipschitz $\phi$-regularity (first introduced in \cite{Bellettini2001}), a regularity class of sets for which one can define a suitable notion of $\phi$-mean curvature. 

\begin{definition}
    \label{def:lipreg}
    We say that a set $E\sb\R^2$ is \textbf{\emph{Lipschitz $\phi$-regular}} if $E$ is Lipschitz and there exists a Lipschitz vector field $X:\partial E \to\R^2$ such that $X\in\partial\phi(\nu_E)$ $\cH^1$-a.e.
\end{definition}
Note that when $\phi$ is smooth and uniformly elliptic, then \cref{def:lipreg} coincides with $C^{1,1}$ regularity. For crystalline $\phi$ in two dimensions, Lipschitz $\phi$-regularity is equivalent to a local cone condition. 

\begin{definition}
    \label{def:cone condition}
    Given $p_i\in\tN_\phi$, we say that an oriented rectifiable curve $\Gamma\sb\R^2$ satisfies the \textbf{\emph{$p_i$ cone condition}} if $p_i \in \partial \phi(\nu_\Gamma(x))$ for $\cH^1$-a.e. $x\in\Gamma^*$.
\end{definition}
\noindent Note that $p_i\in\partial\phi(\nu_\Gamma)$ is equivalent to $\nu_\Gamma(x) \in A[\nu_i, \nu_{i+1}]$, since we can express the subdifferential
    \begin{equation}
        \label{eq:phi subdiff}
        \partial\phi(\nu) = 
        \begin{cases}
            p_i &\nu\in A(\nu_i,\nu_{i+1})\\
            [p_{i-1},p_i] &\nu = \nu_i.
        \end{cases}
    \end{equation}
Thus we are justified in calling \cref{def:cone condition} a ``cone condition" since it is equivalent to $\Gamma$ satisfying 
\begin{equation}
    \Gamma \sb x_0 + C(\nu_i, \nu_{i+1}) \ \ \forall x_0 \in\Gamma \quad\text{where}\quad C(\nu_i, \nu_{i+1}) := \{x\cdot\nu_i \leq 0 \leq x\cdot \nu_{i+1}\} \cup \{ x\cdot \nu_{i+1}\leq 0 \leq x\cdot \nu_i\}.
\end{equation}

\begin{proposition}
    \label{prop:cone lipreg}
    Let $E\sb\R^2$ be a set of finite perimeter such that $\partial E$ is the finite disjoint union of simple curves. Then the following are equivalent: 
    \begin{itemize}
        \item[(i)] $E$ is Lipschitz $\phi$-regular.
        \item[(ii)] For all $x \in \partial E$, $\partial E \cap B(x,r)$ satisfies the $p_i$ cone condition for some $p_i\in\tN_\phi$ and $r>0$.
    \end{itemize}
\end{proposition}

By the proof of \cref{prop:cone lipreg}, which we postpone to the end of this section, the boundary of a Lipschitz $\phi$-regular set $E\sb\R^2$ can be partitioned into regular and fuzzy edges, defined as follows: 
\begin{itemize}
    \item A \textbf{regular edge} is a maximal line segment $I\sb\partial E$ with normal $\nu_I\in \cN_\phi$, such that for some open neighborhood $U\sp I$, $\partial E \cap U$ lies on one side of the line $\{x\cdot\nup_I = s_I\}$. We say that $I$ is \emph{positive} if $\partial E \cap U \sb \{x\cdot\nup_I \leq s_I\}$ and \emph{negative} if $\partial E \cap U \sb \{x\cdot\nup_I \geq s_I\}$, and let $\sigma_I \in \{\pm1\}$ denote the sign of $I$.
    \item A \textbf{fuzzy edge} is a path $J \sb \partial E$ connecting two regular edges $I_1$ and $I_2$. We assign fuzzy edges to have sign $\sigma_J := 0$. Moreover, we define $J$ to be \textbf{type I} if $\nu_{I_1} \neq \nu_{I_2}$, and \textbf{type II} if $\nu_{I_1} = \nu_{I_2}$. 
    
    Although $J$ need not be a line segment, it satisfies the $p_i$ cone condition for a unique $p_i\in \tN_\phi$. Thus we may loosely consider $J$ as an edge in the sense that $x\in J\mapsto p_i \in \partial \phi(\nu_E)$ is a constant section of the anisotropic normal bundle. For this reason, we denote $p_J := p_i$ as the anisotropic outer normal vector of $J$. 
\end{itemize}
We define an \textbf{edge} $I\sb \partial E$ to be a regular edge or fuzzy edge. 

\begin{definition}
    \label{def:curvature}
    For a Lipschitz $\phi$-regular set $E\sb\R^2$, its \textbf{\emph{$\phi$-mean curvature $\kp_E$}} is defined as the function on $\partial E$ that on each edge $I\sb\partial E$ is equal to the constant
    \begin{equation}
        \label{eq:curvature}
        \kp_E(I) := \frac{\sigma_I L_I}{\cH^1(I)}
    \end{equation}
    where $L_I$ is the length of the facet of $W_\phi$ with outer normal $\nu_I$. We will often abbreviate $\alpha_I := \sigma_I L_I$. 
\end{definition}

Informally, one can justify \eqref{eq:curvature} as the ratio between the first variations of $P_\phi$ and $|E|$ upon translation of the edge $I$. We remark it is not straightforward to define $\kp_E$ for a general set of finite perimeter $E$, since the first variation of $P_\phi$ is not a linear functional on vector field perturbations, prohibiting the usual distributional definition. However, the authors in \cite{Bellettini2001} used an $L^2$ projection framework to define a suitable notion of $\kp_E$ for Lipschitz $\phi$-regular sets, that in the case of $n=2$ coincides with \eqref{eq:curvature}.

\begin{proof}[Proof of \cref{prop:cone lipreg}]
    Suppose there exists a Lipschitz vector field $X:\partial E\to \R^2$ such that $X\in\partial\phi(\nu_E)$ on $\partial^*E$. If $x\in \partial E$ satisfies $X(x) \in (p_{i-1},p_i)$, then by continuity and \eqref{eq:phi subdiff}, one has
    $X \in (p_{i-1},p_i) \implies \nu_E = \nu_i$ along $\partial^* E \cap B(x,r)$ for some $r>0$. Similarly, if $X(x) = p_i$ for some $i$, then $X \in (p_{i-1},p_i]\cup [p_i,p_{i+1})$ on some neighborhood $\partial^* E \cap B(x,r)$, on which it holds that $\nu_E \in A[\nu_i, \nu_{i+1}]$ and thus $p_i \in \partial\phi(\nu_E)$. This proves $(i) \Rightarrow (ii)$.

    Now assume (ii), which surely implies that $E$ has Lipschitz boundary. We wish to find a Lipschitz vector field $X$ on $\partial E$ such that $X\in\partial\phi(\nu_E)$. For each $p_i \in\tN_\phi$, we define the set 
        \[ U_i := \{x\in\partial E: \ \partial^* E\cap B(x,r)\text{ satisfies the }p_i\text{ cone condition for some }r>0 \}. \]
    Note that $U_i$ is open in $\partial E$ and $\partial E \sb \cup_{\nu_i\in\cN_\phi} U_i$. By compactness, there exists $r_0>0$ so that for each $x\in \partial E$, $\partial E\cap B(x,r_0) \sb U_i$ for some $i$.

    We observe that $U_{i-1} \cap U_i$ is the union of line segments in $\partial E$ with outer normal $\nu_i$. Moreover, any maximal line segment $(x,y)\sb\partial E$ with normal $\nu_i$ is maximal in $U_{i-1} \cap U_i$, so each endpoint $x,y$ lies in exactly one of $U_{i-1}$ and $U_i$. Note that $[x,y]$ is a regular edge if and only if $x,y$ belong to opposite sets. It follows that any regular edge $I\sb\partial E$ satisfies $|I| \geq 2r_0$; indeed, if $|I| < 2r_0$, then the endpoints of $I$ belong to a common $U_i$. In particular, $\partial E$ has finitely many regular edges.

    We now fix a connected component $\Gamma \sb \partial E$ and exhaustively list its regular edges as $I_1, \dots, I_m$ where $I_j := [x_j, y_j]$, in clockwise order. Let $J_j\sb\Gamma$ be the closed segment (possibly a point) from $y_j$ to $x_{j+1}$, so that we have partitioned $\Gamma$ into interlacing segments $I_1 \cup J_1 \cup \cdots I_m \cup J_m$.

    We claim that each $J_j$ is contained in some $U_{i_j}$. Without loss of generality, fix $J_1$ and let $y_1\in U_1$. Let $\gamma:[0,1]\to J_1$ be a parametrization such that $\gamma(0) = y_1$ and consider \[ t_0 := \sup\{t\in [0,1]: \gamma(s) \in U_1 \text{ for all }s\leq t\}. \]
    We wish to show that $t_0=1$. Suppose $t_0<1$. Then $\gamma(t_0) \not\in U_1$ by the openness of $U_1$, so $\gamma(t_0)\in U_i$ for $i\in\{0,2\}$. For small enough $\eps>0$, $\gamma((t_0-\eps,t_0))$ is a line segment in $U_1\cap U_i$. Letting $I'$ be the maximal line segment extending $\gamma((t_0-\eps,t_0))$, we find that the terminal endpoint of $I'$ is $\gamma(t_0)\in U_i$, while the initial endpoint is in $U_1$. Thus $I'$ is a regular edge, contradicting that the list $I_1, \dots, I_m$ was exhaustive. 

    Therefore, there exists a consecutively adjacent sequence $p_{i_1},\dots, p_{i_m}\in\tN_\phi$ such that $J_j \sb U_{i_j}$ and $\nu_E(I_j) = \nu_{\max\{i_{j-1},i_j\}}$. We can thus define $X:\Gamma \to\R^2$ by setting $X(x) = p_{i_j}$ along $J_j$ and linearly interpolating along $I_j$. Proceeding as above for every component $\Gamma \sb\partial E$, we conclude the proof.
\end{proof}

\subsection{$\phi$-minimal paths}

In subsequent sections of the paper, we will frequently use barriers to perturb sets of finite perimeter without increasing $P_\phi$. Such barriers will require that their (oriented) boundaries are minimal paths with respect to $P_\phi$, motivating the following definition.

\begin{definition}
    \label{def:minimal path}
    We call an oriented rectifiable curve $\Gamma_0\sb\R^2$ a \textbf{\emph{$\phi$-minimal path}} if $P_\phi(\Gamma_0) \leq P_\phi(\Gamma)$ for all oriented rectifiable curves $\Gamma$ with the same (oriented) endpoints as $\Gamma_0$.
\end{definition}

It turns out that $\phi$-minimality is characterized by the $p_i$ cone condition.

\begin{lemma}
    \label{lem:minimal path}
    An oriented rectifiable curve $\Gamma\sb\R^2$ is $\phi$-minimal if and only if it satisfies the $p_i$ cone condition for some $p_i\in\tN_\phi$.
\end{lemma}

\begin{remark}
    Prior works have shown that tangent cones to the Wulff shape $W_\phi$ are locally $\phi$-minimal; it is proven in \cite{Figalli2011} using a very short calibration argument, as well as in \cite{Taylor1978} for all dimensions using the Wulff inequality. Here we present the calibration argument, as it clearly highlights the equality condition.
\end{remark}

\begin{proof} 
    Suppose $\Gamma_0$ is an oriented rectifiable curve satisfying the $p_i$ cone condition. Let $\Gamma$ be any oriented rectifiable curve with the same oriented endpoints as $\Gamma_0$. Treating $\Gamma_0$ and $\Gamma$ as integral 1-currents, it follows that $[\Gamma] - [\Gamma_0] = \partial T$ for some integral 2-current $T$. We recall that $\phi(\nu) = p_i\cdot\nu$ for $\nu\in A[\nu_i, \nu_{i+1}]$ and $\phi(\nu) < p_i\cdot\nu$ for $\nu\in \mathbb{S}^1\setminus A[\nu_i, \nu_{i+1}]$. Applying the divergence theorem to the constant vector field $p_i$, we obtain 
    \begin{align}
        0 = T(\div p_i) = (\partial T)(p_i) &= \int_{\Gamma^*} p_i\cdot\nu_\Gamma\,d\cH^1 - \int_{\Gamma_0^*} p_i\cdot\nu_{\Gamma_0}  \,d\cH^1 \notag \\
        &= \int_{\Gamma^*} p_i\cdot\nu_{\Gamma} \,d\cH^1 - P_\phi(\Gamma_0) \notag\\
        &\leq P_\phi(\Gamma) - P_\phi(\Gamma_0). \label{eq:calibration}
    \end{align}
    Since $\Gamma$ is arbitrary, $\Gamma_0$ is $\phi$-minimal, showing the backward direction.
    
    For the forward direction, if $\Gamma$ is $\phi$-minimal, then there exists a curve $\Gamma_0$ with the same endpoints as $\Gamma$ which satisfies the $p_i$ cone condition. Indeed, if $x_1$ and $x_2$ are the initial and terminal endpoints of $\Gamma$, then $x_2-x_1$ belongs to the cone $\{x\cdot\nu_{i+1} \leq 0 \leq x\cdot \nu_i\}$ for some $i$, so there exists a polygonal path from $x_1$ to $x_2$ consisting of segments with outer normals $\nu_i, \nu_{i+1}$. Thus, we may apply the previous divergence theorem argument, and since $P_\phi(\Gamma) = P_\phi(\Gamma_0)$, we are done by observing that equality holds in \eqref{eq:calibration} iff $\phi(\nu_\Gamma) = \nu_\Gamma \cdot p_i$ for $\cH^1$-a.e. $x\in \Gamma^*$. 
\end{proof}

\begin{corollary}[Perturbation by $\phi$-minimal barrier]
    \label{cor:phi min perturbation}
    Let $E\sb\R^2$ be a set of finite perimeter, $U\sb\R^2$ a bounded open set, and $\Sigma\sb U$ such that $\partial \Sigma \cap U$ is a $\phi$-minimal path (when oriented by $\nu_\Sigma$).
    \begin{enumerate}
        \item If $(E \setminus \Sigma) \cap U \Subset U$, then $P_\phi(E \cap \Sigma; U) \leq P_\phi(E; U)$. 
        \item If $\Sigma\setminus E \Subset U$, then $P_\phi(E \cup \Sigma; U) \leq P_\phi(E; U)$.
    \end{enumerate}
\end{corollary}

\begin{proof}
    We prove only the first statement, as a symmetric argument shows the second statement. By \cref{lem:minimal path}, $\partial \Sigma\cap U$ satisfies the $p_i$ cone condition for some $p_i\in \tN_\phi$, so $\phi(\nu_\Sigma) = p_i\cdot\nu_\Sigma$ $\cH^1$-a.e. along $\partial^*\Sigma \cap U$. Applying the divergence theorem with the constant vector field $p_i$ and invoking \eqref{eq:setminus}, we obtain
    \begin{align*}
        0 = \int_{(E\setminus \Sigma)\cap U} \div p_i &= \int_{\partial^*(E\setminus \Sigma) \cap U} p_i \cdot \nu_{E\setminus \Sigma} \,d\cH^1\\
        &= \int_{\partial^*E \cap \Sigma^{(0)} \cap U} p_i\cdot\nu_E \,d\cH^1 - \int_{\partial^*\Sigma \cap (E^{(1)} \cup \{\nu_\Sigma = \nu_E\}) \cap U} \phi(\nu_\Sigma) \,d\cH^1\\
        &\leq P_\phi(E; \Sigma^{(0)}\cap U) - P_\phi(\Sigma; E^{(1)}\cap U).
    \end{align*}
    By \eqref{eq:intersection}, the claim follows: 
    \begin{align*}
        P_\phi(E\cap \Sigma; U) - P_\phi(E; U) &= P_\phi(E; \Sigma^{(1)}\cap U)  + P_\phi(\Sigma; E^{(1)}\cap U) + P_\phi(\{\nu_E = \nu_\Sigma\}\cap U) - P_\phi(E; U)\\
        &\leq P_\phi(\Sigma;E^{(1)} \cap U) - P_\phi(E; \Sigma^{(0)}\cap U)\\
        &\leq 0. \qedhere
    \end{align*}
\end{proof}

By \cref{lem:minimal path}, $\phi$-minimal paths between fixed endpoints are highly non-unique unless they are a straight line segment with outer normal in $\cN_\phi$. In this case, straight line geodesics satisfy a sharp minimality estimate:

\begin{proposition}
    \label{prop:sharp minimality}
    Let $\Gamma\sb\R^2$ be a simple rectifiable oriented curve from $x_1$ to $x_2$, such that $[x_1,x_2]$ has outer normal $\nu_i\in\cN_\phi$. Then \begin{equation}
        \label{eq:sharp minimality}
        P_\phi(\Gamma) - P_\phi([x_1,x_2]) \geq |p_i-p_{i-1}|h \qquad\text{where}\quad h :=\sup_{z\in \Gamma} |\nu_i\cdot(z-x_1)|.
    \end{equation} 
\end{proposition}

\begin{remark}
    The linear dependence of the righthand side on $h$ is special to crystalline $\phi$, as for smooth elliptic $\phi$ such an estimate holds only with the righthand side replaced by $\min\{h,h^2\}$. This is explained heuristically by the fact that $W_\phi$ separates linearly, rather than quadratically, from its tangent planes. We further remark that the estimate \eqref{eq:sharp minimality} is reminiscent of the linear sharp stability estimate in \cite{Figalli2022} and can be thought of as a local version of this estimate in two dimensions. 
\end{remark}

\begin{proof}
    Without loss of generality, let $\nu_{[x_1,x_2]} = \nu_1 = (0,1)$ and $h>0$. Letting $z$ attain $h$, we may also assume $z \in \{(x-x_1)\cdot\nu_1 < 0\}$, for a symmetric argument holds in the alternative. By translation, we may express $z = (0,0)$, $x_1 = (a,h)$, and $x_2 = (b,h)$. By \cref{lem:minimal path}, it suffices to show \eqref{eq:sharp minimality} for $\Gamma = [x_1,z]\cup[z,x_2]$.
    
    Fix $\eps\in(0,h)$, denote the paths $\Gamma_1^\eps := [x_1,\eps x_1]$ and $\Gamma_2^\eps := [\eps x_2, x_2]$, and let $U^\eps$ be the trapezoidal region with vertices $x_1,x_2,\eps x_1,\eps x_2$. We may consider a vector field $X\in C^1_c(\R^2)$ such that $X(x) = p_0$ for $x\in \Gamma_1^\eps$, $X(x) = p_1$ for $x\in\Gamma_2^\eps$, and $X$ is given on $U^\eps$ by linear interpolation along horizontal lines. Explicitly, for $(x,y)\in U^\eps$, we can express 
    \begin{equation*}
        X(x,y) = p_0 + s(p_1 - p_0) \qquad\text{where}\qquad s = \inv{b-a}\Big(\frac{hx}{y} - a \Big).
    \end{equation*}
    It is straightforward to compute that on $U^\eps$,
    \[ 
        \div X = -|p_1-p_0|\partial_x s = -\frac{|p_1-p_0|}{b-a}\cdot\frac{h}{y}
    \]
    and thus by the divergence theorem,
    \begin{equation}
        \label{eq:sharp minimality div}
        \int_{[x_1,x_2]} X\cdot\nu_1 \,d\cH^1 - \int_{\Gamma_1^\eps \cup [\eps x_1,\eps x_2] \cup \Gamma_2^\eps} X\cdot \nu \,d\cH^1 = \int_{U^\eps} \div X\, dx dy = -(1-\eps)|p_1-p_0|h.
    \end{equation}
    Note that $X\in[p_0,p_1]$ along $\partial U^\eps$ by construction. Thus $X\cdot\nu \leq \phic(X)\phi(\nu) = \phi(\nu)$ for all $\nu\in\S^1$, and $X\cdot\nu_1 = \phi(\nu_1)$. By \eqref{eq:sharp minimality div}, we obtain the bound
    \begin{align*}
        P_\phi(\Gamma) &\geq \int_{\Gamma_1^\eps \cup \Gamma_2^\eps} \phi(\nu) d\cH^1 \geq \int_{\Gamma_1^\eps \cup \Gamma_2^\eps} X\cdot\nu\,d\cH^1 \\
        &= (1-\eps)(P_\phi([x_1,x_2]) + |p_1-p_0|h).
    \end{align*}
    Since $\eps$ is arbitrary, we obtain \eqref{eq:sharp minimality}.
\end{proof}

\subsection{Regularity of almost-minimizers}
Here we prove that almost-minimizers of $P_\phi$ are Lipschitz $\phi$-regular. 
\begin{definition}
    \label{def:almost min}
    Given constants $\Lambda,r_0>0$ such that $\Lambda r_0 \leq 1$, we say a set of finite perimeter $E\sb\R^2$ is a \textbf{\emph{$(\Lambda, r_0)$-minimizer}} of $P_\phi$ (or \textbf{\emph{almost-minimizer}}) if whenever $G$ is such that $E\Delta G \Subset B(x,r_0)$ for some $x\in\partial E$, 
    \begin{equation*}
        P_\phi(E) \leq P_\phi(G) + \Lambda|E\Delta G|.
    \end{equation*}
\end{definition}

\begin{theorem}
    \label{thm:lipreg}
    Let $E\sb\R^2$ be a $(\Lambda,r_0)$-minimizer of $P_\phi$. Then $E$ is Lipschitz $\phi$-regular. In particular, there exists $c = c(\phi)>0$ such that $\partial E \cap B(x,cr_0)$ is a $\phi$-minimal path for all $x\in \partial E$. 
\end{theorem}

First we establish some basic topological structure for almost-minimizers, relying on a structure theorem from \cite{Ambrosio2001} for sets of finite perimeter in two dimensions.

\begin{lemma}
    \label{lem:finite curve}
    Let $E\sb\R^2$ be a $(\Lambda,r_0)$-minimizer of $P_\phi$. Then $\partial E$ is the finite union of oriented rectifiable simple closed curves $\Gamma_j$ such that for each $j$, $\nu_E = \nu_{\Gamma_j}$ holds $\cH^1$-a.e. along $\Gamma_j$.
\end{lemma}
\begin{remark}
    We refer to any such $\Gamma_j$ as a \textbf{boundary component} of $E$, which may be distinct from a connected component of $\partial E$ in the topological sense. The decomposition of $\partial E$ into boundary components may not be unique, but this will not affect future arguments.
\end{remark}
\begin{proof}
    By \cite{Ambrosio2001}, there exist countably many oriented rectifiable simple closed curves $\Gamma_j$ such that
    \begin{align*}
        \partial^*E &= \bigcup_{j=1}^\infty \Gamma_j \quad \text{up to a $\cH^1$-null set}\\
        \nu_E &= \nu_{\Gamma_j} \qquad \cH^1\text{-a.e. along }\Gamma_j.
    \end{align*}
    We need only show there are finitely many $\Gamma_j$.
    
    Consider such a curve satisfying $P_\phi(\Gamma_j) < cr_0$, where $c=c(\phi)>0$ is sufficiently small so that $\Gamma_j$ must be contained in some ball $B(x,r_0)$. There exists some bounded set of finite perimeter $E_j$ and $\sigma_j\in\{\pm1\}$ such that $\Gamma_j$ is $\cH^1$-equivalent to $\partial^*E_j$ and $\nu_{\Gamma_j} = \sigma_j \nu_{E_j}$. First we assume $\sigma_j=+1$ and consider the modified set $\tilde{E} := E\setminus E_j$. By \eqref{eq:setminus}, 
    \begin{align}
        \label{eq:finite comp}
        P_\phi(\tilde{E}) &= P_\phi(E;E_j^{(0)}) + P_\phi(E_j^c;E^{(1)}) + P_\phi(E;\{\nu_{E} = -\nu_{E_j}\}) \notag \\
        &= P_\phi(E;E_j^{(0)})\notag \\
        &\leq P_\phi(E) - P_\phi(E_j).
    \end{align} 
    Because $E$ is a $(\Lambda,r_0)$-minimizer, we deduce \begin{equation}
        \label{eq:finite comp 2}
        P_\phi(E_j) \leq P_\phi(E) - P_\phi(\tilde{E}) \leq \Lambda|E\setminus \tilde{E}| \leq \Lambda |E_j|.
    \end{equation}
    By the Wulff inequality, we obtain $2|W_\phi|^{1/2} |E_j|^{1/2} \leq P_\phi(E_j) \leq \Lambda |E_j|$ and hence $|E_j|\gtrsim_{\phi} \Lambda^{-2}$, implying $P_\phi(E_j) \gtrsim_{\phi}\Lambda^{-1} \geq r_0$. In the case that $\sigma_j = -1$, one can obtain the same inequality by considering $\tilde{E} := E\cup E_j$ and applying the Wulff inequality for the reflected integrand $\tilde{\phi}(\nu) = \phi(-\nu)$, since then $P_\phi(E_j^c) = P_{\tilde{\phi}}(E_j)$. It follows that every $\Gamma_j$ satisfies $P_\phi(\Gamma_j) \geq cr_0$ for some sufficiently small $c$, and thus there are finitely many $\Gamma_j$.
\end{proof}

Next we invoke the sharp minimality result from \cref{prop:sharp minimality} to prohibit the boundaries of almost-minimizers from locally oscillating in the directions $\cN_\phi$.
\begin{proposition}
    \label{prop:local straight}
    There exists $c_0 = c_0(\phi)>0$ such that the following holds: Let $E$ be a $(\Lambda,r_0)$-minimizer of $P_\phi$ and $\Gamma \sb \partial E$ an oriented path from $x$ to $y$ such that $\nu_\Gamma = \nu_E$ a.e. along $\Gamma$. If $P_\phi(\Gamma) \leq c_0r_0$ and $[x,y]$ has outer normal in $\cN_\phi$, then $\Gamma = [x,y]$. 
\end{proposition}
\begin{proof}
    Let $P_\phi(\Gamma)\leq c_0r_0$ and suppose on the contrary that $\Gamma\neq [x,y]$. Let $\cL$ be the line through $x,y$. By restricting to a subsegment of $\Gamma$ if necessary, we may assume that $\Gamma$ lies on one side of $\cL$. 

    Let $U\sb\R^2$ be the region bounded between $\Gamma$ and $\cL$, and assume without loss of generality that $\nu_U = \nu_E$ a.e. along $\Gamma$; a similar argument works if $\nu_U = -\nu_E$. We consider the energy competitor $\tilde{E} = E\setminus U$. Applying \eqref{eq:setminus} and \cref{prop:sharp minimality}, we obtain
    \begin{align*}
        P_\phi(E) - P_\phi(\tilde{E}) &= P_\phi(E;\{\nu_E = \nu_U\}) - P_\phi(U^c;E^{(1)}) + P_\phi(E;U^{(1)}) \\
        &\geq P_\phi(\Gamma) -P_\phi(U^c;\cL) \\
        &= P_\phi(\Gamma) - P_\phi([x,y]) \gtrsim_\phi h.
    \end{align*}
    For $c_0$ sufficiently small, $U\Subset B(x,r_0)$, and thus almost-minimality of $E$ implies \begin{align*}
        h \lesssim_\phi P_\phi(E) - P_\phi(\tilde{E}) \leq \Lambda|U| \leq \Lambda \cH^1(\Gamma) h.
    \end{align*}  
    It follows that $P_\phi(\Gamma) \gtrsim_\phi \cH^1(\Gamma) \gtrsim_\phi \Lambda^{-1} \geq r_0$. Thus, we obtain a contradiction for $c_0$ sufficiently small. 
\end{proof}

\begin{corollary}
    \label{cor:locally minimal components}
    Let $c_0$ be as in \cref{prop:local straight}, suppose $E\sb\R^2$ is a $(\Lambda,r_0)$-minimizer of $P_\phi$, and let $\Gamma\sb\partial E$ be an oriented path such that $\nu_\Gamma = \nu_E$ a.e. along $\Gamma$. If $P_\phi(\Gamma) \leq c_0r_0$, then $\Gamma$ is $\phi$-minimal. 
\end{corollary}

\begin{proof}
    Let $\gamma:[0,\ell]\to\Gamma$ be an arclength parametrization and label $x_0 = \gamma(0)$ and $x_1 = \gamma(\ell)$. If $\nu := \nu_{[x_0,x_1]}$ belongs to $\cN_\phi$, then we are done by \cref{prop:local straight}. Otherwise, $\nu\in A(\nu_i,\nu_{i+1})$ for some $\nu_i,\nu_{i+1}\in\cN_\phi$, and we may express $\gamma(t) = x_0 + a(t)R^{-1}\nu_i + b(t)R^{-1}\nu_{i+1}$ where $a(\ell), b(\ell) > 0$. 

    We claim the bounds $0\leq a(t)\leq a(\ell)$ and $0\leq b(t)\leq b(\ell)$. Indeed, if say $a(t) < 0$ for some $t$, then there exists some other point $t'\in(t,\ell)$ such that $a(t')=0$. However, this would contradict \cref{prop:local straight} since $[\gamma(0),\gamma(t')]$ is orthogonal to $\nu_{i+1}$ and $\cH^1(\gamma([0,t'])) \leq c_0 r_0$, but $\gamma([0,t'])$ is not a line segment. Similarly, if there are any points $s<t$ such that $a(s) > a(t)$, then one may reach a contradiction to \cref{prop:local straight} by finding $t' \in [0,s)$ such that $a(t') = a(t)$. Thus we deduce that $a,b$ are non-decreasing. 

    It follows that for a.e. $t\in[0,\ell]$, we have $\gamma'(t) = a'(t) R^{-1} \nu_i + b'(t) R^{-1} \nu_{i+1}$ lies in the non-negative span of $R^{-1} \nu_i, R^{-1} \nu_{i+1}$, and thus $\nu_E(\gamma(t)) = R\gamma'(t) \in A[\nu_i, \nu_{i+1}]$.
\end{proof}

\medskip

\noindent \textbf{Proof of \cref{thm:lipreg}:}

\medskip 

By \cref{prop:cone lipreg} and \cref{cor:locally minimal components}, each boundary component of $E$ is Lipschitz $\phi$-regular. It suffices to show that the boundary components of $E$ are pairwise disjoint, which follows from the following \cref{prop:no intersection}. \hfill $\square$

\begin{proposition}
    \label{prop:no intersection}
    Let $E\sb\R^2$ be $(\Lambda,r_0)$-minimizer of $P_\phi$. Then, the boundary components of $E$ are pairwise disjoint, and for any two distinct boundary components $\Gamma_1, \Gamma_2 \sb\partial E$, 
    \begin{equation}
        \label{eq:component dist}
        d(\Gamma_1, \Gamma_2) \geq c(\phi)r_0.
    \end{equation}
\end{proposition}

\begin{remark}
    The proof of \cref{prop:no intersection} is built on the simple observation that a quadruple junction is energetically expensive, and thus two boundary components of $E$ cannot get close to each other on a scale much smaller than $r_0$ without violating almost-minimality. Nonetheless, the proof is a bit lengthy, as one must first rule out the possibility that the intersection between boundary components can have cluster points, before one can construct a proper energy competitor.

    This is also the only place in the proof of \cref{thm:lipreg} where we need the assumption that neighboring vectors in $\cN_\phi$ differ in angle by less than 90 degrees. One can see that some restriction on $\phi$ is necessary to avoid the intersection of boundary components, by considering the example of two squares touching at a corner when $\phi=|\cdot|_\infty$ (see \cite[Remark 1]{ChambolleNovaga2015planar}), which satisfies $(0,r_0)$-minimality for small enough $r_0$.
\end{remark}

\begin{proof}
In what follows, $c_0$ is the constant from \cref{prop:local straight}. For a boundary component $\Gamma \sb \partial E$ and for any two points $x,y\in \Gamma$, we define \[ d^\phi_\Gamma(x,y) := \min\{P_\phi(\gamma_1),P_\phi(\gamma_2)\} \] where $\gamma_1,\gamma_2$ are the oriented subsegments of $\Gamma$ joining $x$ and $y$.

\textbf{Step 1.} We claim there exists $c_1 = c_1(\phi) >0$ such that for any boundary component $\Gamma$ of $E$ and $x,y\in\Gamma$,
    \begin{equation}
        |x-y| < c_1r_0 \quad\implies\quad d^\phi_{\Gamma}(x,y) < c_0r_0.
    \end{equation}
    It is equivalent to show \begin{equation}
        \label{eq:d}
        d := \inf\{|x-y|: d^\phi_{\Gamma}(x,y) \geq c_0r_0 \} \geq c_1 r_0.
    \end{equation}
    We proceed by assuming $d < c_1r_0$ and obtaining a contradiction if $c_1$ is sufficiently small. 
    
    Suppose $x_1,x_2\in\Gamma$ are such that $|x_1-x_2|=d$ and $d^\phi_\Gamma(x_1,x_2) \geq c_0r_0$. Let $\tGamma_1 \sb \Gamma$ be a segment satisfying $P_\phi(\tGamma_1) = c_0r_0$ and such that $\tGamma_1$ is bisected at $x_1$ (with respect to $P_\phi$). Similarly define $\tGamma_2$ which is bisected at $x_2$. By \cref{cor:locally minimal components}, $\tGamma_1$ and $\tGamma_2$ satisfy the $p_i$ and $p_j$ cone conditions respectively for some $p_i, p_j\in \tN_\phi$, implying $\tGamma_1 \sb x_1 + C_i$ and $\tGamma_2 \sb x_2 + C_j$ where 
    \[ 
        C_i := \{x\cdot \nu_i \leq 0 \leq x\cdot\nu_{i+1}\} \cup \{x\cdot\nu_{i+1} \leq 0\leq x\cdot \nu_i\}
    \]
    and $C_j$ is defined similarly with $\nu_j, \nu_{j+1}$. Note that if the cones $C_i$ and $C_j$ intersect only at the origin, then $\tGamma_1$ and $\tGamma_2$ must intersect if $c_1$ is sufficiently small, a contradiction. Thus, it must be the case that $A[\nu_j, \nu_{j+1}]$ nontrivially intersects $A[\nu_i,\nu_{i+1}]\cup A[-\nu_i,-\nu_{i+1}]$. Without loss of generality, we let $\nu_i = \nu_1$ and assume $\nu_j \in A[\nu_1, \nu_2] \cup A[-\nu_1, -\nu_2]$. 

    If $\nu_j \in A[\nu_1,\nu_2]$, then $\nu_j\in \{\nu_1, \nu_2\}$, so $\tGamma_2$ satisfies the $p_1$ or $p_2$ cone condition. In either case, both $\tGamma_1$ and $\tGamma_2$ can be viewed as graphs with vertical axis $\nu_2$, with overlapping domains for $c_1$ sufficiently small, and with upward orientation with respect to $\nu_2$. Then there must exist a distinct segment of $\Gamma$ in between $\tGamma_1$ and $\tGamma_2$, as the absence of such a segment forces $\tGamma_1$ and $\tGamma_2$ to have opposite orientations by a connected argument. However, this contradicts that $|x_1-x_2|=d$. 
    
    If $\nu_j \in A[-\nu_1,-\nu_2]$, then $\tGamma_1$ and $\tGamma_2$ are still graphs with respectively upward and downward orientation with respect to $\nu_2$. Without loss of generality, assume $\tGamma_1$ lies above $\tGamma_2$ with respect to $\nu_2$, as a symmetric argument works on the contrary. Note that $\nu_3 \in A(\nu_2, \nu_j)$ since the angle between $\nu_2$ and $-\nu_1$ is obtuse. Thus, provided $c_1$ is sufficiently small, then for all small enough $r\lesssim_\phi r_0$, the line $L_r := \{(x-x_1)\cdot \nu_3 = -r\}$ intersects both $\tGamma_1$ and $\tGamma_2$. We can then consider the open region $U$ bounded by $\tGamma_1$, $\tGamma_2$, the segment $[x_1,x_2]$, and $L_r$, such that $\nu_U = \nu_E$ along $(\tGamma_1 \cup \tGamma_2) \cap \partial^*U$ and $\nu_U = -\nu_3$ along $L_r \cap \partial^*U$; see \cref{fig:wedge competitor}.
    
    We propose the energy competitor $E \setminus U$. Let us assign $L_r$ to be oriented by $\nu_3$, in which case $L_r\cap \partial U$ is a $\phi$-minimal line segment. Thus \cref{prop:sharp minimality} implies
    \begin{equation}
        \label{eq:wedge gap}
        P_\phi(U; \tGamma_1 \cup \tGamma_2) + P_\phi([x_2,x_1]) \geq P_\phi(U^c; L_r) + |p_3-p_2|r.
    \end{equation}
    Invoking \eqref{eq:setminus} and applying \eqref{eq:wedge gap}, we deduce the bound
    \begin{align}
        \label{eq:wedge gap 2}
        P_\phi(E) - P_\phi(E\setminus U) &\geq  P_\phi(\{\nu_E = \nu_U\}) - P_\phi(U^c;E^{(1)}) \notag \notag \\
        &\geq P_\phi(U; \tGamma_1\cup\tGamma_2) - P_\phi([x_1,x_2]) - P_\phi(U^c; L_r) \notag \\
        &\geq |p_3-p_2|r - P_\phi([x_1,x_2]) - P_\phi([x_2,x_1]) \notag \\
        &\geq |p_3-p_2|r - 2L_\phi d.
    \end{align}
    Setting $r = \frac{2L_\phi + 1}{|p_3-p_2|}d$, it follows that $P_\phi(E) - P_\phi(E\setminus U) \geq d$ by \eqref{eq:wedge gap 2}. Moreover, $|U| \lesssim_\phi d^2$ since $U$ is confined to a quadrilateral region with sidelengths $O(d)$ by the cone conditions satisfied by $\tGamma_j$.
    For $c_1$ sufficiently small, we have $U \Subset B(x_1,r_0)$, and thus the $(\Lambda,r_0)$-minimality of $E$ yields the estimate
    \begin{equation}
        \label{eq:wedge almost minimality}
        d \leq P_\phi(E) - P_\phi(E\setminus U) \leq \Lambda |U| \lesssim_\phi \Lambda d^2.
    \end{equation}
    That is, $d \gtrsim_\phi \Lambda^{-1} \geq r_0$. Finally we reach a contradiction to $d < c_1r_0$ for $c_1$ sufficiently small. 
    
    \begin{figure}
        \centering
        \includegraphics[width=0.4\linewidth]{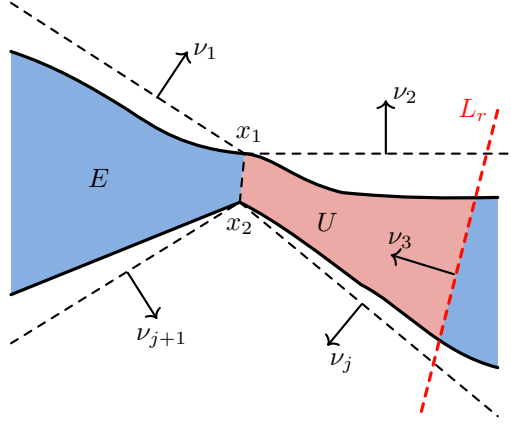}
        \caption{Energy competitor described in step 1 of \cref{prop:no intersection}.}
        \label{fig:wedge competitor}
    \end{figure}

\medskip

\textbf{Step 2.} We claim there exists $c_2 = c_2(\phi)$ such that if $\Gamma_1,\Gamma_2\sb\partial E$ are distinct boundary components, then $|x-y| \geq c_2r_0$ for all distinct $x,y \in \Gamma_1\cap\Gamma_2$. 

    We consider such $x,y$ satisfying $|x-y| < c_2r_0$ and reach a contradiction if $c_2$ is sufficiently small. For $j=1,2$, let $\tGamma_j\sb\Gamma_j$ be the paths joining $x$ and $y$ such that $d^\phi_{\Gamma_j}(x,y) = P_\phi(\tGamma_j)$. Provided $c_2\leq c_1$, step 1 implies $P_\phi(\tGamma_j) \leq c_0r_0$, and thus $\tGamma_j$ is $\phi$-minimal by \cref{cor:locally minimal components}. By restricting $\tGamma_1$ to a shorter curve in $\Gamma_1 \setminus \Gamma_2$ if necessary, we may assume that $\tGamma_1$ and $\tGamma_2 $ intersect only at $x$ and $y$.
    
    Let $U$ be the region bounded between $\tGamma_1$ and $\tGamma_2$. Note there are $\sigma_1,\sigma_2\in\{\pm1\}$ such that $\nu_U = \sigma_j \nu_E$ a.e. along $\tGamma_j$. We assume without loss of generality that $P_\phi(\Gamma_1) \geq P_\phi(\Gamma_2)$ and that $\sigma_1 = 1$; a symmetric argument holds if $\sigma_1 = -1$. 
    
    We observe that $P_\phi(U) \leq C P_\phi(\tGamma_1)$ where $C = \max\{2, 1 + L_\phi^2\}$. Indeed, if $\sigma_2 = 1$, then $P_\phi(U) \leq 2P_\phi(\tGamma_1)$, and if $\sigma_2 = -1$, then we consider the reflected anisotropy $\psi(\nu) := \phi(-\nu)$ to bound $P_\phi(U) = P_\phi(\tGamma_1) + P_\psi(\tGamma_2) \leq (1 + L_\phi^2)P_\phi(\tGamma_1)$. For $c_2$ sufficiently small, $U \Subset B(x,r_0)$, so by \eqref{eq:setminus} and almost-minimality,
        \begin{equation}
            \label{eq:no intersection step 2}
            \inv{C} P_\phi(U) \leq P_\phi(\tGamma_1) \leq P_\phi(\{\nu_E = \nu_U\}) \leq P_\phi(E) - P_\phi(E\setminus U) \leq \Lambda |U|.
        \end{equation}
    Combining \eqref{eq:no intersection step 2} with the Wulff inequality, we obtain $P_\phi(U) \gtrsim_\phi \Lambda^{-1} \geq r_0$ and thus $P_\phi(\tGamma_1) \geq \inv{C} P_\phi(U) \gtrsim_\phi r_0$. Finally, by $\phi$-minimality, $P_\phi(\tGamma_1) \leq P_\phi([x,y]) \leq L_\phi|x-y|$, so we have $|x-y| \gtrsim_\phi r_0$ and obtain a contradiction if $c_2$ is sufficiently small. 

\medskip

\textbf{Step 3.} Now we can show that the boundary components of $E$ are pairwise disjoint.

    Suppose for sake of contradiction that distinct boundary components $\Gamma_1$ and $\Gamma_2$ intersect at a point $x_0$, and let $\Gamma_1, \dots, \Gamma_m\sb\partial E$ be every boundary component of $E$ which contains $x_0$. Without loss of generality, let $x_0=0$. Let $0 < r \leq \min\{c_1/2,c_2\}r_0$. By step 1 and \cref{cor:locally minimal components}, for each $1\leq j\leq m$, the restriction $\tGamma_j := \Gamma_j \cap B(0,r)$ is a $\phi$-minimal path. Moreover, by step 2, distinct $\tGamma_j$ intersect only at the origin. 
    
    Each $\tGamma_j$ is split at the origin into two halves $\Gamma_j^-$ and $\Gamma_j^+$. We now take the set $\{\Gamma_1^\pm, \dots, \Gamma_m^\pm\}$ and list its elements as $I_1, \dots, I_{2m}$ in counterclockwise order around $0$. We observe that $I_j$ and $I_{j+1}$ must have ``opposite" orientations, in that $0$ is an initial point of $I_j$ if and only if $0$ is a terminal point of $I_{j+1}$. Indeed, if say $0$ were the initial point of both $I_j$ and $I_{j+1}$, then a connectedness argument would imply the existence of another path $\tilde{I} \sb \partial E$ in between $I_j$ and $I_{j+1}$ which contains $0$, contradicting that $\Gamma_1, \dots, \Gamma_m$ exhaust all boundary components passing through $0$.

    Thus, for all $j$, \cref{cor:locally minimal components} implies that $I_j\cup I_{j+1}$ is a $\phi$-minimal path, so satisfies the $p_{i_j}$ cone condition for some $p_{i_j} \in \tN_\phi$. If we define $U_j$ as the open region obtained by spanning $B(0,r)$ from $I_j$ to $I_{j+1}$ in counterclockwise order, then 
    \[ U_j \spq S_j := \{x\in B(0,r): x\cdot \nu_{i_j} \geq 0, x\cdot \nu_{i_j+1} \geq 0 \}. \]
    The regions $S_j$ must be pairwise disjoint, and there are at least $2m\geq 4$ such regions. However, this is a contradiction, since each $S_j$ spans an angle of strictly larger than 90 degrees. 

\medskip

\textbf{Step 4.} Letting $\Gamma_1, \dots, \Gamma_m$ be the boundary components of $E$, we know by step 3 that \[ d := \min_{i\neq j} d(\Gamma_i, \Gamma_j) > 0. \]

    Letting $x_1\in\Gamma_1$ and $x_2\in\Gamma_2$ attain $d = |x_1-x_2|$, and assuming $d \leq cr_0$ for $c$ sufficiently small, one obtains a contradiction using the same argument as in step 1. This proves \eqref{eq:component dist}. \qedhere 
\end{proof}

\subsection{Euler-Lagrange equation}

We now turn our attention to area-constrained minimizers of
\begin{equation}
    \cE(E) := P_\phi(E) + \int_E g(x)dx
\end{equation}
where $g$ is some continuous coercive potential. 

\begin{lemma}[Euler-Lagrange equation]
    \label{lem:EL}
    Let $E$ be an area-constrained minimizer of $\cE$. There exists a constant $\lambda\in\R$ so that along any edge $I\sb\partial E$,
    \begin{equation}
        \label{eq:EL}
         \kappa^\phi_E(I) + \dashint_I g\, dP_\phi = \lambda \qquad \text{where}\qquad \dashint_I g\,dP_\phi := \inv{P_\phi(I)}\int_I g\,dP_\phi.
    \end{equation}
    Moreover, \eqref{eq:EL} also holds if $I$ is a line segment within a fuzzy edge.
\end{lemma}
\begin{proof}
It is straightforward to check that $E$ is an almost-minimizer of $P_\phi$ and thus is Lipschitz $\phi$-regular. Consider a line segment $I\sb \partial E \cap \{x\cdot\nup_I = s_I\}$. For sufficiently small $t$, we consider a perturbed Lipschitz $\phi$-regular set $E_t$ such that $I$ is translated to $\{x\cdot\nup_I = s_{I} + t\}$, and the portions of $\partial E$ adjacent to $I$ are correspondingly shortened or lengthened depending on $t$. It is straightforward to check the asymptotics 
    \begin{align}
        \label{eq:EL perimeter approx}
        P_\phi(E_t) &= P_\phi(E) + \alpha_I \phi(\nu_I)t = P_\phi(E) + \kappa^\phi_E(I)P_\phi(I)t \\
        \label{eq:EL area approx}
        |E_t| &= |E| + P_\phi(I)t + o(t)\\
        \label{eq:EL potential approx}
        \int_{E_t}g(x)dx &= \int_E g(x)dx + t\int_I g\, dP_\phi + o(t).
    \end{align}
    
If $I$ is a fuzzy edge, we similarly define $E_t$ so that $I$ is translated by $tp_I$, in which case $P_\phi(E_t) = P_\phi(E)$ by $\phi$-minimality, and one has the same asymptotics \eqref{eq:EL area approx} and \eqref{eq:EL potential approx}. For instance, we note that for $s\in[0,t]$, the level set $\{\sdp_E = s\} \cap (E_t\Delta E)$ coincides with $I + sp_I$ up to $o(t)$ error in length, and thus we can verify \eqref{eq:EL area approx} using the anisotropic coarea formula \eqref{eq:anisotropic coarea} and \cref{lem:eikonal}:
\begin{equation*}
    |E_t| - |E| = \int (\chi_{E_t} - \chi_E) \phi(\sdp_E) dx = \int_0^t P_\phi(\{\sdp_E \leq s\}\cap (E_t\Delta E))ds = t P_\phi(I) + o(t).
\end{equation*}
One similarly verifies \eqref{eq:EL potential approx} in this case using \eqref{eq:weighted coarea}.
    
Now let $I,J\sb\partial E$ be line segments or fuzzy edges which are disjoint. We can then define $E_{t,s}$ which translates $I, J$ respectively by heights $t, s$. Setting $ P_\phi(I)t = - P_\phi(J)s$, we have $|E_{t,s}| = |E| + o(t)$ and
\begin{equation}
    \cE(E_{t,s}) - \cE(E) = \ps{\okp_E(I) + \dashint_I g\,dP_\phi -\ps{\okp_E(J) + \dashint_J g\,dP_\phi}} P_\phi(I)t + o(t).
\end{equation}
Applying \cite[Lemma 3]{Figalli2011}, we may obtain a further perturbation $\tilde{E}_{t,s}$ such that $|\tilde{E}_{t,s}| = |E|$ and $\cE(\tilde{E}_{t,s}) = \cE(E_{t,s}) + o(t)$. Since $\cE(E_{t,s}) \geq \cE(E)$, it follows that
    \[ \okp_E(I) + \dashint_I g\, dP_\phi = \okp_E(J) + \dashint_J g\, dP_\phi. \]
Thus $\okp_E(I) + \dashint_I g\,dP_\phi$ is constant, as desired. 
\end{proof}

\section{Area-preserving $\phi$-curvature flow}
\label{sec:flow solutions}

\subsection{ODE flow}
\label{sec:ODE flow}
For Lipschitz $\phi$-regular initial data, we can define the area-preserving $\phi$-mean curvature flow as a system of ODEs parameterizing translates of $F$.

First let us clarify the notion of a translate. Fix a Lipschitz vector field $X$ on $\partial F$ such that $X \in \partial \phi(\nu_F)$ along $\partial^*F$. There exists some $\eps_0 = \eps_0(\phi,F)$, such that $\{\dd^\phic_F < \eps_0\}$ is a $\phic$-tubular neighborhood of $\partial F$, i.e. for all $|t| < \eps_0$, the function $x\mapsto x + tX(x)$ maps $\partial F$ bijectively onto $\{\sdp_F = t\}$. Denoting $U(\Gamma) := \set{x + tX(x): x\in \Gamma, |t| < \eps_0}$ for any $\Gamma\sb\partial F$, we observe that
    \begin{itemize}
        \item If $I \sb \partial F$ is a regular edge, then $U(I)$ is a trapezoidal region on which $\sdp_F(x) = x\cdot \nu_I^\phi - s_I$.
        \item If $J \sb \partial F$ is a fuzzy edge, then $\{\sdp_F = t\} \cap U(J) = J + tp_J$ for $|t| < \eps_0$. 
    \end{itemize}
We define another quantity of interest $\eps_1=\eps_1(\phi,F)$ as follows. For a fuzzy edge $J\sb\partial F$ with neighboring regular edges $I_1, I_2$, we define the height of $J$ to be $|s_{I_1} - s_{I_2}|$. We then define $\eps_1$ to be the minimum height of all type II fuzzy edges, and otherwise $\eps_1 := +\infty$ if there are no such edges.

Now consider tuples $h := (h_I)_I$, where $I$ ranges over all edges of $F$. We further require that whenever a regular edge $I$ is adjacent to a fuzzy edge $J$, 
\begin{equation}
    \label{eq:translate h requirement}
    \begin{cases}
        h_I \leq h_J &\text{if }I\text{ is positive}\\
        h_I \geq h_J &\text{if }I\text{ is negative}
    \end{cases}
\end{equation}
so that $\{x\cdot\nup_I = s_I + h_I\}$ and $J + h_J p_J$ must intersect. For all such $h$ satisfying $\|h\|_\infty < \min\{\eps_0, \inv{2} \eps_1\}$, there exists a unique Lipschitz regular set $F(h)$ with translated edges $I(h)$ satisfying
    \begin{equation}
        \label{eq:translate}
        \begin{cases}
            I(h) \sb \{x\cdot \nu_I^\phi =  s_I + h_I\} &\text{if }I\text{ is regular}\\
            I(h) \sb I + h_I p_I &\text{if }I\text{ is fuzzy.}
        \end{cases}
    \end{equation}
We call $F(h)$ a \textbf{translate} of $F$ (see \cref{fig:translate}). The bound $\|h\|_\infty < \inv{2} \eps_1$ ensures that type II fuzzy edges do not vanish, although we allow for type I fuzzy edges to vanish. Conveniently, the surface energy of translates is linear in $h$ (\cref{lem:translate perimeter}): 
\begin{equation}
    \label{eq:translate perimeter}
    P_\phi(F(h)) = P_\phi(F) + \sum_{I\sb\partial F\text{ regular}} \alpha_I \phi(\nu_I) h_I.
\end{equation}

\begin{figure}
    \centering
    \includegraphics[width=0.7\linewidth]{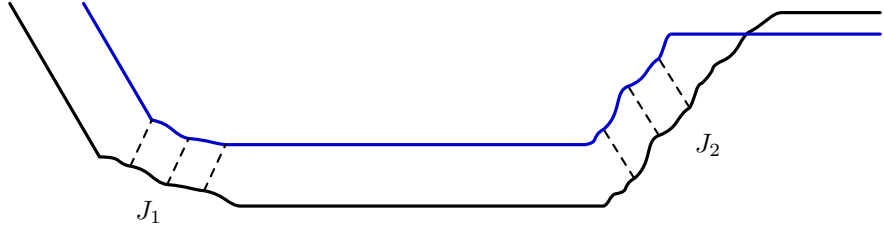}
    \caption{Translate of a Lipschitz $\phi$-regular set. $J_1$ and $J_2$ are respectively type I and type II fuzzy edges.}
    \label{fig:translate}
\end{figure}

Within the space of translates, we express the evolution \eqref{eq:cr flow} of $F$ under area-preserving $\phi$-mean curvature flow as a system of ODEs:
\begin{equation}
    \label{eq:ODE flow h}
    \frac{dh_I}{dt} = - \kp_{F(h)}(I(h)) + \okp_{F(h)}
\end{equation}
recalling $\okp_F := \inv{P_\phi(F)} \int_{\partial F} \kp_F dP_\phi$. Indeed, we can verify the rate of energy dissipation
\begin{align*}
    \frac{d}{dt}P_\phi(F(h)) &= \sum_{I\sb\partial F \text{ regular}} \alpha_I \phi(\nu_I) (-\kp_{F(h)}(I(h)) + \okp_{F(h)})\\
    &= \int_{\partial F(h)} \kp_{F(h)} (-\kp_{F(h)} + \okp_{F(h)}) dP_\phi\\
    &= - \int_{\partial F(h)} |\kp_{F(h)} - \okp_{F(h)}|^2 dP_\phi
\end{align*}
and that \eqref{eq:ODE flow h} preserves area:
\begin{equation}
    \label{eq:ODE flow area pres}
    \frac{d}{dt}|F(h)| = \sum_{I\sb\partial F} \frac{dh_I}{dt} P_\phi(I) = \int_{\partial F(h)} (-\kp_{F(h)} + \okp_{F(h)}) dP_\phi = 0. 
\end{equation}

Since all fuzzy edges of $F(h)$ travel at the same velocity, they should occupy the same level set of $\sdp_F$ at any given time. Thus, it is convenient to define a new system of coordinates as follows: enumerating the regular edges of $F$ as $I_1, \dots, I_m$, we let $u := (\ol{u}, u_1, \dots, u_m)$, where $h_J = \ol{u}$ for all fuzzy edges $J$, and $h_{I_j} = \ol{u} + u_j$. Abusing notation, we refer to such translates as $F(u)$ and their regular edges as $I_j(u)$. In terms of $u$, \eqref{eq:ODE flow h} is equivalent to the evolution
\begin{equation}
    \label{eq:ODE flow}
    \begin{cases}
        \frac{du_j}{dt} = -\kp_{F(u)}(I_j(u))\\
        \frac{d\ol{u}}{dt} = \okp_{F(u)}.
    \end{cases}
\end{equation}
The vector field $Y(u) := (\okp_{F(u)}, -\kp_{F(u)}(I_j(u)))$ may be discontinuous if $F$ has fuzzy edges, since $|I_j(u)|$ may experience jumps with respect to $u_j$. Thus, \eqref{eq:ODE flow} is understood in the Filippov sense \cite{Filippov1988}: solutions to \eqref{eq:ODE flow} are absolutely continuous curves $u(t)$ satisfying the inclusion $\frac{du}{dt} \in \co Y(u(t))$ for a.e. $t$, where $\co Y(u)$ is the convex hull of limit points of $Y(u)$. We also note that the criterion \eqref{eq:translate h requirement} is equivalent to $\sigma_{I_j} u_j \leq 0$. Thus, we fix the domain $\Omega_F := \{\sigma_{I_j} u_j \leq 0\}$ and remark that solutions to \eqref{eq:ODE flow} remain in $\Omega_F$ since $\sigma_{I_j} \frac{du_j}{dt} \leq 0$.

\begin{proposition}[Well-posedness of ODE flow]
    \label{prop:ODE wellposed}
    Let $F\sb\R^2$ be a Lipschitz $\phi$-regular set. There exist constants $\eps, C$ such that for any $u^0 \in \Omega_F$ satisfying $|u^0| < \eps$, there is a unique local-in-time solution $u(t)$ to \eqref{eq:ODE flow} such that $u(0) = u^0$. Defining $\Psi^F_t$ to be the associated flow map, i.e. $\Psi^F_t(u^0) := u(t)$, then $\Psi^F_t$ obeys Lipschitz stability:
    \begin{equation}
        \label{eq:ODE flow stability}
        |\Psi^F_t(u) - \Psi^F_t(u')| \leq e^{Ct}|u-u'|
    \end{equation}
    for all $t$ such that both solutions exist. Moreover, $C$ depends only on $\phi$, $P_\phi(F)$, $\chi(F)$, and $\eps_0(F)$, while $\eps$ depends on the same parameters and also $\eps_1(F)$.
    
\end{proposition}
\begin{proof}
    Existence follows from \cite[Sec 2.7]{Filippov1988}. We claim it suffices to show that $Y$ satisfies the one-sided Lipschitz inequality
    \begin{equation}
        \label{eq:one sided lipschitz}
        (Y(u) - Y(u')) \cdot(u-u') \leq C|u-u'|^2.
    \end{equation}
    Indeed, by standard approximation, \eqref{eq:one sided lipschitz} implies that $(y-y')\cdot(u-u') \leq C|u-u'|^2$ for all $y\in \co Y(u)$ and $y' \in \co Y(u')$. In particular, for any two solutions $u(t), u'(t)$ and a.e. $t$,
    \[
    \frac{d}{dt}|u(t) - u'(t)|^2 = 2(\frac{du}{dt} - \frac{du'}{dt})\cdot(u - u') \leq 2C|u(t) - u'(t)|^2.
    \]
    Thus uniqueness and \eqref{eq:ODE flow stability} follow from Gr\"onwall's inequality. 
    
    It remains to check \eqref{eq:one sided lipschitz}. Using \eqref{eq:gauss bonnet} and \eqref{eq:translate perimeter}, we can express 
    \begin{equation}
        \label{eq:translate okp}
        \ol{Y}(u) = \okp_{F(u)} = \frac{\chi(F) P_\phi(W_\phi)}{P_\phi(F(u))} = \frac{\chi(F) P_\phi(W_\phi)}{P_\phi(F) + \chi(F) P_\phi(W_\phi) \ol{u} + \sum_j \alpha_{I_j} \phi(\nu_{I_j}) u_j}
    \end{equation}
    so $\ol{Y}(u)$ is Lipschitz.

    We now check that $Y_j(u) = -\frac{\alpha_{I_j}}{|I_j(u)|}$ satisfies an upper Lipschitz inequality with respect to $u_j$. For $\eps < \eps_0$ sufficiently small, we have $|I_j(u)| \geq \inv{2} \min_j |I_j| \gtrsim_\phi \eps_0$ for $|u| < \eps$. We denote $\ol{e}$ and $e_j$ to be the standard basis vectors for $u$. We break into cases:
    \begin{enumerate}
        \item If $I_j$ is adjacent to two regular edges, then we have the explicit formula 
        \begin{equation}
             |I_j(u)| = |I_j| + \alpha_{I_j} \ol{u} + [\csc\theta_j \phi(\nu_{I_{j-1}}) u_{j-1} + \csc\theta_{j+1} \phi(\nu_{I_{j+1}}) u_{j+1} - (\cot\theta_j + \cot\theta_{j+1})\phi(\nu_{I_{j}})u_j].
        \end{equation}
        where $\theta_j, \theta_{j+1}$ are the angles between the pairs $\nu_{I_{j-1}}, \nu_{I_j}$ and $\nu_{I_j}, \nu_{I_{j+1}}$ respectively. In particular, $|I_j(u)|$ is an affine function of $u$, so $Y_j(u)$ is Lipschitz.

        \item Suppose $I_j$ is adjacent to two fuzzy edges, say $J$ and $J'$, and assume $I_j$ is negative. Note that for $0\leq u_j < \eps_1$, $I_j(u)$ has endpoints on the translated fuzzy edges $J(\ol{u}\ol{e}) = J + \ol{u}p_J$ and $J'(\ol{u}\ol{e}) = J' + \ol{u} p_{J'}$, so we may express
        \begin{equation}
            \label{eq:length reg fuzzy neighbors}
            |I_j(u)| = |I_j(u_je_j)| + \alpha_{I_j} \ol{u}. 
        \end{equation}
        Moreover, $|I_j(u_je_j)|$ is monotone increasing in $u_j$ due to the cone conditions satisfied by $J$ and $J'$ via $\phi$-minimality. If $I_j$ is positive, then \eqref{eq:length reg fuzzy neighbors} also holds for $-\eps_1 < u_j \leq 0$, and $|I_j(u_je_j)|$ is monotone decreasing. In either case, recalling that $\alpha_{I_j}$ shares the same sign as $I_j$, it follows that
        \begin{align*}
            (Y_j(u) - Y_j(u'))(u_j - u_j') &= -\alpha_{I_j} \ps{\inv{|I_j(u)|} - \inv{|I_j(u')|}}(u_j - u_j') \\
            &= \alpha_{I_j} \frac{(|I_j(u_je_j)| - |I_j(u_j'e_j)|)(u_j - u_j')}{|I_j(u)| |I_j(u')|} + \alpha_{I_j}^2 \frac{(\ol{u} - \ol{u}')(u_j - u_j')}{|I_j(u)||I_j(u')|}\\
            &\leq \alpha_{I_j}^2 \frac{(\ol{u} - \ol{u}')(u_j - u_j')}{|I_j(u)||I_j(u')|}\\
            &\leq C|u-u'|^2.
        \end{align*}

        \item If $I_j$ is adjacent to both a regular edge and a fuzzy edge, then one may argue with a combination of the previous two cases. 
    \end{enumerate}
    
    We deduce \eqref{eq:one sided lipschitz}. \qedhere
\end{proof}

    So far, we have only defined the evolution $F_t := F(u(t))$ within a tubular neighborhood $\{\dd_F^\phic < \eps\}$. Provided $\eps_0(F_t)$ and $\eps_1(F_t)$ stay positive, one can continue the flow outside the tubular neighborhood by resetting $u(t) := \Psi_{t-t_0}^{F_{t_0}}(0) + u(t_0)$ for $t\geq t_0$ for a suitable choice of $t_0$, and $F_t$ can still be viewed as a translate of $F$ in the sense of \eqref{eq:translate}.

    If at some time $T$ one has $\lim_{t\to T} \eps_1(F_t) =0$ but $\liminf_{t\to T} \eps_0(F_t) > 0$, then there must be two parallel regular edges $I_1, I_2$ joined by a type II fuzzy edge $J$, such that $\lim_{t\to T} |s_{I_1(t)} - s_{I_2(t)}| = 0$. The limiting set $F_T := \lim_{t\to T} F_t$ (in Hausdorff distance) is still a Lipschitz $\phi$-regular set, where $J$ has vanished and $I_1, I_2$ have merged into a single line segment that is either a fuzzy edge or a segment of a fuzzy edge (depending on the orientation of the other neighbors of $I_1$ and $I_2$); see \cref{fig:vanishing fuzzy}. Importantly, $\eps_0(F_T) \geq \liminf_{t\to T} \eps_0(F_t)$, so one can (uniquely) restart the evolution \eqref{eq:ODE flow} with coordinates adapted to $F_T$. 

    \begin{figure}
        \centering
        \includegraphics[width=0.6\linewidth]{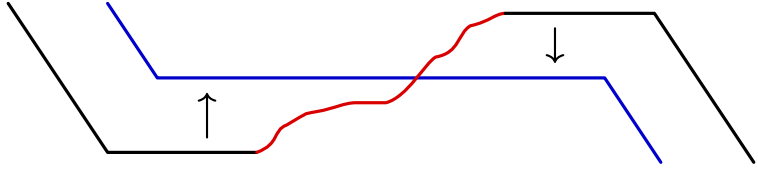}
        \caption{Type II fuzzy edge vanishing in the ODE flow}
        \label{fig:vanishing fuzzy}
    \end{figure}

    We formalize this discussion by defining the ODE evolution to extend past vanishing type II fuzzy edges whenever possible. We only consider evolutions in which vanishing times are distinct, for sake of simplicity in studying the stability of the evolution. 

\begin{definition}
    \label{def:ODE flow}
    We say that a family of Lipschitz $\phi$-regular sets $(F_t)_{t\in[0,T)}$ satisfies the \textbf{\emph{area-preserving $\phi$-mean curvature ODE flow}} from $F$, and define the flow map $\Psi_t(F) := F_t$, if there exist times $0 = t_0 < t_1 < \cdots < t_n = T$ such that
    \begin{itemize}
        \item For all $0\leq i < n$ and $s\in [t_i, t_{i+1})$, there exists $\eps>0$ such that $F_t = F_s(\Psi^{F_s}_{t-s}(0))$ for all $t\in[s,s+\eps)$. 

        \item For all $1\leq i < n$, a single type II fuzzy edge of $F_t$ vanishes as $t\to t_i^-$, $\liminf_{t\to t_i^-}\eps_0(F_t) > 0$, and $F_{t_i}$ is the limit of $F_t$ in Hausdorff distance. 
    \end{itemize}
\end{definition}

We conclude this section with a slightly more general stability result than \eqref{eq:ODE flow stability}. First we introduce the notion of a rough translate, wherein we relax the requirement for fuzzy edges to be translates of those of $F$.

\begin{definition}
    Given Lipschitz $\phi$-regular sets $E,F$, we say that $E$ is a \textbf{\emph{rough translate}} of $F$ if $\partial E \sb \{\dd_F^\phic < \eps_0\}$ and the following hold:
    \begin{itemize}
        \item The regular edges $I\sb\partial F$ are in one-to-one correspondence with the regular edges $I^E\sb\partial E$, such that $\nu_{I^E} = \nu_I$.
        \item If $I_1$ and $I_2$ are adjacent or joined by a fuzzy edge $J$, then $I_1^E$ and $I_2^E$ are also either adjacent or joined by a fuzzy edge $J^E$ satisfying the $p_J$ cone condition.
    \end{itemize}
\end{definition}

\begin{proposition}
    \label{prop:ODE stability hausdorff}
    Let $E, F\sb\R^2$ be Lipschitz $\phi$-regular which are rough translates, and suppose $F_t := \Psi_t(F)$ and $E_t := \Psi_t(E)$ exist for $t\in[0,T]$ and satisfy $\min\{\eps_0(E_t),\eps_0(F_t)\}\geq\eps$ for all $t\in[0,T]$. Then there exist $\delta, C>0$ depending only on $\phi, \eps, P_\phi(F), |F|, \chi(F)$ such that if $d_H(E,F) < e^{-CT}\delta$, then
    \begin{equation}
        \label{eq:ODE stability hausdorff}
        d_H(E_t, F_t) \leq Ce^{Ct} d_H(E,F) \qquad\forall t\leq T.
    \end{equation}
    Moreover, if $F_t$ experiences a vanishing type II fuzzy edge at time $t_1 < T - Ce^{CT}d_H(E,F)$, then $E_t$ also experiences vanishing at time $t_1'$ such that $|t_1' - t_1| \leq Ce^{Ct_1}d_H(E,F)$. 
\end{proposition}

\begin{remark}
    At first glance, it may seem that \eqref{eq:ODE stability hausdorff} is a straightforward extension of \eqref{eq:ODE flow stability}. However, the proximity of $E,F$ in Hausdorff distance only implies $L^1$ proximity of the corresponding vector fields $Y^E$ and $Y^F$ (see \eqref{eq:length diff area bound}), in which case DiPerna-Lions theory suggests only that the flow maps $\Psi_t^E$ and $\Psi_t^F$ are close in $L^1$. To obtain uniform proximity of trajectories, we exploit the fact that $|Y_j^E|$ and $|Y_j^F|$ are bounded below. 
\end{remark}

\begin{proof}   
    The constant $C$ has only the aforementioned dependencies and may change from line to line. First, we assume that $E_t$ and $F_t$ have no vanishing fuzzy edges for $t\leq T$. Then we can express $F_t = F(u(t))$ and $E_t = E(u'(t))$ where $u(t), u'(t)$ are respectively solutions to \eqref{eq:ODE flow}, and with a departure in notation from earlier in the section, we let $\{I_j(t)\}_{j=1}^m$ and $\{I_j'(t)\}_{j=1}^m$ denote the corresponding evolution of edges. By $\eps \lesssim_\phi |I_j(t)|, |I_j'(t)| \lesssim_\phi P_\phi(F)$, we have the bound $\inv{C} \leq |\frac{du_j}{dt}|, |\frac{du_j'}{dt}| \leq C$. Moreover by the Wulff inequality and recalling $|F(u(t))| = |F|$ by \eqref{eq:ODE flow area pres},
    \begin{equation*}
        \Big|\frac{d\ol{u}}{dt}\Big| = |\okp_{F(u(t))}| = \frac{|\chi(F)|P_\phi(W_\phi)}{P_\phi(F(u(t)))} \leq \frac{|\chi(F)|P_\phi(W_\phi)^{1/2}}{|F|^{1/2}} \leq C
    \end{equation*}
    and thus $|\frac{du}{dt}|, |\frac{du'}{dt}|\leq C$.

    We reduce to the assumption that $s_{I_j'} = s_{I_j}$ by considering the translate $\tilde{F} = F(u^0)$ such that $\ol{u}^0 =0$ and $u^0_j = s_{I_j'} - s_{I_j}$, in which case the regular edges $\tilde{I}_j$ of $\tilde{F}$ satisfy $s_{\tilde{I}_j} = s_{I_j'}$. Note that $|u^0| \lesssim_\phi d_H(E,F)$. By \cref{prop:ODE wellposed}, the translated flow $\tilde{F}_t := F(\Psi^F_t(u^0))$ satisfies
    \begin{equation}
        d_H(\tilde{F}_t, F_t) \lesssim_\phi |\Psi^F_t(u^0) - \Psi^F_t(0)| \leq Ce^Ct|u_0| \lesssim_\phi Ce^{Ct} d_H(E,F).
    \end{equation}
    If we are able to show $d_H(E_t, \tilde{F}_t) \leq Ce^{Ct} d_H(E, \tilde{F})$, then \eqref{eq:ODE stability hausdorff} follows by the triangle inequality.
    
    Thus we may assume $s_{I_j'} = s_{I_j}$ for all $j$. Recall from the proof of \cref{prop:ODE wellposed} that we can express
    \begin{align}
        |I_j(t)| &= \ell_j(u_j(t)) + f_j(u(t)) \label{eq:length formula} \\
        |I_j'(t)| &= \ell_j'(u_j(t)) + f_j(u(t)) \label{eq:length formula k}
    \end{align}
    where $\ell_j(s)$ is the length of $I_j$ in $F(se_j)$, $\ell_j'(s)$ is the length of $I_j'$ in $E(se_j)$, and $f_j$ is a linear function depending only on $\phi$. The key is to control the difference between $\ell_j$ and $\ell_j'$, which are not uniformly close, but are close in $L^1$. For $d_H(E,F)<\delta$ sufficiently small, we can bound
    \begin{equation}
        \label{eq:length diff area bound}
        \abs{\int_0^{h} \abs{\ell_j'(s) - \ell_j(s)} ds} \leq \inv{\phi(\nu_{I_j})} |E\Delta F| \lesssim_\phi P_\phi(F) d_H(E,F).
    \end{equation}
    To exploit \eqref{eq:length diff area bound}, we define inverse maps $\tau_j(s)$, $\tau'_j(s)$ such that $u_j(\tau_j(s)) = s$ and $u_j'(\tau_j'(s))=s$, which are well-defined since $u_j, u_j'$ are monotone. Note that 
    \begin{equation}
        \label{eq:ODE length integral}
        \int_0^{u_j(t)} |I_j(\tau_j(s))| ds = \int_0^t |I_j(u(t))| \dot{u}_j(t) dt = \int_0^t -\alpha_{I_j} dt = -\alpha_{I_j}t.
    \end{equation}
    The same identity holds for $u_j'(t)$, so we have the equivalence 
    \begin{equation}
        \label{eq:ODE length integral 2}
        \int_0^{u_j(t)} |I_j(\tau_j(s))|ds = -\alpha_{I_j}t = \int_0^{u_j'(t)} |I_j'(\tau_j'(s))|ds.
    \end{equation}
    By symmetry it makes no harm to assume $|u_j(t)| < |u_j'(t)|$ (for a fixed $t$), and we may estimate using \eqref{eq:length formula}-\eqref{eq:length diff area bound} and \eqref{eq:ODE length integral 2}
    \begin{align}
        \inv{C} |u_j'(t) - u_j(t)| &\leq \abs{\int_{u_j(t)}^{u'_j(t)} |I_j'(\tau_j'(s))|ds }\notag \\
        &= \abs{
            \int_0^{u_j(t)} |I_j'(\tau_j'(s))| - |I_j(\tau_j(s))|ds  
        }\notag\\
        &\leq \abs{\int_0^{u_j(t)} (\ell_j'(s) - \ell_j(s)) ds} + \abs{\int_0^{u_j(t)} f_j\big(u'(\tau_j'(s)) - u(\tau_j(s))\big)ds} \notag\\
        &\leq C\ps{d_H(E,F) + \int_0^{u_j(t)} |u'(\tau_j'(s)) - u(\tau_j(s))|ds}. \label{eq:uj inequality}
    \end{align}
    We can estimate the last integrand as 
    \begin{equation}
        \label{eq:uj inequality 2}
        |u'(\tau_j'(s)) - u(\tau_j(s))| \leq |u'(\tau_j'(s)) - u'(\tau_j(s))| + |u'(\tau_j(s)) - u(\tau_j(s))|
    \end{equation}
    and since $|\frac{du}{dt}| \leq C$, $-\sigma_{I_j}\frac{du'_j}{dt} \geq \inv{C}$, and $u_j(\tau_j(s)) = u_j'(\tau_j'(s))$, we can further estimate the first term
    \begin{align}
        \label{eq:uj inequality 3}
        |u'(\tau_j'(s)) - u'(\tau_j(s))| &\leq C |\tau_j'(s) - \tau_j(s)| \notag \\
        &\leq C|u_j'(\tau_j'(s)) - u_j'(\tau_j(s))| \notag \\
        &= C|u_j(\tau_j(s)) - u_j'(\tau_j(s))|.
    \end{align}
    By \eqref{eq:uj inequality 2} and \eqref{eq:uj inequality 3}, 
    \begin{align*}
        \int_0^{u_j(t)} |u'(\tau_j'(s)) - u(\tau_j(s))|ds &\leq C \int_0^{u_j(t)} |u'(\tau_j(s)) - u(\tau_j(s))| ds\\
        &\leq C \int_0^t |u'(t) - u(t)| dt.
    \end{align*}
    Thus \eqref{eq:uj inequality} yields 
    \begin{equation}
        |u_j'(t) - u_j(t)| \leq C \ps{d_H(E,F) +  \int_0^t |u'(t) - u(t)|dt}.
    \end{equation}
    
    Estimating $\ol{u}' - \ol{u}$ is more straightforward since $P_\phi(F(u))$ is an affine function of $u$:
    \begin{align}
        \label{eq:ol u inequality}
        \abs{\ol{u}'(t) - \ol{u}(t)} &\leq \chi(F) P_\phi(W_\phi) \int_0^t \abs{\inv{P_\phi(E(u'(t)))} - \inv{{P_\phi(F(u(t)))} } } dt \notag\\
        &\leq C \int_0^t |u'(t) - u(t)|dt.
    \end{align}
    By \eqref{eq:uj inequality} and \eqref{eq:ol u inequality}, we obtain
    \begin{equation}
        |u'(t) - u(t)| \leq C \ps{d_H(E,F) + \int_0^t |u'(t) - u(t)|dt }
    \end{equation}
    and thus $|u'(t) - u(t)| \leq Ce^{Ct} d_H(E,F)$ by Gr\"onwall's inequality. Finally, \eqref{eq:ODE stability hausdorff} follows from the fact that $d_H(E_t, F_t) \leq d_H(E,F) + C|u'(t) - u(t)|$. 

    Now let us consider the case that $F_t$ has a type II fuzzy edge $J$ which vanishes at $t_1 < T$. Then $J$ is adjacent to parallel regular edges $I_1$ and $I_2$ which are respectively negative and positive, such that $s_{I_1(t)} - s_{I_2(t)}$ increases to 0 as $t\to t_1^-$. By the previous argument, we can estimate 
    \begin{align*}
        s_{I_1'(t)} - s_{I_2'(t)} &= u_1'(t) - u_2'(t) + (s_{I_1'} - s_{I_2'}) \\
        &= u_1(t) - u_2(t) + (s_{I_1} - s_{I_2}) + O(e^{Ct}d_H(E,F))\\
        &= s_{I_1(t)} - s_{I_2(t)} + O(e^{Ct}d_H(E,F)).
    \end{align*}
    Since $\frac{d}{dt}(u_1' - u_2') = -\frac{\alpha_{I_1}}{|I_1(t)|} + \frac{\alpha_{I_2}}{|I_2(t)|} = \frac{L_{I_1}}{|I_1(t)|} + \frac{L_{I_2}}{|I_2(t)|} \geq \inv{C}$, it's clear that the fuzzy edge $J'$ joining $I_1'$ and $I_2'$ must also vanish at a time $t_1'$ such that $|t_1' - t_1| \leq Ce^{Ct_1}d_H(E,F)$. 

    By symmetry, it makes no harm to assume $t_1 \leq t_1'$. By the previous argument, one can apply \eqref{eq:ODE stability hausdorff} to deduce $d_H(E_{t_1}, F_{t_1}) \leq Ce^{Ct_1} d_H(E,F)$. Since vanishing times for $F$ are distinct, the restarted flow $F_t = F_{t_1}(\Psi^{F_{t_1}}_{t-t_1}(0))$ exists for $t\leq t_1'$, and we can estimate
    \begin{align*}
        d_H(E_{t_1'}, F_{t_1'}) &\leq d_H(E_{t_1}, F_{t_1}) + d_H(E_{t_1}, E_{t_1'}) + d_H(F_{t_1}, F_{t_1'}) \\
        &\leq d_H(E_{t_1}, F_{t_1}) + Ce^{C(t_1'-t_1)}(t_1'-t_1)\\
        &\leq Ce^{Ct_1'} d_H(E,F).
    \end{align*}
    Since $E_{t_1'}, F_{t_1'}$ are rough translates, one can re-apply the previous argument to obtain \eqref{eq:ODE stability hausdorff} until the next vanishing time of either flow. Iterating in this manner, we conclude the proof.
\end{proof}

\begin{corollary}
    \label{cor:ODE stability hausdorff 2}
    Let $F\sb\R^2$ be Lipschitz $\phi$-regular such that $F_t := \Psi_t(F)$ exists for $t\in[0,T)$. For all $T_0 < T$ and $\alpha\in (0,1)$, there exist $\delta, C > 0$ such that if $E$ is a rough translate of $F$ satisfying $d_H(E,F) < e^{-CT_0} \delta$ and $\eps_0(E) > \alpha \eps_0(F)$, then $E_t := \Psi_t(E)$ exists on $[0,T_0]$, and
    \begin{equation}
        \label{eq:ODE stability hausdorff 2}
        d_H(E_t, F_t) \leq Ce^{Ct} d_H(E,F) \quad\forall t < T_0.
    \end{equation}
    Moreover, $\delta, C$ depend only on $\phi, F, \alpha$, and $\inf_{t\in [0,T_0]} \eps_0(F_t)$.
\end{corollary}

\begin{remark}
    The parameter $\alpha$ is needed since rough translates may have arbitrarily short regular edges.
\end{remark}

\begin{proof}    
    We employ a standard bootstrapping argument. Let $E$ be a rough translate of $F$ satisfying $d_H(E,F) < e^{-CT_0}\delta$ and $\eps_0(E) > \alpha \eps_0(F)$, and let $T'$ be the maximal time of existence for $E_t$. Let $\eps_* := \inf_{t\in [0,T_0]} \eps_0(F_t) > 0$ and consider
    \[ t_* := \inf\{t < T': \eps_0(E_t) < \aa\eps_*\} > 0.  \]
    Since $\eps_0(E_t) \geq \aa\eps_*$ for all $t\leq t_*$, \cref{prop:ODE stability hausdorff} implies the existence of $\delta, C$ such that \eqref{eq:ODE stability hausdorff 2} holds for $t\leq t_*$. We merely need to check that $t_* \geq T_0$ provided $\delta$ is sufficiently small. 
    
    Suppose for sake of contradiction that $t_* < T_0$. By continuity, $\eps_0(E_{t_*}) = \aa\eps_*$. For brevity we denote $d := d_H(E,F)$. Note that for any Lipschitz $\phi$-regular $F$, one has $\eps_0(F) \leq \ell_0(F) := \min_j \frac{|I_j|}{L_{I_j}}$. Moreover, $\partial E \sb \{\dd_F^\phic < Ce^{Ct}d\}$ implies 
    \begin{equation}
        \label{eq:eps0}
        \eps_0(E_t) \geq \min\{\ell_0(E_t), \eps_0(F_t) - Ce^{Ct}d\}. 
    \end{equation}

    To lower bound $\ell_0(E_t)$, let us consider a negative edge $I_j'$ of $E$. Borrowing the notations $\ell_j, \ell_j', u_j, u_j'$ from the proof of \cref{prop:ODE stability hausdorff}, it is straightforward to see $\ell_j'(s) \geq \ell_j(s - cd)$ for some $c=c(\phi)$. Since $|u_j(t) - u_j'(t)| \leq Ce^{Ct}d$ and $\frac{du_j}{dt} \geq \inv{C}$, we can find $t_1 \in [t_* - Ce^{Ct}d, t_*]$ such that $u_j(t_1) \leq u_j'(t_*) - cd$. By the monotonicity of $\ell_j$, we obtain
    \begin{align*}
        |I_j'(t_*)| &= |I_j(t_1)| + \ell_j'(u_j'(t_*)) - \ell_j(u_j(t_1)) + f_j(u'(t_*) - u(t_1)) \\
        &\geq |I_j(t_1)| + \ell_j'(u_j'(t_*)) - \ell_j(u_j'(t_*) - cd) - Ce^{Ct_*}d\\
        &\geq |I_j(t_1)| - Ce^{Ct_*}d.
    \end{align*}
    One obtains the same bound if $I_j'$ is positive, and thus $\ell_0(E_{t_*}) \geq \ell_0(F_{t_1}) - Ce^{Ct_*}d \geq \eps_0(F_{t_1}) - Ce^{Ct_*}d$. Combining this with \eqref{eq:eps0}, we find that $\aa\eps_* = \eps_0(E_{t_*}) \geq \eps_* - Ce^{Ct_*}d$ and thus $(1-\alpha)\eps_* \leq Ce^{CT_0}\delta$, from which we obtain a contradiction for $\delta$ sufficiently small. 
\end{proof}

\subsection{Flat flow}
Here we establish the existence of the flat flow solution to \eqref{eq:cr flow} and develop the estimates necessary for studying long-time behavior in \cref{sec:long time}. We closely follow the treatments in \cite{Mugnai2016, KimKwon2024}, though a key difference is that the Euler-Lagrange equation holds only in an averaged sense over edges, and thus the $L^2$ bound on the Lagrange multiplier (\cref{prop:lambda L2 bound}) requires a different argument.

\begin{definition}
    \label{def:flat flow}
    For a bounded set of finite perimeter $E_0\sb\R^2$ and a fixed timestep $\tau>0$, we define an \textbf{\emph{approximate flat $\phi$-flow}} of $E_0$ to be an evolution of sets $\Et: [0,\infty) \to BV(\R^2;\{0,1\})$ given by setting $\Et_0 := E_0$, iterating the minimization problem $\Et_{k+1} \in \argmin \cF_\tau(\cdot,\Et_k)$ where  
    \begin{equation}
        \label{eq:energy vol}
        \cF_\tau(E,F) := P_\phi(E) + \inv{\tau} \int_E \sdp_F(x) dx + \inv{\tau^{1/2}}\abs{|E|-|E_0|}
    \end{equation}
    and defining $\Et(t) := \Et_{\floor{t/\tau}}$. We then define a \textbf{\emph{flat $\phi$-flow}} $E(t)$ of $E_0$ to be an $L^1$-subsequential limit of approximate flat flows $\Etn(t)$ as $\tau_n\to 0$, i.e. $\lim_{n\to\infty} |\Etn(t)\Delta E(t)| = 0$ for all $t\geq0$. 
\end{definition}

\begin{theorem}
    \label{thm:flat flow existence}
    Let $E_0\sb\R^2$ be a set of finite perimeter. There exists a flat $\phi$-flow $\{E(t)\}_{t\geq0}$ starting from $E_0$. Moreover any such flow satisfies $P_\phi(E(t)) \leq P_\phi(E_0)$ and $|E(t)| = |E_0|$ for all $t\geq0$, and for $C=C(L_\phi)$,
    \begin{equation}
        \label{eq:holder cty in time}
        |E(t)\Delta E(s)| \leq CP_\phi(E)|t-s|^{1/2} \qquad \forall s,t\in[0,\infty)
    \end{equation}
\end{theorem}

We define the dissipation 
\[ \cD(E,F) := \int_{E\Delta F} \dd_F^\phic(x)dx = \int_E \sdp_F(x) dx - \int_F \sdp_F(x)dx. \]
Then minimizers $E$ of $\cF_\tau(\cdot,F)$ satisfy the energy dissipation inequality 
\begin{equation}
    \label{eq:dissipation}
    P_\phi(E) + \inv{\tau} \cD(E,F) + \inv{\tau^{1/2}}||E|-|E_0|| \leq P_\phi(F) + \inv{\tau^{1/2}}||F|-|E_0||.
\end{equation}
Thus approximate flat flows satisfy an iterated dissipation inequality: for any $i\leq k$,
\begin{equation}
    \label{eq:iterated dissipation}
    P_\phi(\Et_k) + \inv{\tau} \sum_{j=i}^{k-1} \cD(\Et_{j+1},\Et_j) + \inv{\tau^{1/2}}||\Et_k| - |E_0|| \leq P_\phi(\Et_i) + \inv{\tau^{1/2}}||\Et_i|- |E_0||.
\end{equation}

The following estimates are proven in \cite[Appendix B]{KimKwon2024}, with no assumptions on the regularity of $\phi$. We note that the results in \cite{KimKwon2024} are stated only for even anisotropies (\emph{i.e.} norms), but the evenness assumption is not necessary for any of the proofs. 
\begin{lemma}
    \label{lem:standard estimate flat flow cr}
    Let $E\in\argmin \cF_\tau(\cdot, F)$. There exist constants $c_0,C>0$ depending only on $L_\phi$ such that
    \begin{itemize}
        \item[(i)] ($L^\infty$ estimate) \begin{equation}
            \label{eq:Linfty estimate cr}
            \sup_{E\Delta F} \dd^\phic_F \leq c_0\tau^{1/2}
        \end{equation}
        \item[(ii)] ($L^1$ estimate) If $F\in\argmin \cF_\tau(\cdot, G)$ for some $G$, then for all $\ell\leq c_0\tau^{1/2}$,
        \begin{equation}
            \label{eq:L1 estimate cr}
            |E\Delta F|\leq C\ps{\ell P_\phi(F) + \inv{\ell} \cD(E,F)}
        \end{equation}
        \item[(iii)] ($L^2$ estimate) \begin{equation}
            \label{eq:L2 estimate cr}
            \int_{\partial E} (\dd^\phic_F)^2 d\cH^1 \leq C\cD(E,F)
        \end{equation}
    \end{itemize}
    Moreover, if $\Et(t)$ is an approximate flat $\phi$-flow starting from $E_0$, then 
    \begin{itemize}
        \item[(iv)] (H\"older continuity in time) For all $\tau\leq s < t < \infty$,
        \begin{equation}
            \label{eq:Holder cty in time tau}
            |\Et(s)\Delta \Et(t)| \leq CP_\phi(E_0) \max\{t-s, \tau\}^{1/2}
        \end{equation}
        \item[(v)] ($L^2$ velocity bound)
        \begin{equation}
            \label{eq:L2 velocity bound}
            \int_\tau^\infty \int_{\partial \Et(t)} \bigg(\frac{\dd_{\Et(t-\tau)}^\phic}{\tau}\bigg)^2 d\cH^1 dt \leq CP_\phi(E_0). 
        \end{equation}
    \end{itemize}
\end{lemma}

By the $L^\infty$ estimate \eqref{eq:Linfty estimate cr}, it is straightforward to check that any $E\in\argmin \cF_\tau(\cdot, F)$ is a $(\Lambda,r_0)$-minimizer of $P_\phi$ for $\Lambda \sim_\phi \tau^{1/2}$, and thus by a slight variant of \cref{lem:EL} (owing to the soft volume constraint), $E$ satisfies the Euler-Lagrange equation
\begin{equation}
    \label{eq:EL flat}
    \kappa^\phi_E(I) = -\dashint_I \frac{\sdp_F}{\tau} dP_\phi + \lambda
\end{equation}
for all edges $I\sb\partial E$ where $\lambda$ is a Lagrange multiplier such that $\lambda = \tau^{-1/2}\text{sgn}(|E_0| - |E|)$ whenever $|E|\neq |E_0|$. Moreover $\partial^*E \sb \{\nu_E \in \cN_\phi\} \cup \{\sdp_F = \lambda\tau\}$.

\begin{proposition}
    \label{prop:lambda L2 bound}
    Let $\Et(t)$ be an approximate flat flow of $E_0$. There exist constants $C, \tau_0$ depending only on $\phi, |E_0|, P_\phi(E_0)$ such that for all $\tau\leq \tau_0$ and $T>0$,
    \begin{enumerate}
        \item[(i)] $\int_0^T |\lt(t)|^2 dt \leq C(T+1)$
        \item[(ii)] $\Et(T) \sb E_0 + C(T+1) W_\phi$
        \item[(iii)] $|\{t\in[0,T]: |\Et(t)| \neq |E_0|\}| \leq C(T+1)\tau$
    \end{enumerate}
\end{proposition}

\begin{proof}
Fix $\tau\leq \tau_0$. By the dissipation inequality we have $\inv{\sqrt{\tau}}||\Et(t)| - |E_0|| \leq P_\phi(E_0)$, so in particular $|\Et(t)| \geq \inv{2}|E_0|$ for $\tau_0$ sufficiently small.

We claim that for any $t>0$, there exists a boundary component $\Gamma \sb \partial \Et(t)$ such that $P_\phi(\Gamma)\geq 1/C$. Indeed, suppose all boundary components $\Gamma_j \sb \partial \Et(t)$ satisfy $P_\phi(\Gamma_j) < \inv{C}$. Letting $E_j\sb\R^2$ be the region bounded by $\Gamma_j$, the Wulff inequality implies $\frac{P_\phi(E_j)}{|E_j|} \geq \frac{4|W_\phi|}{P_\phi(E_j)} \gtrsim_\phi C$. Then we can bound
\begin{equation}
    \inv{2}|E_0| \leq |\Et(t)| \leq \sum_j |E_j| \lesssim_\phi \inv{C}\sum_j P_\phi(\Gamma_j) = \inv{C}P_\phi(\Et(t)) \leq \inv{C}P_\phi(E_0)
\end{equation}
from which we reach a contradiction for $C$ sufficiently large, proving the claim.

Integrating the Euler-Lagrange equation \eqref{eq:EL flat} over $\Gamma$ and applying Cauchy-Schwarz yields the bound
\begin{align}
    \label{eq:lambda bound}
    |\lt(t)| P_\phi(\Gamma) &= \bigg|\int_\Gamma \Big(\kp_E + \frac{\sdp_{\Et(t-\tau)}}{\tau}\Big) dP_\phi \bigg| \notag \\
    &\leq P_\phi(W_\phi) + P_\phi(\Gamma)^{1/2} \bigg(\int_\Gamma \bigg(\frac{\sdp_{\Et(t-\tau)}}{\tau}\bigg)^2 dP_\phi\bigg)^{1/2}.
\end{align}
By integrating the square of \eqref{eq:lambda bound} and invoking \eqref{eq:L2 velocity bound}, we may bound
\begin{align*}
    \int_0^T |\lt(t)|^2 dt &\leq C\bigg(T + \int_\tau^T \int_{\partial\Et(t)} \bigg(\frac{\sdp_{\Et(t-\tau)}}{\tau}\bigg)^2 dP_\phi dt \bigg) \leq C(T+1)
\end{align*}
proving (i). Note that (iii) then follows from (i) by Markov's inequality: 
\begin{equation*}
    |\{t\in[0,T]: |\Et(t)| \neq |E_0|\}| \leq |\{t\in[0,T]: |\lt(t)| \geq \tau^{-1/2}\}| \leq \tau \int_0^T |\lt(t)|^2dt.
\end{equation*}

It remains to prove (ii). For fixed $\tau$, we define $r_t := \inf\{r: \Et(t) \sb rW_\phi\}$. Let $t\geq \tau$, and consider $x\in \partial \Et(t)$ such that $\phic(x) = r_t$. By the Lipschitz $\phi$-regularity of $\Et(t)$ and \cref{prop:cone lipreg}, $x$ must lie on some positive edge $I\sb\partial \Et(t) \cap \partial (r_t W_\phi)$. Moreover, since $\Et(t-\tau) \sb r_{t-\tau}W_\phi$, we have $r_t - r_{t-\tau} \leq \sdp_{\Et(t-\tau)}(y)$ for all $y\in I$. Then the Euler-Lagrange equation implies the bound \begin{equation}
    \label{eq:r diff}
    r_t - r_{t-\tau} \leq \dashint_I \sd_{\Et(t-\tau)}^\phic dP_\phi = C \tau (-\kappa^\phi_{\Et(t)}(I) + \lt(t)) < C\lt(t)\tau. 
\end{equation}
By summing \eqref{eq:r diff}, we thus obtain 
\begin{equation}
    \label{eq:rT}
    r_T - r_0 \leq C\int_0^T \lt(t) dt
\end{equation}
so (ii) follows from (i) and an application of Cauchy-Schwarz. 
\end{proof}

\medskip

\noindent \textbf{Proof of \cref{thm:flat flow existence}:}

\medskip 

By \cref{prop:lambda L2 bound}(ii) and the perimeter bound $P_\phi(\Et(t)) \leq P_\phi(E_0)$, we may find via diagonalization a subsequence $\tau_n\to0$ such that $\Etn(q)$ converges to some set $E(q)$ in $L^1$ for all $q\in \Q_{\geq 0}$. We then define $E(t) := L^1\text{-}\lim_{q\to t} \Etn(q)$ for all $t\geq0$, in which case \eqref{eq:Holder cty in time tau} implies $\Etn(t) \to E(t)$ for all $t\geq0$ and \eqref{eq:holder cty in time}. The perimeter bound $P_\phi(E(t)) \leq P_\phi(E_0)$ follows by lower-semicontinuity of $P_\phi$. Moreover the dissipation inequality implies the bound $ ||\Etn(t)| - |E_0|| \leq P_\phi(E_0)\tau_n^{1/2}$, yielding $|E(t)| = |E_0|$. \hfill $\Box$

\section{Weak-strong uniqueness}
\label{sec:weak strong uniqueness}

In this section we prove \cref{thm:weak strong informal}, namely the consistency of the flat flow solution with the ODE evolution for Lipschitz $\phi$-regular initial data. We restate the theorem more rigorously as follows:

\begin{theorem}
    \label{thm:weak strong uniqueness}
    Let $E_0\sb\R^2$ be Lipschitz $\phi$-regular and let $T$ be the maximal time of existence for the ODE evolution $\Psi_t(E_0)$. Further assume that type II fuzzy edges vanish in $\Psi_t(E_0)$ at distinct times. If $E(t)$ is an area-preserving flat flow starting from $E_0$, then $E(t) = \Psi_t(E_0)$ for all $t < T$. Moreover, for any $T_0 < T$, there exist $C, \tau_0$ depending only on $\phi,E_0,T_0$ such that
    \begin{equation}
        \label{eq:MM ODE approx}
        d_H(\Et(t), \Psi_t(E_0)) \leq Ce^{Ct} \tau^{1/2} \qquad \forall t\leq T_0, \tau\leq \tau_0.
    \end{equation}
\end{theorem}

The central task at hand is proving that the approximate flat flow $\Et(t)$ produces translates of $E_0$ for sufficiently small $\tau$. Once the approximate flow is known to produce translates, it is mostly standard to show that $\Et(t)$ is a discrete approximation to the ODE evolution with local truncation error $O(\tau^2)$ by using the Euler-Lagrange equation \eqref{eq:EL flat}. The argument is essentially that of Euler's method, although we note that Euler's method may fail to have quadratic local error in the presence of fuzzy edges, since then the velocity $-\frac{\alpha_I}{|I|} + \okp_E$ is only BV. 

\subsection{Minimizer is a translate}
\label{sec:MM translate}

\begin{theorem}
    \label{thm:MM translate}
    Let $F\sb\R^2$ be bounded and Lipschitz $\phi$-regular. There exists a constant $\tau_0$ such that for all $\tau \leq \tau_0$, any minimizer $E\in \argmin \cF_\tau(\cdot, F)$ is a translate of $F$, as defined in \cref{sec:ODE flow}, with the exception that fuzzy edges in $F$ may vanish in $E$. Moreover, $\tau_0$ depends only on $\phi$, $|E_0|$, $P_\phi(F)$, $\chi(F)$, and $\eps_0(F)$. 
\end{theorem}

The proof of \cref{thm:MM translate} is rather lengthy, but boils down to standard energy competition arguments. In most cases, our construction of energy competitors follows a consistent recipe, namely using $\phi$-minimal barriers to transport mass from a region $E^+$ where $\sdp_F$ is high to $E^-$ where $\sdp_F$ is low. This is an extension of corner barriers in \cite{Almgren1995}, adapted for constructing area-preserving perturbations. For convenience sake, we package the essential criteria for these constructions into the following lemma.

\begin{lemma}
    \label{lem:area preserving competitor}
    Let $E\sb\R^2$ be a set of finite perimeter and let $m < M$ be constants. Suppose that $U^+, U^-$ are open sets such that for all sufficiently small $\delta_+, \delta_- > 0$, there exist regions $\Sigma^+, \Sigma^-\sb\R^2$ satisfying the following:
    \begin{enumerate}
        \item $\partial \Sigma^\pm \cap U^\pm$ is a $\phi$-minimal path (when oriented by $\nu_{\Sigma^\pm}$).
        \item Defining the sets $E^+ := (E \setminus \Sigma^+) \cap U^+$ and $E^- := (\Sigma^- \setminus E) \cap U^-$, we have $E^\pm \Subset U^\pm$.
        \item $E^+ \sb \{\sdp_F \geq M - \delta_+\}$ and $E^- \sb \{\sdp_F \leq m + \delta_-\}$.
        \item $|E^\pm|$ is positive, depends continuously on $\delta_\pm$, and $\lim_{\delta_\pm\to 0} |E^\pm| = 0$. 
    \end{enumerate}
    Then there exists $\delta_+, \delta_-$ such that $\cF_\tau((E\cup E^-)\setminus E^+, F) < \cF_\tau(E, F)$. In particular, $E \not\in \argmin \cF_\tau(\cdot, F)$. 
\end{lemma}
\begin{proof}
    By (4), we may choose $\delta_\pm>0$ sufficiently small so that $m + \delta_- < M - \delta_+$ and $|E^+| = |E^-|$. Let $\tilde{E} := (E \cup E^-)\setminus E^+$. For such $\delta_\pm$, $E^+$ and $E^-$ are disjoint, so $|\tilde{E}| = |E|$. Moreover, by (3),
    \begin{align*}
        \int_{\tilde{E}} \sdp_F dx - \int_E \sdp_F dx &= \int_{E^-} \sdp_F dx - \int_{E^+} \sdp_F dx \\
        &\leq [(m + \delta_-) - (M - \delta_+)]\cdot |E^+|\\
        &< 0.
    \end{align*}
    Lastly we claim $P_\phi(\tilde{E}) \leq P_\phi(E)$. By (1) and (2), \cref{cor:phi min perturbation} implies
    \begin{align*}
        P_\phi(\tilde{E}; U^+) &= P_\phi((E\cup E^-) \cap \Sigma^+; U^+) \leq P_\phi(E \cup E^-; U^+)
    \end{align*}
    and thus $P_\phi(\tilde{E}) \leq P_\phi(E \cup E^-)$ since $E^+ \Subset U^+$. Similarly, by another application of \cref{cor:phi min perturbation}, $P_\phi(E\cup E^-; U^-) = P_\phi(E \cup \Sigma^-; U^-) \leq P_\phi(E; U^-)$. Thus $P_\phi(\tilde{E}) \leq P_\phi(E)$, so $\cF_\tau(\tilde{E}, F) < \cF_\tau(E, F)$, as desired.
\end{proof}
Often we will choose $\Sigma^\pm$ to be corner barriers of the form
\begin{align}
    \label{eq:corner barrier}
    \Sigma^+(x_0,p_i, t) &:= \{(x-x_0)\cdot\nup_i < t\} \cup \{(x-x_0)\cdot\nup_{i+1} < t\} \\
    \Sigma^-(x_0,p_i, t) &:= \{(x-x_0)\cdot\nup_i < t\} \cap \{(x-x_0)\cdot\nup_{i+1} < t\}.
\end{align}

The proof of \cref{thm:MM translate} is broken into the following steps:
\begin{enumerate}
    \item $\partial E$ contains no small ``bubbles". (\cref{lem:simple curve})
    
    \item Every regular edge $I\sb \partial F$ is parallel to a corresponding line segment $I^E\sb \partial E$. (\cref{lem:reg edges translate})

    \item If $I_1, I_2 \sb \partial F$ are adjacent regular edges, then $I_1^E$ and $I_2^E$ are also adjacent. (\cref{prop:fuzzy edges translate}(a))
    
    If $I_1, I_2 \sb \partial F$ are regular edges joined by a fuzzy edge $J$, then $I_1^E, I_2^E$ are joined by a fuzzy edge $J^E \sb \Gamma$ such that $J_E \sb J + \lambda\tau p_J$, with the exception that $I_1^E$ and $I_2^E$ may be adjacent or merge into a single line segment. (\cref{prop:fuzzy edges translate}(b))
\end{enumerate}
\noindent In the author's opinion, the proof of \cref{thm:MM translate} is far easier to intuit from pictures than from text, particularly for \cref{prop:fuzzy edges translate}. We refer the reader to Figures \ref{fig:reg edges translate} and \ref{fig:fuzzy edges translate} for some visualizations of the barriers we construct.

Throughout this section, we fix the neighborhood $U_0 := \{\dd^\phic_F < c_0\tau^{1/2}\}$ where $c_0$ is the constant from the $L^\infty$ estimate \eqref{eq:Linfty estimate cr}, so that $\partial E \sb U_0$ for $\tau\leq \tau_0$. For $\tau_0$ sufficiently small so that $c_0\tau^{1/2} < \eps_0(F)$, $U_0$ is a tubular neighborhood. We also fix the vector field $X:\partial F\to \R^2$ such that $X = p_J$ along any fuzzy edge $J$ and $X$ linearly interpolates between the values at its endpoints along a regular edge. Then $X$ parametrizes $U_0$ in that $\{\sdp_F = t\} = \{x + tX(x): x \in \partial F\}$ for all $|t| < c_0\tau^{1/2}$. Moreover, defining $U(I) := \{x + tX(x): x\in I, |t| < c_0 \tau^{1/2}\} \sb U_0$ for any $I\sb\partial F$, we will repeatedly use the facts that
\begin{itemize}
    \item $\sdp_F(x) = x \cdot \nup_I - s_I$ for $x\in U(I)$ for a regular edge $I\sb\partial F$
    \item $\{\sdp_F = t\} \cap U(J) = J + tp_J$ for a fuzzy edge $J\sb\partial F$ and $|t| < c_0\tau^{1/2}$
\end{itemize} 

\begin{lemma}
    \label{lem:simple curve}
    Let $E,F$ be as in \cref{thm:MM translate}. For each boundary component $\Gamma_0 \sb \partial F$, $\partial E \cap U(\Gamma_0)$ is a simple curve.
\end{lemma}
\begin{proof}
We assume for simplicity that $\partial F$ is connected, for the argument can be straightforwardly adapted to the case of $F$ having several boundary components. Note that there exists some boundary component $\Gamma \sb \partial E$ such that $d_H(\Gamma, \partial F) \leq C\tau^{1/2}$. It suffices to show that $\partial E = \Gamma$. It is straightforward to see from $\phi$-minimality that $P_\phi(\Gamma;U(I)) \geq P_\phi(I) - C\tau^{1/2}$ for any edge $I\sb\partial F$, and thus $P_\phi(\Gamma) \geq P_\phi(F) - C\tau^{1/2}$. 
 
Suppose there exists another boundary component $\Gamma'\sb\partial E$. Since $P_\phi(\Gamma) \leq P_\phi(E) \leq P_\phi(F)$ by the dissipation inequality, it follows that $P_\phi(\Gamma') \leq C \tau^{1/2}$. Then we can obtain a strict energy competitor to $E$ by merely translating $\Gamma'$ closer to $F$, but let us make this precise.

Let $G\sb\R^2$ be bounded such that $\partial G = \Gamma'$, so $\nu_{G} = \sigma\, \nu_E$ $\cH^1$-a.e. along $\Gamma'$ for constant $\sigma\in\{\pm1\}$. Without loss of generality, let $\sigma = 1$. Since $P_\phi(G) \leq C\tau^{1/2}$, Lipschitz $\phi$-regularity of $F$ implies for sufficiently small $\tau_0$ that $G \sb U(\tilde{\Gamma})$ for some $\phi$-minimal $\tilde{\Gamma}\sb\partial F$. In particular, $\tilde{\Gamma}$ satisfies the $p_i$ cone condition for some $i$, so $\{\sdp_F=t\} \cap U(\tilde{\Gamma}) = \tilde{\Gamma} + tp_i$. 

Set $E' := E \cap G$, and consider the perturbation $\tilde{E} := (E \setminus E') \cup (-\eps p_i + E')$. Since the boundary components of $E$ are disjoint, surely $|\tilde{E}| = |E|$ and $P_\phi(\tilde{E}) = P_\phi(E)$ for sufficiently small $\eps$. Moreover,
\begin{align*}
    \cD(\tilde{E}, F) - \cD(E, F) &= \int_{-\eps p_i + E'} \sdp_F dx - \int_{E'} \sdp_F dx = -\eps |E'| <0.
\end{align*}
It follows that $\cF_\tau(\tilde{E}, F) < \cF_\tau(E, F)$, a contradiction.
\end{proof}

\begin{lemma}[Regular edges translate]
    \label{lem:reg edges translate}
    Let $E, F, \tau$ be as in \cref{thm:MM translate}. There exists $C = C(\phi,F)$ such that for every regular edge $I\sb\partial F$, there exists a parallel line segment $I_E \sb \partial E$ such that $\partial E \cap U(\tilde{I}) = I_E \cap U(\tilde{I})$ for some truncation $\tilde{I} \sb I$ such that $|\tilde{I}| \geq |I| - C\tau^{1/2}$. 
\end{lemma}
\begin{proof}
    Let $\nu_I = \nu_1$ for simplicity. We fix a closed segment $\tilde{I} \sb I$ such that for any $x_0\in U(\tilde{I})$, the lines $\{(x-x_0)\cdot \nu_0 =0\}$ and $\{(x-x_0)\cdot \nu_2=0\}$ intersect $\partial U(I)$ only along $\{x\cdot\nup_1 = s_I \pm c_0\tau^{1/2}\} \cap \partial U(I)$. Note that we may find $\tilde{I}$ such that $\cH^1(I) - \cH^1(\tilde{I}) \lesssim_\phi c_0\tau^{1/2}$.

    It suffices to show $\partial E \cap U(\tilde{I}) \sb \{x \cdot \nup_1 = s\}$ for some $s$. Suppose the contrary. Then $s_+ > s_-$ where $s_+ := \max_{x\in\partial E\cap U(\tilde{I})} (x\cdot\nu_1^\phi)$ and $s_- := \min_{x\in\partial E\cap U(\tilde{I})} (x\cdot\nu_1^\phi)$. Moreover, neither of the line segments $I^\pm := \{x\cdot\nu_1^\phi = s_\pm \} \cap U(\tilde{I})$ is a subset of $\partial E$. We aim to reach a contradiction by invoking \cref{lem:area preserving competitor} to construct a strict energy competitor. Let us first describe a perturbation of $E$ along $I^+$ (depicted in \cref{fig:reg edges translate}), breaking into two cases: 

    \begin{figure}
        \centering
        \begin{tikzpicture}
            \node [anchor=south west, inner sep=0] at (0,0) {\includegraphics[width=0.8\linewidth]{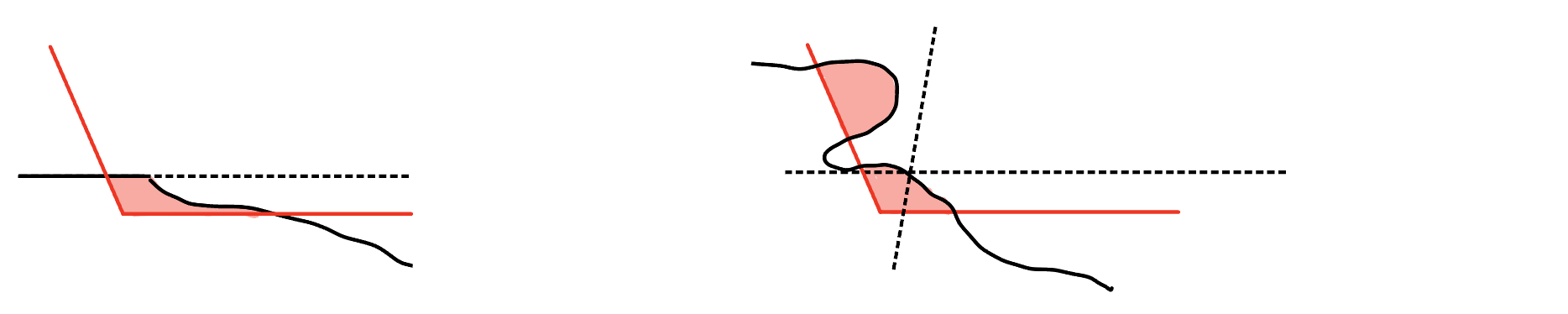}};
            \node [anchor=west] at (3.6,1.25) {$x\cdot\nup_1= s_+$};
            \node [anchor=west] at (11,1.25) {$x\cdot\nup_1= s_+$};
            \node at (1.4,1.4) {$x_+$};
            \filldraw [black] (1.26,1.17) circle (2pt);
            \node at (8.1,1.4) {$x_+$};
            \filldraw [black] (7.68,1.19) circle (2pt);
        \end{tikzpicture}
        \caption{Perturbations in case 1 (left) and case 2 (right) of \cref{lem:reg edges translate}.}
        \label{fig:reg edges translate}
    \end{figure}
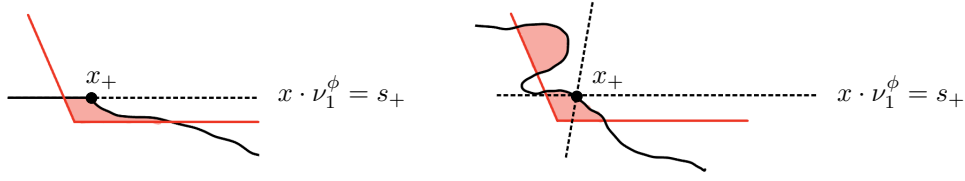
    
    \begin{enumerate}
        \item \ul{$\partial E \cap \inter(I^+) \neq \emptyset$}: Since $I^+\not\sb \partial E$, there must exist some $x_+ \in\partial(\partial E\cap I^+) \cap \inter(I^+)$. By maximality of $s_+$, we have $E\cap B(x_+, r) \sb \{x\cdot\nup_1 \leq s_+\}$ for small $r>0$, so by Lipschitz $\phi$-regularity, $\partial E$ satisfies either $\nu_E \in A[\nu_0,\nu_1]$ or $\nu_E \in A[\nu_1, \nu_2]$ near $x_+$. We assume the former (a symmetric argument works in the latter case), so that $\partial E \cap B(x_+,r)$ satisfies $x\cdot \nu_1^\phi = s_+$ before $x_+$ and $x\cdot \nu_1^\phi < s_+$ after $x_+$. It follows that for sufficiently small $\delta_+ >0$, the region
        \begin{equation}
            E^+ := (E \setminus \Sigma^+(x_+,p_0, -\delta_+)) \cap B(x_+,r)
        \end{equation}
        is compactly contained in $B(x_+, r)$ for a suitable value of $r = o(\delta_+)$. Moreover, $|E^+|$ tends continuously to 0 as $\delta_+\to 0$. 

        \item \ul{$\partial E \cap \inter(I^+) = \emptyset$}: We take $x_+ \in \partial E \cap \partial I^+$ and suppose without loss of generality that $x_+$ is the initial endpoint of $I^+$, so that $[x_+, x_+ + \eps R^{-1}\nu_1] \sb I^+$ for some $\eps>0$. Then for sufficiently small $\delta_+>0$, $E^+ := (E \setminus \Sigma^+(x_+,p_0, -\delta_+))\cap U(I)$ is compactly contained in $U(I)$; the concern is that the portion of $\partial E$ before $x_+$ may exit $\partial U(I)$ along the ceiling $\{x\cdot\nup_1 = s_I + c_0\tau^{1/2}\}$, but this is ruled out by how we defined $\tilde{I}$. We observe that $\lim_{\delta_+ \to 0} |E^+|$ may be positive, but we may make $|E^+|$ tend continuously to 0 by also translating $\Sigma^+$ in the direction $R^{-1}\nu_1$ as $\delta_+\to 0$.
    \end{enumerate}
    Similarly, we may define a region $E^- \sb (\Sigma^-(x_-,p_j,\delta_-) \setminus E) \cap U(I)$ for $j\in\{0,1\}$ such that $|E^-|$ tends continuously to 0 as $\delta_-\to 0$ (with the possible exception of needing to also translate the barrier as in case 2). By construction, we have $E^+ \sb \{\sdp_F \geq s_+ - s_I - \delta_+\}$ and $E^- \sb \{\sdp_F \leq s_- - s_I + \delta_-\}$. Thus we obtain a contradiction by \cref{lem:area preserving competitor}.
\end{proof}

\begin{remark}
    \label{rem:no backtracking}
    A consequence of \cref{lem:reg edges translate} is that $P_\phi(\partial E \setminus I^E; U(I)) = O(\tau^{1/2})$ for any regular edge $I\sb\partial F$, i.e. $\partial E$ cannot backtrack in $U(I)$ by more than $O(\tau^{1/2})$ in length.
\end{remark}

\begin{proposition}[Fuzzy edges translate]
    \label{prop:fuzzy edges translate}
    Let $F, E, \tau$ be as in \cref{thm:MM translate}. Let $I_1, I_2\sb\partial F$ be regular edges, and $I^E_1, I^E_2 \sb \partial E$ be the corresponding regular edges obtained from \cref{lem:reg edges translate}. 
    \begin{enumerate}
        \item[(a)] If $I_1$ and $I_2$ are adjacent, then $I_1^E$ and $I_2^E$ are also adjacent.
        \item[(b)] If $I_1$ and $I_2$ are joined by a fuzzy edge $J$, then one of the following must hold:
        \begin{itemize}
            \item[(i)] $I_1^E$ and $I_2^E$ are joined by a fuzzy edge $J^E \sb \partial E \cap (J + \lambda\tau p_J)$
            \item[(ii)] $I_1^E$ and $I_2^E$ are adjacent, and $J$ is type I
            \item[(iii)] $I_1^E=I_2^E$ (i.e. the edges have merged) and $J$ is type II
        \end{itemize}
    \end{enumerate}
\end{proposition}

\begin{remark}
    Note that case (i) is the expected evolution for a fuzzy edge $J$, while cases (ii) and (iii) correspond to the vanishing of $J$. We also note that (iii) is the only case in which $I_1^E$ is not a regular edge. All of our energy competitors are constructed according to \cref{lem:area preserving competitor}. The only exception is that when $J$ is of type II and expected to vanish, we instead apply a projection argument to ensure that $I_1^E$ and $I_2^E$ do not ``overshoot", i.e. that $s_{I_2^E} - s_{I_1^E}$ and $s_{I_2} - s_{I_1}$ do not have opposite sign.
\end{remark}

\begin{proof}
    We remark that (a) will follow as a special case of the argument for (b) by formally treating adjacent edges $I_1, I_2$ as being joined by a type I fuzzy edge of zero length. Thus we only prove (b).
    
    Let $I_1, I_2 \sb \partial F$ be regular edges joined by a fuzzy edge $J$. Suppose without loss of generality that $I_2$ comes after $I_1$, and $I_1$ is a negative edge with $\nu_{I_1} = \nu_1$, in which case $p_J=p_1$. 
    
    We assume (ii) and (iii) do not hold and will show (i) follows. Then there is a nontrivial closed path $\Gamma \sb \partial E$ which joins $I_1^E$ to $I_2^E$, and let $x_1\in I_1^E, x_2\in I_2^E$ be the endpoints of $\Gamma$. 
    
    We claim it suffices to show $\Gamma \sb \{\sdp_F = t\}$ for some constant $t$. Indeed, if $\Gamma\sb\{\sdp_F = t\}$, then $\Gamma$ satisfies the $p_1$ cone condition, so is a fuzzy edge by the orientations of $I_1, I_2$. Thus $t = \lambda\tau$ by the Euler-Lagrange equation \eqref{eq:EL flat}. Moreover, we must have $x_1, x_2 \in U(J)$. Indeed, if $x_1$ is in the interior of $U(I_1)$, then $\Gamma$ would extend $I_1^E$ near $x_1$, violating the maximality of $I_1^E$. Similarly, if $x_1 \in U(I_2)$, then one violates the maximality of either $I_1^E$ or $I_2^E$ depending on the value of $\nu_{I_2}\in \{\nu_1,\nu_2\}$. Because $\Gamma$ is simple, $\Gamma$ must coincide with the path in $\{\sdp_F = \lambda\tau\} \cap U(J)$ which joins $x_1$ and $x_2$. Thus $\Gamma \sb J + \lambda\tau p_J$, proving (i). 
    
    We split the proof into cases depending on the type of $J$. By \cref{rem:no backtracking}, we may fix a neighborhood $V := U(\tilde{J})$ such that $\Gamma \Subset V$, where $\tilde{J} \sb \partial F$ is an extension of $J$ satisfying $\cH^1(\tilde{J}) \leq \cH^1(J) + C\tau^{1/2}$. We also denote the half-spaces $H_j := \{x\cdot\nup_j \leq s_{I_j^E}\}$ for $j=1,2$. 

    \medskip
    
    \noindent \textbf{Case 1:} $J$ is type I, so $I_2$ is a negative edge with $\nu_{I_2} = \nu_2$. 

        Let $m := \min_{x\in\Gamma} \sdp_F(x)$ and $M := \sup_{x\in\Gamma \setminus (H_1\cup H_2)} \sdp_F(x)$. We claim $m=M$. Suppose on the contrary that $m<M$. Then $\Gamma\setminus \Sigma_m\neq\emptyset$, where for $|t|< c_0\tau^{1/2}$, we define the barrier
        \begin{equation*}
            \Sigma_t := \{\sdp_F \leq t\} \cup H_1 \cup H_2.
        \end{equation*}
        
        Note that the perturbation
        \begin{equation}
            E^+ := (E \setminus \Sigma_{M-\delta_+}) \cap V
        \end{equation}
        is contained in $\{\sdp_F \geq M - \delta_+\}$, and $E^+\Subset V$ since $\Gamma\Subset V$ and $H_1 \cup H_2 \sb \Sigma_{M-\delta_+}$.
        
        If $m < \min\{\sdp_F(x_1), \sdp_F(x_2)\}$, then we define 
        \begin{equation}
            \label{eq:E- type I}
            E^- := (\{\sdp_F \leq m + \delta_-\} \setminus E) \cap V
        \end{equation}
        where $\delta_- > 0$ is such that $m + \delta_- < \min\{\sdp_F(x_1), \sdp_F(x_2)\}$. Note that we have the bound $\sdp_F(x_j) \leq s_{I_j^E} - s_{I_j}$ since $\tilde{J}$ satisfies the $p_1$ cone condition, which implies $E^-\Subset V$ since $\{\sdp_F \leq m + \delta_-\} \cap U(I_j) \sb H_j$.
        
        On the other hand, if $\sdp_F(x_j) = m$ for some $j$, then $\partial E$ must satisfy the $p_1$ cone condition near $x_j$ since $I_j^E \cap V \sb \{\sdp_F \geq m\}$. In this case, we can define for a suitable value of $r = o(\delta_-)$ the perturbation
        \begin{equation}
            E^- := (\Sigma^-(x_j, p_1, \delta_-) \setminus E) \cap B(x_j,r) \Subset V
        \end{equation}
        which still satisfies $E^- \sb \{\sdp_F \leq m + \delta_-\}$.
        
        The barriers $\Sigma_{M-\delta_+}$, $\{\sdp_F \leq m + \delta_-\}$, and $\Sigma^-(x_j,p_1,\delta_-)$ all have $\phi$-minimal boundary in $V$ since they satisfy the $p_1$ cone condition. Moreover $|E^\pm| >0$ and tends continuously to 0 as $\delta_\pm\to 0$, so we obtain a contradiction by \cref{lem:area preserving competitor}. Thus we deduce $M=m$, so $\Gamma \setminus (H_1 \cup H_2) \sb \{\sdp_F = m\}$. 
        
        We further claim that $\sdp_F(x_j)=m$ for $j=1,2$. Indeed, if $\sdp_F(x_1) > m$, then $E \cap B(x_1,r) \sb \{(x-x_1)\cdot\nu_1 \leq 0\}$ for small $r>0$, so $\partial E$ must satisfy the $p_0$ cone condition near $x_0$. Then one obtains a contradiction from \cref{lem:area preserving competitor} with $E^+ := (E \setminus \Sigma^+(x_1,p_0,-\delta_+)) \setminus B(x_1,r)$ and $E^-$ defined as in \eqref{eq:E- type I}.
        
        It follows that $\Gamma$ contains the path in $\{\sdp_F=m\}$ connecting $x_1$ and $x_2$. By simplicity, $\Gamma$ must coincide with this path, so $\Gamma \sb \{\sdp_F=m\}$, as desired. 

    \definecolor{myblue}{RGB}{0,111,255}
    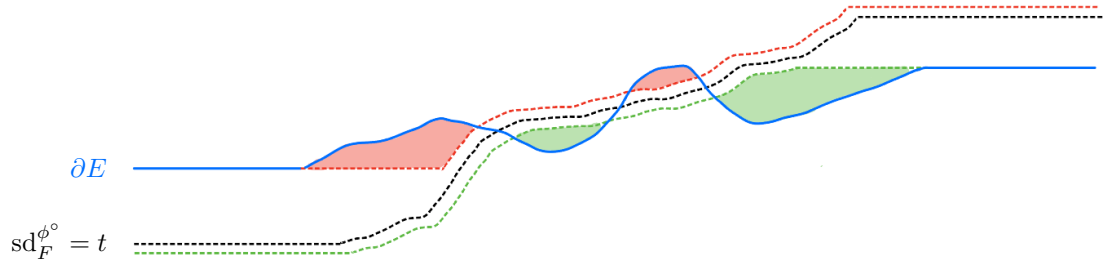
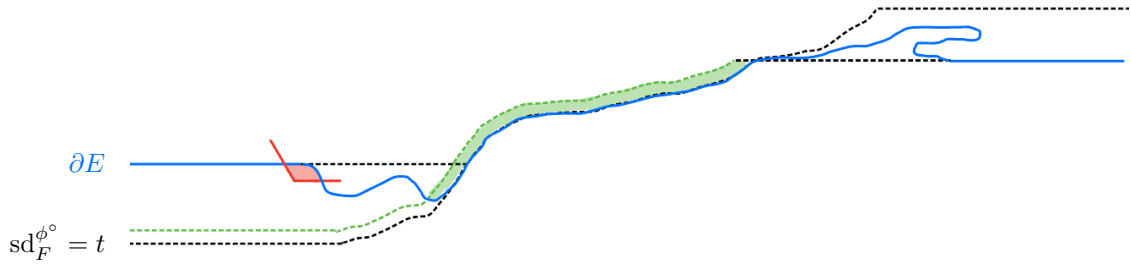
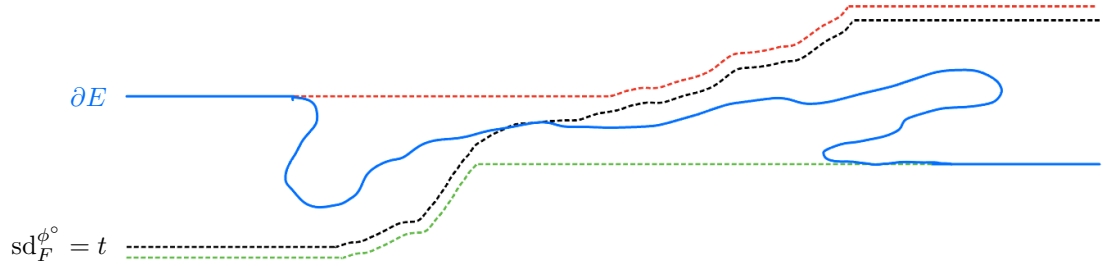
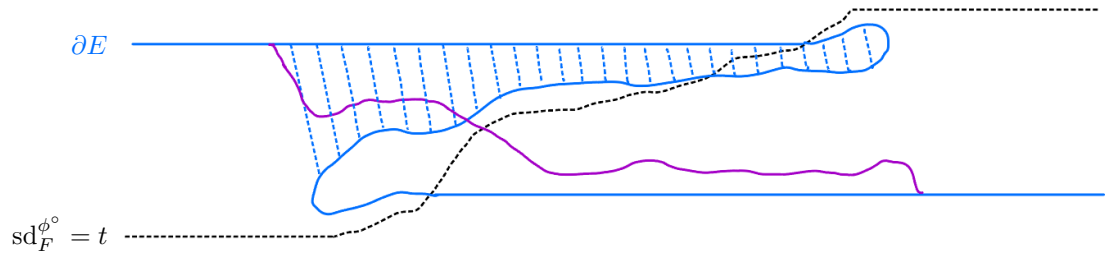
\begin{figure}
        \centering
        \begin{subfigure}{\textwidth}
            \centering
            \begin{tikzpicture}
                \node [anchor=south west, inner sep=0] at (0,0) {\includegraphics[width=0.8\linewidth]{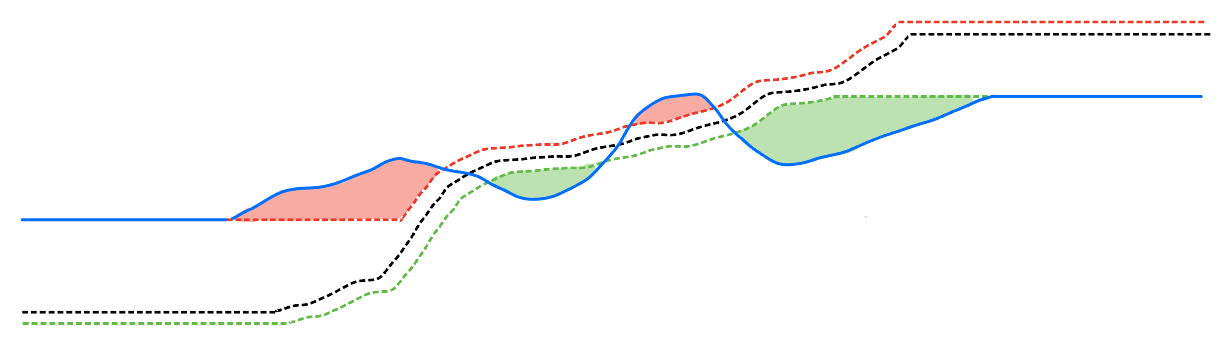}};
                \node [anchor=east] at (0,1.4) {\color{myblue} $\partial E$};
                \node [anchor=east] at (0,0.43) {$\sdp_F =t$};
            \end{tikzpicture}
            \caption{Barriers used to show $m \geq M$ in subcase 2a}
            \label{fig:fuzzy edges translate 2a 1}
        \end{subfigure}
        
        \medskip
        \begin{subfigure}{\textwidth}
            \centering
            \begin{tikzpicture}
                \node [anchor=south west, inner sep=0] at (0,0) {\includegraphics[width=0.8\linewidth]{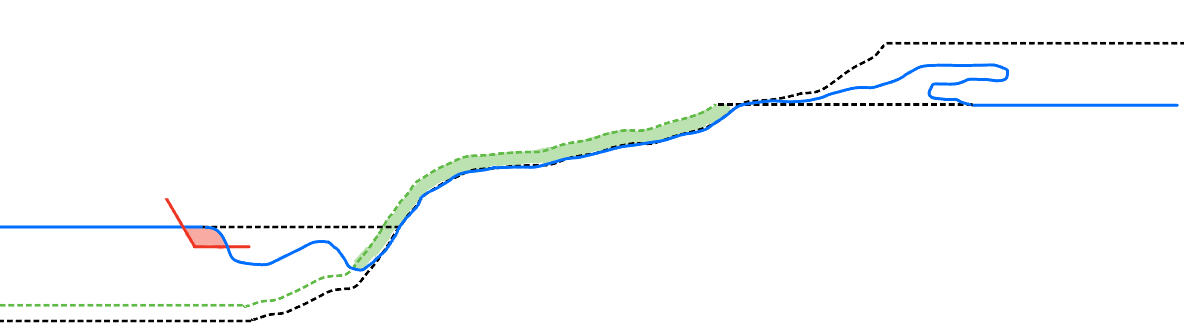}};
                \node [anchor=east] at (-0.2,1.2) {\color{myblue} $\partial E$};
                \node [anchor=east] at (-0.2,0.15) {$\sdp_F =t$};
            \end{tikzpicture}
            \caption{Barriers used to show $\sdp_F(x_1)=m$ in subcase 2a}
            \label{fig:fuzzy edges translate 2a 2}
        \end{subfigure}
        
        \bigskip
        \begin{subfigure}{\textwidth}
            \centering
            \begin{tikzpicture}
                \node [anchor=south west, inner sep=0] at (0,0) {\includegraphics[width=0.8\linewidth]{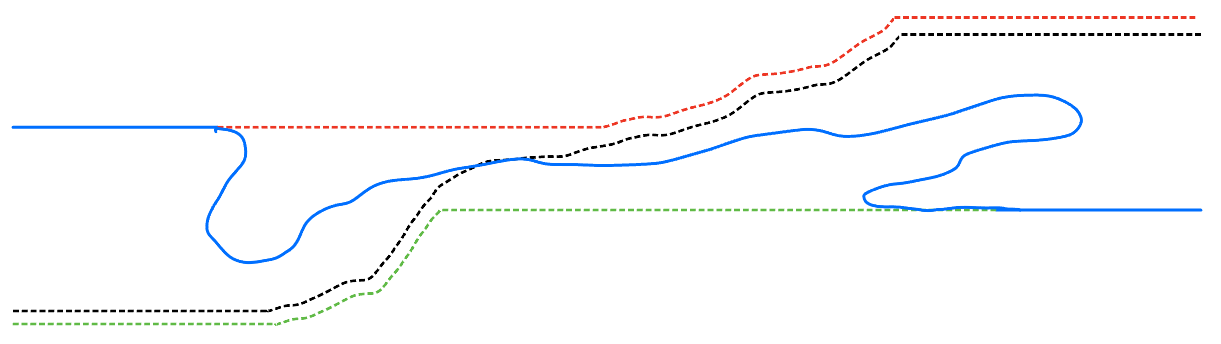}};
                \node [anchor=east] at (0,2.35) {\color{myblue} $\partial E$};
                \node [anchor=east] at (0,0.4) {$\sdp_F =t$};
            \end{tikzpicture}
            \caption{$\Gamma$ is not constrained to a level set of $\sdp_F$ in subcase 2b}
            \label{fig:fuzzy edges translate 2b 1}
        \end{subfigure}

        \bigskip
        \begin{subfigure}{\textwidth}
            \centering
            \begin{tikzpicture}
                \node [anchor=south west, inner sep=0] at (0,0) {\includegraphics[width=0.8\linewidth]{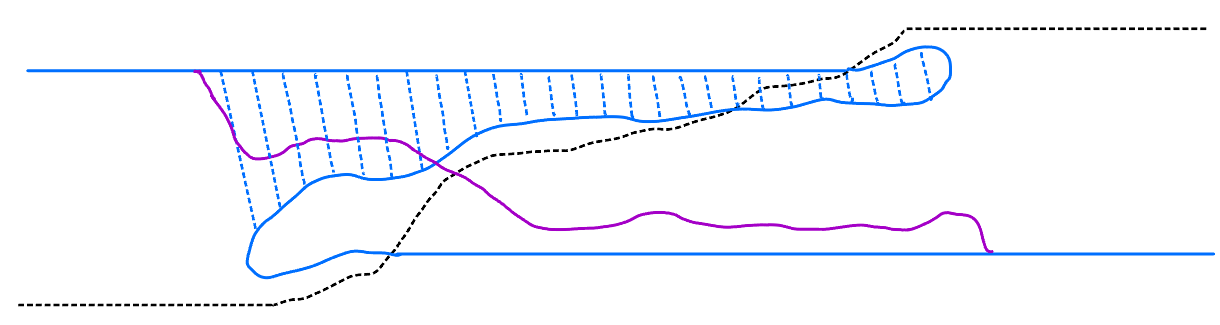}};
                \node [anchor=east] at (0.1,2.75) {\color{myblue} $\partial E$};
                \node [anchor=east] at (0.1,0.25) {$\sdp_F =t$};
            \end{tikzpicture}
            \caption{Energy competitor by localized projection in subcase 2b}
            \label{fig:fuzzy edges translate 2b 2}
        \end{subfigure}
        \caption{Energy competitors in case 2 of the proof of \cref{prop:fuzzy edges translate}}
        \label{fig:fuzzy edges translate}
    \end{figure}

        \medskip
        \noindent \textbf{Case 2:} $J$ is type II, so $I_2$ is a positive edge with $\nu_{I_2} = \nu_1$ and $s_{I_2} > s_{I_1}$. 

        For brevity we denote $h_j := s_{I_j^E} - s_{I_j}$. Note that in this case, $I_1$ and $I_2$ are parallel, $\sdp_F(x_1) \leq h_1$, and $\sdp_F(x_2) \geq h_2$. We define the barriers 
        \begin{align}
            \Sigma^+_t &:= \{\sdp_F \leq t\} \cup H_1 \qquad \text{for } t \geq h_2 \\
            \Sigma^-_t &:= \{\sdp_F \leq t\} \cap H_2 \qquad\text{for } t \leq h_1
        \end{align}
        Note that $\partial\Sigma^\pm_t$ have $\phi$-minimal boundary in $V$. Moreover, for $t\geq \sdp_F(x_2)$, the set $(E\setminus \Sigma^+_t) \cap V$ is compactly contained in $V$ since $\Gamma\Subset V$ and $H_j \cap U(I_j) \sb \Sigma^+_t$. Indeed, the latter inclusion is trivial for $j=1$, and for $j=2$, one has 
        \begin{equation*}
            H_2 \cap U(I_2) = \{\sdp_F \leq h_2\} \cap U(I_2) \sbq \{\sdp_F \leq t\}.
        \end{equation*}
        One similarly sees that $(\Sigma^-_t \setminus E)\cap V \Subset V$ for $t\leq \sdp_F(x_1)$.  

        Similar to Case 1, we consider the quantities $m := \inf_{x\in\Gamma \cap \inter(H_2)} \sdp_F(x)$ and $M := \sup_{x\in\Gamma \setminus H_1} \sdp_F(x)$. 

        \medskip

        \noindent \textbf{Subcase 2a:} $s_{I_1^E} < s_{I_2^E}$.
        
        Note that $m\leq M$ since $\Gamma \cap (\inter(H_2)\setminus H_1)$ is nonempty. Suppose that $m < M$. 

        If $m < h_1$ and $M > h_2$, then for sufficiently small $\delta_\pm>0$, the perturbations 
        \begin{align}
            \label{eq:E+ type II}
            E^+ &:= (E\setminus \Sigma^+_{M-\delta_+}) \cap V\\
            \label{eq:E- type II}
            E^- &:= (\Sigma^-_{m+\delta_-} \setminus E) \cap V
        \end{align}
        are compactly contained in $V$ by the prior discussion; see \cref{fig:fuzzy edges translate 2a 1}.
        
        In the case that $m = h_1$, then the minimality of $m$ implies $m \leq \sdp_F(x_1) \leq h_1$ since $x_1\in \inter(H_2)$ and thus $\sdp_F(x_1) = m$. Moreover $\partial E$ must satisfy the $p_1$ cone condition near $x_1$, so we can replace the definition of \eqref{eq:E- type II} with 
        \begin{equation}
            E^- := (\Sigma^-(x_1, p_1, \delta_-) \setminus E) \cap B(x_1,r) \Subset V
        \end{equation}
        Similarly, if $M = h_2$, then we replace \eqref{eq:E+ type II} with 
        \begin{equation}
            E^+ := (E \setminus \Sigma^+(x_2,p_1,-\delta_+)) \cap B(x_2,r) \Subset V
        \end{equation}

        In all cases, one has $E^+ \sb \{\sdp_F \geq M - \delta_+\}$ and $E^- \sb \{\sdp_F \leq m + \delta_-\}$, so we reach a contradiction by \cref{lem:area preserving competitor}. It follows that $m=M$ and $\Gamma \cap (\inter(H_2)\setminus H_1) \sb \{\sdp_F =m\}$. We further claim $\sdp_F(x_j) =m$ for $j=1,2$. Indeed if $\sdp_F(x_1) > m$, then $\partial E$ satisfies the $p_0$ cone condition near $x_1$, and we reach a contradiction with $E^+ := (E\setminus \Sigma^+(x_1,p_0,-\delta_+)) \setminus B(x_1,r)$ and $E^-$ defined as in \eqref{eq:E- type II}; see \cref{fig:fuzzy edges translate 2a 2}. One similarly reaches a contradiction if $\sdp_F(x_2) > m$. As in Case 1, it follows that $\Gamma$ coincides with the path in $\{\sdp_F =m\}$ connecting $x_1$ and $x_2$.

        \medskip 
        
        \noindent \textbf{Subcase 2b:} $s_{I_1^E} \geq s_{I_2^E}$.

        We expect this subcase to never hold (unless $I_1^E=I_2^E$, which we have assumed does not hold), so we seek a genuine contradiction.

        First we claim $m\geq M$. This inequality is trivial if $\Gamma \cap \inter(H_2)$ or $\Gamma \setminus H_1$ is empty, so assume both are nonempty. Then we may bound
        \begin{equation}
            m \leq \inf_{x\in \Gamma\cap \inter(H_2)} (x\cdot\nup_1 - s_{I_1}) < s_{I_2^E} - s_{I_1} \leq h_1
        \end{equation}
        and similarly $M > h_2$. If $m< M$, then we reach a contradiction by using the barriers $E^\pm$ defined in \eqref{eq:E+ type II} and \eqref{eq:E- type II} as in the previous subcase, hence proving the claim. Unlike the previous subcase, the bound $m\geq M$ does not confine $\Gamma$ to a level set of $\sdp_F$ (see \cref{fig:fuzzy edges translate 2b 1}).
        
        We now take a slight detour and employ a projection argument to constrain the range of possible values of $\nu_E$ in $V$ (see \cref{fig:fuzzy edges translate 2b 2}). Fix a vector $\nu\in\S^1$ such that $\nu\cdot\nu_1 \geq C\tau^{1/2}$ and $\nu\cdot p_1 > 0$; such a vector exists since $\nu_1\cdot p_1 > 0$. We also note that the assumption $s_{I_1^E} > s_{I_2^E}$ forces $0 < s_{I_2} - s_{I_1} < 2c_0\tau^{1/2}$ since $|s_{I_j^E} - s_{I_j}| < c_0\tau^{1/2}$. Then for suitably large $C = C(\phi,F)$, we may find a segment $L \sb \{x\cdot\nup_1 = s_{I_1} - c_0\tau^{1/2}\}$ and $\eps>0$ such that the map $f(x,t) = x + t\nu$ satisfies $V \sb f(L\times(0,\eps)) \cap U_0 \sb U(\tI_1\cup J\cup \tI_2)$, where $\tI_j \sb I_j$ are the truncations from the proof of \cref{lem:reg edges translate}. For $x\in L$, we define the slice $E_x := \{t\in (0,\eps): x + t\nu \in E\}$, and further define the ``projection" $\tilde{E}$ such that $\tilde{E}_x = (0,\cH^1(E_x))$ for all $x\in L$ and $(\tilde{E}\Delta E)\setminus f(L\times(0,\eps)) = \emptyset$. Then $P_\phi(\tilde{E}) \leq P_\phi(E)$ by \cref{lem:projection}, and $|\tilde{E}| = |E|$ by Fubini's theorem and the fact that $\cH^1(\tilde{E}_x) = \cH^1(E_x)$ for $x\in L$. Moreover, by $\nu\cdot p_1 > 0$ and the minimality of $E$, we can bound 
        \begin{align}
            \label{eq:projection dissipation comparison}
            0 \leq \cD(\tilde{E},F) - \cD(E,F) &= \int_{L} \ps {\int_{\tilde{E}_x} \sdp_F(x+t\nu)dt - \int_{E_x}\sdp_F(x+t\nu)dt} \nu\cdot\nu_1\, dx \notag \\
            &= \int_L \ps{ \int_{\tilde{E}_x \setminus E_x} \sdp_F(x+t\nu) dt - \int_{E_x\setminus \tilde{E}_x} \sdp_F(x+t\nu)dt} \nu\cdot\nu_1\, dx \notag \\
            &\leq 0 
        \end{align}
        Thus $E_x = \tilde{E}_x$ for almost every $x\in L$, so $\nu_E \cdot \nu \geq 0$. Applying this argument for all choices of $\nu$ satisfying $\nu\cdot\nu_1\geq C\tau^{1/2}$ and $\nu\cdot p_1 >0$, we deduce $\arg \nu_E \in [\arg\nu_1 - C\tau^{1/2}, \arg p_1]$ or $\arg\nu_E \in [\arg p_1, \arg\nu_1 + C\tau^{1/2}]$. In particular, $\arg\nu_E \geq \min\{\arg p_1, \arg\nu_1 - C\tau^{1/2}\}$.
        
        Let us now return to the task of confining $\Gamma$. We claim that $\sdp_F(x_1) \leq M$ and $\sdp_F(x_2) \geq m$. Suppose $\sdp_F(x_1) > M$, in which case $\partial E$ satisfies the $p_0$ condition at $x_1$. Let $x_3\in\Gamma$ be the first point after $x_1$ such that $\nu_E=\nu_1$ for a short duration after $x_3$; such a point must exist, else $\Gamma$ will never connect to $I_2^E$. Then $\partial E$ also satisfies the $p_0$ condition at $x_3$, and by the previously obtained lower bound on $\arg\nu_E$, $x_3 \in \Sigma^-(x_1,p_1,0)$. In particular, $\sdp_F(x_3) < \sdp_F(x_1)$, so one obtains a contradiction with the barriers $\Sigma^+(x_1,p_0,-\delta_+)$ and $\Sigma^-(x_3, p_0,\delta_-)$. Thus $\sdp_F(x_1) \leq M$, and one similarly deduces $\sdp_F(x_2) \geq m$. 
        
        Altogether, we have shown $\sdp_F(x_1) \leq M \leq m \leq \sdp_F(x_2)$. If $s_{I_1^E} = s_{I_2^E}$, then $I_1^E$ and $I_2^E$ must intersect, contradicting our initial assumption that $I_1^E \neq I_2^E$. If $s_{I_1^E} > s_{I_2^E}$, then the $p_1$ cone condition of $\tilde{J}$ implies $(x_1-x_2)\cdot R^{-1}\nu_1 > 0$, so in particular the family of slices $(x+\R\nu_1)\cap E\cap U_0$ ranging from $x_1$ to $x_2$ are not intervals, contradicting the conclusion of our projection argument for $\nu=\nu_1$. This concludes the proof of Case 2.
\end{proof}

\subsection{Convergence to the ODE flow} 

Now that we know the approximate flow typically produces translates, we are ready to show that these translates are a close approximation to \eqref{eq:ODE flow}. We need only establish such estimates under the assumption that no vanishing occurs, since vanishing can only occur at a bounded number of timesteps.

\begin{lemma}
    Let $E_0,F,\tau_0$ be as in \cref{thm:MM translate} and suppose $|F| = |E_0|$. For $|t| < c_0 \tau^{1/2}$ and $\tau \leq \tau_0$, $F_t := \{\sdp_F \leq t\}$ satisfies
    \begin{equation}
        \label{eq:energy translate}
        \cF_\tau(F_t, F) = P_\phi(F) \ps{1 + \frac{|t|}{\tau^{1/2}} + \frac{t^2}{2\tau}} + O(\tau^{1/2}).
    \end{equation}
    where the implicit constant depends only on $\phi$ and $\chi(F)$. 
\end{lemma}
\begin{proof}
    We first note that 
    \begin{equation}
        \label{eq:perimeter Ft}
        P_\phi(F_t) = P_\phi(F) + \sum_j \alpha_{I_j} t = P_\phi(F) + t\int_{\partial F} \kp_F\,dP_\phi =  P_\phi(F)[1 + \okp_F t].
    \end{equation}
    By \cref{lem:eikonal} and the anisotropic coarea formula, we can compute the area and dissipation terms:
    \begin{align}\
        \label{eq:area Ft}
        |F_t| &= \int_{F_t} \phi(\nabla \sdp_F) dx
        = |F| + \int_0^{t} P_\phi(F_s) ds
        = |F| + P_\phi(F) (t + \inv{2} \okp_F t^2)\\
        \label{eq:dissipation Ft}
        \cD(F_t, F) &= \int_{F_t} \sdp_F dx - \int_F \sdp_F dx 
        = \int_0^t s P_\phi(F_s) ds
        = P_\phi(F) (\frac{t^2}{2} + \inv{3} \okp_F t^3). 
    \end{align}
    By summing \eqref{eq:perimeter Ft}, \eqref{eq:area Ft}, \eqref{eq:dissipation Ft}, we obtain the desired asymptotic
    \begin{align*}
        \cF_\tau(F_t, F) &= P_\phi(F)\ps{
            1 + \okp_F t + \inv{\tau}\Big(\frac{t^2}{2} + \frac{\okp_F}{3} t^3\Big) + \inv{\tau^{1/2}}\Big|t + \frac{\okp_F}{2}t^2\Big|
        }\\
        &= P_\phi(F) \ps{1 + \frac{t^2}{2\tau} + \frac{|t|}{\tau^{1/2}} } + O(\tau^{1/2})
    \end{align*}
    and the implicit constant depends only on $c_0$ and $P_\phi(F)\okp_F = P_\phi(W_\phi)\chi(F)$.
\end{proof}

\begin{proposition}
    \label{prop:lagrange mult est}
    Let $E_0, E, F$ be as in \cref{thm:MM translate}. Suppose $|F| = |E_0|$ and that no type II fuzzy edges have vanished in $E$. Then $|E| = |E_0|$ and $|\lambda - \okp_F| \leq C\tau$ where $C$ depends only on $\phi, |E_0|, P_\phi(F), \chi(F), \eps_0(F)$.
\end{proposition}

\begin{proof}
    All implicit constants depend only on the quantities mentioned in the statement. 

    First we claim that $s_{I^E} = s_I + \lambda\tau + O(\tau)$ for every regular edge $I^E\sb\partial E$. Suppose without loss of generality that $I^E$ is negative. Then $\sdp_F(x) \leq s_{I^E} - s_I$ for all $x\in I^E$ due to the cone conditions satisfied by $I$'s neighbors, so the Euler-Lagrange equation \eqref{eq:EL flat} implies
    \begin{equation*}
        s_{I^E} - s_I \geq \dashint_{I^E} \sdp_F d\cH^1 = (-\kp_E(I^E) + \lambda)\tau > \lambda\tau.
    \end{equation*}
    Moreover, \cref{prop:fuzzy edges translate} implies $I^E \sb \{\sdp_F \geq \lambda\tau\}$. Since $I^E \cap U(\tI) \sb \{\sdp_F = s_{I^E} - s_I \}$, where $\tI$ is the truncation of $I$ from \cref{lem:reg edges translate}, we obtain the opposite bound
    \begin{align*}
       -\alpha_I \tau = \int_{I^E} (\sdp_F -\lambda \tau) d\cH^1 
       &\geq \int_{I^E \cap U(\tI)} (\sdp_F -\lambda\tau) d\cH^1
       = (s_{I^E} - s_I - \lambda\tau)[|I| + O(\tau^{1/2})]
    \end{align*}
    from which the claim follows. 
    
    By \cref{prop:fuzzy edges translate} and the hypothesis, the fuzzy edges of $\partial E$ live in $\{\sdp_F = \lambda\tau\}$, so the previous claim implies the estimates \begin{align*}
        |P_\phi(E) - P_\phi(F_{\lambda\tau})| &= \bigg|\sum_{j=1}^m \alpha_{I_j} (s_{I_j^E} - s_{I_j} - \lambda\tau)\bigg| \leq C\tau\\
        |E \Delta F_{\lambda\tau}| &\leq C\tau \\
        |\cD(E, F) - \cD(F_{\lambda\tau}, F)| &\leq \|\sdp_F\|_{L^\infty(E\Delta F_{\lambda\tau})} |E\Delta F_{\lambda\tau}| \leq C\tau^{3/2}
    \end{align*}
    and thus by invoking \eqref{eq:energy translate},
    \begin{align}
        \label{eq:energy estimate for lambda tau}
        \cF_\tau(E, F) &= \cF_\tau(F_{\lambda\tau}, F) + O(\tau^{1/2}) \notag \\
        &= P_\phi(F)\ps{1 + |\lambda|\tau^{1/2} + \inv{2}\lambda^2\tau} + O(\tau^{1/2}).
    \end{align}
    Since $\cF_\tau(E, F) \leq P_\phi(F)$ by minimality, \eqref{eq:energy estimate for lambda tau} implies $\lambda = O(1)$, and thus $s_{I^E} - s_I = O(\tau)$ and $P_\phi(E) = P_\phi(F) + O(\tau)$. Moreover, since $|\lambda| < \tau^{-1/2}$ for sufficiently small $\tau$, the area constraint $|E| = |E_0|$ must be satisfied. We also note that the truncations $\tilde{I}\sb I$ from \cref{lem:reg edges translate} can now be chosen to satisfy $|\tI| \geq |I| - C\tau$. 
    
    It remains to estimate $\lambda = \okp_F + O(\tau)$. To this end, we claim a more precise area estimate:
        \begin{equation}
            \label{eq:area diff MM}
            |E| = |F| + \int_{\partial E} \sdp_F dP_\phi + O(\tau^2).
        \end{equation}
    Indeed, by the Euler-Lagrange equation and the fact that $\int_{\partial E} \kp_E dP_\phi = P_\phi(W_\phi)\chi(E)$ is a topological invariant, \eqref{eq:area diff MM} implies
        \begin{align*}
            0 = |E| - |F| &= \tau \int_{\partial E} (\lambda - \kp_E)dP_\phi + O(\tau^2) \\
            &= P_\phi(E) \lambda\tau - \tau \int_{\partial F} \kp_F dP_\phi + O(\tau^2) \\
            &= P_\phi(F) (\lambda - \okp_F)\tau + O(\tau^2)
        \end{align*}
    and thus $\lambda - \okp_F = O(\tau)$.

    Let $I\sb\partial F$ be a regular edge. We remark by \cref{prop:fuzzy edges translate} that $I^E$ either extends outside of $U(I)$ or has endpoints with distance at most $O(\tau)$ from the boundary of $U(I)$. In either case, $E\cap U(I)$ differs in area from $\{x\cdot\nup_I \leq s_{I^E}\}\cap U(I)$ by at most $O(\tau^2)$, so
    \begin{align}
        \label{eq:area diff I}
        \int_{U(I)} (\chi_E - \chi_F) &= |\{x\cdot\nup_{I^E}\leq s_{I^E}\}\cap U(I)| - |\{x\cdot\nup_I\leq s_I\}\cap U(I)| + O(\tau^2) \notag \\
        &= (s_{I^E} - s_I) \phi(\nu_I) |I| + O(\tau^2) \notag\\
        &= \int_{I^E \cap U(I)}\sdp_F dP_\phi + O(\tau^2)\notag\\
        &= \int_{\partial E \cap U(I)} \sdp_F dP_\phi + O(\tau^2)
    \end{align}
    where the last step follows from the fact that $\cH^1((\partial E \setminus I^E)\cap U(I)) = O(\tau)$. 
    
    If $J\sb\partial F$ is a fuzzy edge, then we may decompose $J$ into $J_1\cup \tilde{J}\cup J_2$ where $J^E = \tilde{J} + \lambda\tau p_J$. Let $I_j$ be the regular edges adjacent to $J_j$. Recall that for any $x \in \partial E \cap U(J)$, $x - \sdp_F(x)p_J \in J$. By slicing along $p_J$ and using the identity $\nu_{I_j} \cdot p_J = \phi(\nu_{I_j})$, we have the exact formula
    \begin{align}
        \label{eq:area diff J 1}
        \int_{U(J_j)} (\chi_E - \chi_F) &= \int_{I_j^E \cap U(J)} \ps{\int_{(x+\R p_J)\cap U(J)}\chi_E - \chi_F} \nu_{I_j}\cdot \frac{p_J}{|p_J|} d\cH^1(x) \notag\\
        &= \int_{I_j^E \cap U(J)} \sdp_F(x) (\nu_{I_j}\cdot p_J) d\cH^1(x) \notag\\
        &= \int_{I_j^E \cap U(J)} \sdp_F dP_\phi.
    \end{align}
    For the contribution over $\tilde{J}$, we recall that $J^E = \tilde{J} + \lambda\tau p_J$ and apply the anisotropic coarea formula \eqref{eq:sdf coarea}
    \begin{align}
        \label{eq:area diff J 2}
        \int_{U(\tilde{J})} (\chi_E - \chi_F) &= \int_0^{\lambda\tau} P_\phi(\{\sdp_F \leq t\}; U(\tilde{J})) dt = \lambda\tau P_\phi(\tilde{J}) = \int_{J^E} \sdp_F dP_\phi.
    \end{align}
    By \eqref{eq:area diff J 1} and \eqref{eq:area diff J 2} we obtain 
    \begin{equation}
        \label{eq:area diff J}
        \int_{U(J)} (\chi_E - \chi_F) = \int_{\partial E \cap U(J)} \sdp_F dP_\phi
    \end{equation}
    and thus \eqref{eq:area diff MM} follows from summing \eqref{eq:area diff I} and \eqref{eq:area diff J} over all edges in $F$. 
\end{proof}

\begin{proposition}
    \label{prop:MM finite diff}
    Let $E, F, \tau$ be as in \cref{thm:MM translate} and suppose $|F| = |E_0|$. Let $u(t) := \Psi_t^F(0)$.
    \begin{itemize}
        \item[(i)] If $E$ is a translate $F(u^{\tau})$, then $u^\tau = u(\tau) + O(\tau^2)$.
        \item[(ii)] There exists a constant $C'$ such that if $u(t)$ is defined for all $t\leq C'\tau^{1/2}$, then no type II fuzzy edges vanish in $E$.
    \end{itemize}
\end{proposition}
\begin{remark}
    We note that it is not immediate that one should expect quadratic local error in (i). For instance, the Euler scheme will generally fail to approximate $u(t)$ up to quadratic error in the presence of fuzzy edges, since then the vector field is only BV. To obtain quadratic error, we rely on the fact that the Euler-Lagrange equation encodes the local area difference between $E$ and $F_{\lambda\tau}$. The estimates are similar to those in \cref{prop:ODE stability hausdorff}.
\end{remark}

\begin{proof}
    Let $I_1, \dots, I_m$ be the regular edges of $F$. First we prove (ii). Suppose $J\sb\partial F$ is a type II fuzzy edge, and say $J$ is adjacent to $I_1$ and $I_2$ which are respectively negative and positive. Recall that the vanishing of $J$ occurs when $u_1(t) - u_2(t)$ increases to $s_{I_2} - s_{I_1}$. Since $\frac{d}{dt}(u_1 - u_2) \geq \inv{C}$, we can bound $s_{I_2} - s_{I_1} \geq u_1(t) - u_2(t) \geq t/C$ for any $t\leq C'\tau^{1/2}$. For $C'$ sufficiently large, it follows that $s_{I_2} - s_{I_1} \geq \frac{C'}{C}\tau^{1/2} \geq 2 c_0\tau^{1/2}$. Since $|s_{I_j^E} - s_{I_j}| < c_0\tau^{1/2}$, we must have $s_{I_1^E} < s_{I_2^E}$, i.e. $J$ cannot vanish in $E$.
    
    Now we prove (i). By \cref{thm:MM translate} we can express $E = F(u^\tau)$ where $u^\tau = (\lambda\tau, u_1^\tau, \dots, u_m^\tau)$. Moreover $u_j^\tau = O(\tau)$ by the proof of \cref{prop:lagrange mult est}, and $u(\tau) = O(\tau)$ since the vector field $Y(u)$ is bounded. 

    Recall the function $\ell_j(s)$ which is the length of the translation of $I_j$ in $F(se_j)$. We claim that 
    \begin{equation}
        \label{eq:MM finite diff}
        \int_0^{u_j^\tau} \ell_j(s) ds = -\alpha_{I_j}\tau + O(\tau^2).
    \end{equation}
    Define $\tilde{U}(I_j)$ to be the union of $U(I_j)$ with $U(J)$ for any fuzzy edges $J$ adjacent to $I_j$. Following the same calculations as in \eqref{eq:area diff I} and \eqref{eq:area diff J 1} and invoking the Euler-Lagrange equation, one obtains the estimate
    \begin{align*}
        \inv{\phi(\nu_{I_j})} \int_{\tilde{U}(I_j)} (\chi_E - \chi_{F_{\lambda\tau}}) &= \int_{I_j^E} (\sdp_F - \lambda\tau) d\cH^1 + O(\tau^2) \notag\\
        &= -\kp_E(I_j^E) |I_j^E| \tau + O(\tau^2) \notag\\
        &= -\alpha_{I_j}\tau + O(\tau^2).
    \end{align*}
    By slicing along lines parallel to $I_j$, we can also express 
    \begin{align*}
        \inv{\phi(\nu_{I_j})} \int_{\tilde{U}(I_j)} (\chi_E - \chi_{F_{\lambda\tau}}) &= \int_0^{u_j^\tau} |(E\Delta F_{\lambda\tau}) \cap \tilde{U}(I_j) \cap \{x\cdot\nup_{I_j} = s_I + \lambda\tau + s\}| ds \\
        &= \int_0^{u_j^\tau} (\ell_j(s) + O(|u^\tau|)) ds\\
        &= \int_0^{u_j^\tau} \ell_j(s) ds + O(\tau^2)
    \end{align*}
    from which \eqref{eq:MM finite diff} follows. 
    
    Meanwhile for the ODE flow, let $I_j(t)$ denote the edges in the evolution $F(u(t))$ and define $\tau_j(s)$ so that $u_j(\tau_j(s)) = s$. Recall from the proof of \cref{prop:ODE stability hausdorff} that $|I_j(\tau_j(s))| = \ell_j(s) + O(\tau)$ and 
    \begin{equation}
        \label{eq:ODE length integral 3}
        \int_0^{u_j(t)} |I_j(\tau_j(s))| ds = -\alpha_{I_j} t.
    \end{equation}
    By \eqref{eq:MM finite diff} and \eqref{eq:ODE length integral 3}, it follows that 
    \begin{align*}
        0 &= \int_0^{u_j^\tau} \ell_j(s) ds - \int_0^{u_j(\tau)} |I_j(\tau_j(s))| ds + O(\tau^2)\\
        &= \int_{u_j(\tau)}^{u_j^\tau} \ell_j(s) ds + O(\tau^2)
    \end{align*}
    and thus $u_j^\tau - u_j(\tau) = O(\tau^2)$ since $\ell_j(s) \geq |I_j|$.
    
    It remains to check $\ol{u}^\tau = \ol{u}(\tau) + O(\tau^2)$. By \cref{prop:lagrange mult est}, $\ol{u}^\tau = \lambda\tau = \okp_F \tau + O(\tau^2)$. Moreover, since $P_\phi(F(u))$ is affine in $u$,
    \begin{equation*}
        \ol{u}(\tau) = \int_0^\tau \okp_{F(u(t))} dt = \int_0^\tau \okp_F \frac{P_\phi(F)}{P_\phi(F(u(t))}dt =  \int_0^\tau \okp_F(1+ O(t))dt = \okp_F \tau + O(\tau^2). \qedhere
    \end{equation*}
\end{proof}

\bigskip 

\noindent\textbf{Proof of \cref{thm:weak strong uniqueness}:}

    Fix $\tau \leq \tau_0$ and $T_0 < T$, and let $E^\tau_{k+1} \in \argmin \cF_\tau(\cdot, E^\tau_k)$ where $\Et_0 := E_0$. We first assume there is no vanishing in $\Psi_t(E_0)$. Let $u(t) := \Psi_t^{E_0}(0)$. By iterating Propositions \ref{prop:lagrange mult est} and \ref{prop:MM finite diff}, we have $|\Et_k| = |E_0|$ and may express $E^\tau_k = E_0(u_k^\tau)$ such that $u_{k+1}^\tau - u_k^\tau = \Psi_\tau^{\Et_k}(0) + O(\tau^2)$ and thus $|u_{k+1}^\tau - \Psi^{E_0}_\tau(u_k^\tau)| \leq C\tau^2$. By Lipschitz stability of $\Psi^{E_0}_\tau$ from \eqref{eq:ODE flow stability}, a standard numerical analysis argument yields the estimate
    \begin{equation}
        |u_k^\tau - u(k\tau)| \leq e^{Ck\tau}\tau.
    \end{equation}
    Defining $u^\tau(t) := u^\tau_{\floor{t/\tau}}$ for $t\geq0$, we obtain \eqref{eq:MM ODE approx}:
    \begin{equation}
        \label{eq:MM ODE approx tau}
        d_H(\Et(t), \Psi_t(E_0)) \leq C|u^\tau(t) - u(t)| \leq Ce^{Ct}\tau \qquad\forall t\leq T_0.
    \end{equation}

    Now suppose a type II fuzzy edge $J$ vanishes in $\Psi_t(E_0)$ at time $t=t_*$. By \cref{prop:MM finite diff}(ii), there exist $t_1 < t_* < t_2$ such that $t_2 - t_1 \leq C\tau^{1/2}$, \eqref{eq:MM ODE approx tau} holds for $t\leq t_1$, and so that $J$ has not yet vanished in $\Et(t_1)$ and has vanished in $\Et(t_2)$. Note that $d_H(\Psi_{t_1}(E_0), \Psi_{t_2}(E_0)) \leq C(t_2-t_1) \leq C\tau^{1/2}$ since the flow map is Lipschitz in time.
    
    Let $k_1 \in [\lfloor t_1/\tau\rfloor, \lfloor t_2/\tau\rfloor)$ be the first index such that $J$ has vanished in $\Et_{k_1+1}$, and let $k_2$ be the first index after $k_1$ satisfying $|\Et_{k_2}| = |E_0|$. By \cref{prop:lambda L2 bound}(iii) and \cref{prop:lagrange mult est}, we have $k_2 - k_1 \leq C$ and $|\Et_k| = |E_0|$ for $k\geq k_2$. In particular, it makes no harm to assume $k_2 \leq \lfloor t_2/\tau\rfloor$ by increasing $t_2$ by at most $O(\tau)$. By \cref{prop:MM finite diff}(i), we have $d_H(\Et_{k+1}, \Et_k) \leq C\tau$ for all $k\in [\lfloor t_1/\tau\rfloor, \lfloor t_2/\tau\rfloor) \setminus [k_1,k_2)$, while for $k\in [k_1,k_2)$, we have the weaker bound $d_H(\Et_{k+1},\Et_k) \leq C\tau^{1/2}$ from the $L^\infty$ estimate \eqref{eq:Linfty estimate cr}. By the triangle inequality, it follows that 
    \begin{equation*}
        d_H(\Et(t_1), \Et(t_2)) \leq (k_2-k_1)C\tau^{1/2} + \frac{t_2-t_1}{\tau}C\tau \leq C\tau^{1/2}.
    \end{equation*}
    
    By applying \eqref{eq:MM ODE approx tau} for $t=t_1$ and recalling that $\Psi_t$ is Lipschitz in time, 
    \begin{align}
        \label{eq:MM ODE t1 to t2}
        d_H(\Et(t_2), \Psi_{t_2}(E_0)) &\leq d_H(\Et(t_1), \Psi_{t_1}(E_0)) + d_H(\Et(t_1), \Et(t_2)) + d_H(\Psi_{t_1}(E_0), \Psi_{t_2}(E_0)) \notag \\
        &\leq Ce^{Ct_1}\tau + C\tau^{1/2} \leq Ce^{Ct_1}\tau^{1/2}.
    \end{align}
    Now suppose no vanishing occurs for $\Et(t)$ or $\Psi_t(E_0)$ for $t\in[t_2,T_0]$. By repeating the derivation for \eqref{eq:MM ODE approx tau}, one obtains 
    \begin{equation}
        \label{eq:MM ODE past t2}
        d_H(\Et(t), \Psi_{t-t_2}(\Et(t_2))) \leq Ce^{C(t-t_2)}\tau \qquad \forall t\in [t_2, T_0].
    \end{equation}
    Moreover, $\Et(t_2)$ and $\Psi_{t_2}(E_0)$ are rough translates, so for $\tau_0$ sufficiently small, \cref{cor:ODE stability hausdorff 2} implies
    \begin{equation}
        \label{eq:MM ODE stability past t2}
        d_H(\Psi_{t-t_2}(\Et(t_2)), \Psi_t(E_0)) \leq Ce^{C(t-t_2)} d_H(\Et(t_2), \Psi_{t_2}(E_0)) \leq Ce^{Ct}\tau^{1/2} \qquad\forall t\in[t_2,T_0]
    \end{equation}
    Then \eqref{eq:MM ODE approx} once again follows from \eqref{eq:MM ODE past t2} and \eqref{eq:MM ODE stability past t2}. 
    
    In the case of several vanishing times, one can obtain \eqref{eq:MM ODE approx} by iterating the previous argument. Then for all $t\leq T_0$, $\Et(t)$ converges in Hausdorff distance to $\Psi_t(E_0)$; the convergence is surely also in $L^1$, so $E(t) = \Psi_t(E_0)$. Since $T_0$ is arbitrary, we obtain the first statement. \hfill $\Box$

\section{Long-time convergence of flat flow}
\label{sec:long time}

Here we prove \cref{thm:convergence flat flow cr}, whose argument is similar to those found in \cite{Julin2022, KimKwon2024}. The main difference is a crystalline quantitative Alexandrov theorem, whose proof is in fact simpler than its counterparts in \cite{Julin2022, KimKwon2024} due to the existence of negative edges being ruled out by a small $L^2$ deviation in curvature. A second key difference is the eventual regularity statement, which follows from the stability of the ODE flow with respect to Hausdorff distance. 

\begin{proposition}[Crystalline QAT]
    \label{prop:QAT}
    Fix $m,M>0$. There exist constants $\eps, C >0$ depending only on $\phi,m,M$ such that for any Lipschitz $\phi$-regular set $E\sb\R^2$ satisfying $P_\phi(E)\leq M$, $|E|=m$ and $\|\kappa^\phi_E - \ol{\kappa}^\phi_E\|_{L^2(\partial E)} \leq \eps$, there exists a disjoint union $F$ of equally sized Wulff shapes such that $|E| = |F|$ and
    \begin{align*}
        \label{eq:QAT}
        d_H(E,F) &\leq C\|\kappa^\phi_E - \ol{\kappa}^\phi_E\|_{L^2(\partial E)}\\
        |P_\phi(E) - P_\phi(F)| &\leq C\|\kappa^\phi_E - \ol{\kappa}^\phi_E\|_{L^2(\partial E)}^2.
    \end{align*} 
    Moreover, the edges of $E$ consist of (1) positive edges which are translations of those of $F$ and (2) possibly fuzzy edges whose collective length is at most $C\|\kp_E - \okp_E\|_{L^2}$.
\end{proposition}
\begin{proof}
    Let $C = C(\phi,m,M)$ be a constant which may change line by line. For brevity we denote $\eps := \|\kappa^\phi_E - \ol{\kappa}^\phi_E\|_{L^2}$. By \cref{lem:gauss bonnet} and following the argument in \cite[Prop 2.1]{Julin2022} or \cite[App A]{KimKwon2024},
    we may take $\eps$ sufficiently small so that $E$ has connected components $E_1, \dots, E_d$ which are all simply connected such that $d\leq C$ and $1/C \leq \ol{\kappa}^\phi_E \leq C$.

    Fix a component $E_k$, let $I_1,\dots, I_m$ be all of the regular edges in $\partial E_k$, and let $\nu_{i_j} := \nu_E(I_j)$. By expanding
    \begin{align}
        \eps^2 = \int_{\Gamma} |\kappa^\phi_E - \ol{\kappa}^\phi_E|^2 d\cH^1 &= \sum_{\sigma_{I_j}=1} \ps{\frac{L_{i_j}}{|I_j|} - \ol{\kappa}^\phi_E}^2 |I_j| + \sum_{\sigma_{I_j}=-1} \ps{\frac{L_{i_j}}{|I_j|} + \ol{\kappa}^\phi_E}^2 |I_j| + (\ol{\kappa}^\phi_E)^2 \cH^1(\partial E\cap \{\kappa^\phi_E=0\})),
    \end{align}
    we deduce the following:
    \begin{enumerate}
        \item $\cH^1(\partial E\cap \{\kappa^\phi_E=0\})) \leq C\eps^2$ 

        \item For $\eps$ sufficiently small, $\partial E$ has no edges of negative curvature. Indeed, for any $I_j$ such that $\sigma_{I_j} = -1$, 
        \[ \frac{L_{i_j}^2}{|I_j|} \leq \ps{\frac{L_{i_j}}{|I_j|} + \ol{\kappa}^\phi_E}^2 |I_j| \leq \eps^2 \qquad \implies \qquad \eps^2 \geq \frac{\min_{\nu_i\in\cN_\phi}L_i^2}{P(E)} \] from which we obtain a contradiction if $\eps$ is small.

        \item Since $\partial E$ contains only positive edges or type I fuzzy edges, we may reindex in such a way that $\nu_E(I_j) = \nu_j$. For each $j$, we have \begin{equation}
            \label{eq:QAT edge length}
            \abs{|I_j| - (\ol{\kappa}^\phi_E)^{-1} L_j} \leq C\eps. 
        \end{equation}
        Indeed, 
        \begin{align*}
            \sum_{j=1}^m |L_j - \ol{\kappa}^\phi_E |I_j|| &\leq \ps{\sum_{j=1}^m |I_j|}^{1/2} \ps{\sum_{j=1}^m \ps{L_j - \ol{\kappa}^\phi_E |I_j|}^2\inv{|I_j|} }^{1/2} \leq P(E)^{1/2} \eps. 
        \end{align*}

    \end{enumerate}
    
    It follows from observation 1 and \eqref{eq:QAT edge length} that $d_H(E_k, W_\phi(x_k,(\ol{\kappa}^\phi_E)^{-1})) \leq C\eps$ for some $x_k\in\R^2$. Moreover there exist $r_k \in [(\ol{\kappa}^\phi_E)^{-1}-C\eps,(\ol{\kappa}^\phi_E)^{-1}+C\eps]$ such that $|E_k| = |F_k|$ where $F_k := W_\phi(x_k,r_k)$. Letting $\tilde{E}_k \spq E_k$ be the translate of $F_k$ given by extending the positive edges of $E_k$, we notice that $P_\phi(\tilde{E}_k) = P_\phi(E_k)$ by $\phi$-minimality and that $|\tilde{E}_k\setminus E_k| \leq C\eps^4$ by observation 1. Applying \cref{lem:translated Wulff} to $\tilde{E}_k$, we have the estimate
    \begin{align}
        \abs{P_\phi(E) - P_\phi\ps{\cup_{k=1}^d F_k}} &\leq \sum_{k=1}^d |P_\phi(E_k) - P_\phi(F_k)|\notag\\
        &= \sum_{k=1}^d |P_\phi(\tilde{E}_k) - P_\phi(F_k)|\notag\\
        &\leq \sum_{k=1}^d \inv{r_k}(||\tilde{E}_k| - |F_k|| + C\eps^2)\notag\\
        &\leq C\eps^2. \label{eq:QAT rk}
    \end{align}

    Finally, we claim the desired estimate $|P_\phi(E) - P_\phi(F)| \leq C\eps^2$ holds for the set $F:= \cup_{k=1}^d W_\phi(x_k,r)$, where $r$ is chosen so that $|E|=|F|$. By \eqref{eq:QAT rk}, it suffices to show $|P_\phi(F) - P_\phi(\cup_{k=1}^d F_k)| \leq C\eps^2$, which is equivalent to showing
    \begin{equation}
        \label{eq: rk vs r}
        \bigg| dr - \sum_{k=1}^d r_k\bigg| \leq C\eps^2.
    \end{equation}
    The previous inequality can be derived algebraically using the identity $dr^2 = \sum_{k=1}^d r_k^2$ as in \cite{Julin2022} or \cite{KimKwon2024}.
\end{proof}

\begin{corollary}
    \label{cor:QAT}
    Under the same assumptions as \cref{prop:QAT} and letting $P_d := 2\sqrt{|W_\phi|md}$ be the surface energy of $d$ Wulff shapes of area $m/d$, then for $C = C(\phi,m,M)$, 
    \begin{equation}
        \label{eq:QAT cor}
        \min_{d\in\N} |P_\phi(E) - P_d| \leq C\|\kappa^\phi_E - \ol{\kappa}^\phi_E\|_{L^2(\partial E)}^2.
    \end{equation}
\end{corollary}
\begin{proof}
    For sets such that $\|\kp_E - \okp_E\|_{L^2} \leq \eps$, where $\eps$ is the constant from \cref{prop:QAT}, we are done by \cref{prop:QAT}. Otherwise, \eqref{eq:QAT cor} holds trivially with the constant $C := M/\eps^2$.
\end{proof}

\begin{lemma}
    \label{lem:L2 dev bound}
    Let $E\in\argmin \cF_\tau(\cdot, F)$. Then for $C =C(L_\phi)$, \[ \int_{\partial^*E} |\kappa^\phi_E - \ol{\kappa}^\phi_E|^2 dP_\phi \leq \frac{C}{\tau^2}\cD(E,F). \]
\end{lemma}
\begin{proof}
    Letting $I$ index over all edges in $\partial E$, we may bound
    \begin{align*}
        \|\kappa^\phi_E- \ol{\kappa}^\phi_E\|_{L^2(\partial E, dP_\phi)}^2 &\leq \|\kappa^\phi_E- \lambda\|_{L^2(\partial E, dP_\phi)}^2 = \sum_I (\kappa^\phi_E(I) - \lambda)^2 P_\phi(I) \\
        &= \sum_I \inv{P_\phi(I)} \bigg(\int_I \frac{\sdp_F}{\tau}dP_\phi \bigg)^2 \qquad\text{by }\eqref{eq:EL flat}\\
        &\leq \sum_I \int_I \bigg(\frac{\sdp_F}{\tau}\bigg)^2 dP_\phi \qquad\text{by Cauchy-Schwarz}\\
        &= \inv{\tau^2} \int_{\partial E} (\dd_F^\phic)^2 d P_\phi
    \end{align*}
    and \eqref{eq:L2 estimate cr} concludes the proof.
\end{proof}

\medskip

\noindent \textbf{Proof of \cref{thm:convergence flat flow cr}}:

\medskip

Let $E(t)$ be a flat flow of $E_0$ and $\Etn(t)$ a sequence of approximate flat flows converging to $E(t)$. Let $m = |E_0|$. In what follows, $C$ will be a constant which may change from line to line and depends only on $\phi$, $E_0$, and the flat flow $E(t)$. $C_0$ is a constant with the same dependencies but will remain static. We will frequently pass to subsequences of $\tau_n$ without relabeling.

Due to the dissipation inequality \eqref{eq:dissipation}, the functions \[ p_n(t) := P_\phi(\Etn(t)) + \inv{\tau^{1/2}}||\Etn(t)| - m| \]
are monotone decreasing, and thus there exists a subsequential pointwise limit $p(t)$. 

We assume $p(t)$ is not eventually constant (the alternative can be argued exactly as in \cite{Julin2022}). Recall that $P_d := 2\sqrt{|W_\phi|md}$ is the surface energy of $d$ Wulff shapes of area $m/d$. Note that $\lim_{t\to\infty} p(t) \in [P_d, P_{d+1})$ for some $d$, and thus there exists $\delta>0$ and $T_0>0$ such that 
\begin{equation}
    \label{eq:Pd}
    P_d < p_n(t) < P_{d+1}-\delta \quad \text{for all } t\geq T_0.
\end{equation}

\textbf{Step 1.} Here we show the exponential decay of cumulative dissipations. The argument is identical to those found in \cite{Julin2022, KimKwon2024}, though we repeat it here to highlight the application of the quantitative Alexandrov theorem. More precisely, we claim that for fixed $T > T_0$ and for all $n\geq n_0(T)$ and $t \in [T_0,T]$,
    \begin{equation}
        \label{eq:dissipation decay}
        \inv{\tau_n} \sum_{j=\floor{t/\tau_n}}^{\floor{T/\tau_n}-1} \cD(\Etn_{j+1},\Etn_j) \leq 2P_\phi(E_0) e^{-t/C_0}.
    \end{equation}
    
    First we show that the lefthand side obeys a discrete differential inequality \begin{equation}
        \label{eq:dissipation diff ineq}
        \sum_{j=\floor{t/\tau_n}}^{\floor{T/\tau_n}-1} \cD(\Etn_{j+1},\Etn_j) \leq \frac{C}{\tau_n}\cD(\Etn(t),\Etn(t-\tau_n))
    \end{equation}
    for all $t\in [T_0,T]\setminus Z_n$ where $Z_n := \{t\in [T_0,T]: |\Etn(t)| \neq m\}$. Indeed, whenever $|\Etn(t)| = m$, \eqref{eq:dissipation diff ineq} follows from the iterated dissipation inequality \eqref{eq:dissipation} and the quantitative Alexandrov theorem:
    \begin{align*}
       \inv{\tau_n}\sum_{j=\floor{t/\tau_n}}^{\floor{T/\tau_n}-1} \cD(\Etn_{j+1},\Etn_j) &\leq P_\phi(\Etn(t)) - P_\phi(\Etn(T))\\
        &\leq P_\phi(\Etn(t)) - P_d \by\eqref{eq:Pd}\\
        &\leq C\|\kappa^\phi_{\Etn(t)} - \ol{\kappa}^\phi_{\Etn(t)}\|_{L^2(\partial \Etn(t))}^2 \by \text{\cref{cor:QAT}} \\
        &\leq \frac{C}{\tau_n^2} \cD(\Etn(t), \Etn(t-\tau_n)) \by\text{\cref{lem:L2 dev bound}}.
    \end{align*}
    The exponential decay \eqref{eq:dissipation decay} then follows from applying the algebraic lemma \cite[Lemma 3.1]{Julin2022} to the lefthand side of \eqref{eq:dissipation diff ineq} and the fact that $|Z_n| \leq C\tau_n$. 

\medskip

\textbf{Step 2.} Using the $L^1$ estimate \eqref{eq:L1 estimate cr} in tandem with \eqref{eq:dissipation decay}, it is standard to show exponential $L^1$-convergence of the flat flow: there exists a set $E_\infty$ such that for $t\geq T_0$,
    \begin{equation}
        \label{eq:exp L1 convergence II}
        |E(t)\Delta E_\infty| \leq Ce^{-t/2C_0}.
    \end{equation}
Let $\cC(d,m)$ denote the collection of sets $F\sb\R^2$ which are a disjoint union of $d$ equally sized Wulff shapes such that $|F|=m$. 
Let us now prove \eqref{eq:exp L1 convergence} by showing that $E_\infty$ belongs to $\cC(d,m)$. 

For fixed $t$ such that $t-e^{-t/4C_0} \geq T_0$, we may invoke \cref{lem:L2 dev bound} and \eqref{eq:dissipation decay} to bound
    \begin{align*}
        \int_{t-e^{-t/4C_0}}^t \|\kp_{\Etn(s)} - \okp_{\Etn(s)}\|_{L^2}^2 ds & \leq \frac{C}{\tau_n^2} \int_{t-e^{-t/4C_0}}^t \cD(\Etn(s), \Etn(s-\tau_n))ds \leq Ce^{-t/C_0}.
    \end{align*}
By applying Markov's inequality to the previous bound, and taking $n$ sufficiently large so that 
    \[ \big|\{s\in [t-e^{-t/4C_0}, t]: |\Etn(s)| \neq m\}\big| \leq C\tau_n \leq \inv{2} e^{-t/4C_0}, \]
there exists some $t_n \in [t-e^{-t/4C_0},t]$ such that $|\Etn(t_n)| = m$ and $\|\kp_{\Etn(t_n)} - \okp_{\Etn(t_n)}\|_{L^2} \leq Ce^{-t/4C_0}$. 

For $T_0$ sufficiently large, the QAT in the form of \cref{prop:QAT} implies 
    \begin{align*}
        d_H(\Etn(t_n), F_n(t)) &\leq Ce^{-t/4C_0}
    \end{align*}
for some $F_n(t) \in \cC(d,m)$, and thus $|\Etn(t_n) \Delta F_n(t)| \leq Ce^{-t/4C_0}$. Since $F_n(t)$ is bounded by \cref{prop:lambda L2 bound}(ii), upon passing to a subsequence we have $t_n\to t'\in [t-e^{-t/4C_0}, t]$ and $F_n(t)$ converges in $L^1$ to some set $F(t)$. Note that $F(t)\in \cC(d,m)$ since $\cC(d,m)$ is closed under $L^1$ convergence. Moreover, $\Etn(t_n)\to E(t')$ in $L^1$ by the H\"older continuity in time \eqref{eq:Holder cty in time tau}: 
\[ 
    \limsup_{n\to\infty} |E(t') \Delta \Etn(t_n)| \leq \limsup_{n\to\infty} |E(t') \Delta \Etn(t')| + C|t'-t_n|^{1/2} = 0. 
\]
Thus $|E(t') \Delta F(t)| \leq Ce^{-t/4C_0}$, and we may bound
\begin{equation}
    |F(t)\Delta E_\infty| \leq |F(t) \Delta E(t')| + |E(t')\Delta E_\infty| \leq Ce^{-t/4C_0}.
\end{equation}
Now seeing $t$ as arbitrary, we find that $E_\infty$ is an $L^1$ limit point of $\cC(d,m)$, and thus $E_\infty \in \cC(d,m)$.

\medskip 

\textbf{Step 3.} It remains to show that if the Wulff shapes in $E_\infty$ have positive pairwise distance, then $E(t)$ eventually becomes the ODE evolution and satisfies the stronger convergence \eqref{eq:exp hausdorff convergence}.

We can express $E_\infty = \cup_{j=1}^d W_\phi(x_j,r)$ and assume $\delta := \min_{i\neq j} d(W_\phi(x_i,r), W_\phi(x_j,r)) > 0$. We fix a time $t_0 \geq T_0$ and the length scale $\eps := e^{-t_0/4C_0}$. Repeating the construction from step 2 and recalling \cref{prop:QAT}, we may find a sequence $\Etn(t_n)$ such that $t_n\in [t_0-\eps, t_0]$ converges to some $t'$, $d_H(\Etn(t_n), E_\infty) \leq C\eps$, and $\Etn(t_n)$ is a rough translate of $E_\infty$ whose fuzzy edges have length $O(\eps)$. By Arzel\`a-Ascoli, we may pass to a subsequence such that $\Etn(t_n)$ converges in Hausdorff distance to $E(t')$. More specifically, we apply standard compactness to the positive edges of $\Etn(t_n)$ and Arzel\`a-Ascoli to $O(\eps)$ neighborhoods of corners/fuzzy edges, which may be seen as uniformly Lipschitz graphs over $E_\infty$.

By setting $t_0$ sufficiently large so that $e^{-t_0/4C_0} \ll \delta$, we can ensure that all $O(\eps)$ perturbations of $E_\infty$ still consist of $d$ disjoint components. In particular, the ODE evolutions of each $\Etn(t_n)$ exist for at least $2\eps$ time. By \cref{thm:weak strong uniqueness} and \cref{cor:ODE stability hausdorff 2}, we may estimate for $t\in (t', t_0 + \eps]$
\begin{align*}
    d_H\big(\Etn(t), \Psi_{t-t'}(E(t'))\big) &\leq d_H\big(\Etn(t), \Psi_{t-t_n}(\Etn(t_n))\big) + d_H\big(\Psi_{t-t_n}(\Etn(t_n)), \Psi_{t-t'}(E(t'))\big)\\
    &\leq Ce^{C(t-t_n)}(\tau_n^{1/2} + d_H(\Etn(t_n), E(t')) + |t'-t_n|).
\end{align*}
Thus as $n\to\infty$, $\Etn(t)$ converges to $\Psi_{t-t'}(E(t'))$ in Hausdorff distance and in $L^1$, so $E(t) = \Psi_{t-t'}(E(t'))$. In particular, $E(t) = \Psi_{t-t_0}(E(t_0))$ for all $t\in [t_0, t_0 + e^{-t_0/4C_0}]$. Because $t_0$ is arbitrary, we conclude that $E(t)$ eventually coincides with the ODE flow.

To verify the estimate \eqref{eq:exp hausdorff convergence}, we observe that for all large enough $t$, the bound $d_H(E(t'), E_\infty) \leq Ce^{-t/4C_0}$ holds for some $t'\in [t- e^{-t/4C_0},t]$, and thus $d_H(E(t), E_\infty) \leq Ce^{-t/4C_0}$ since solutions of the ODE flow are Lipschitz in time. \hfill $\Box$

\appendix
\section{Anisotropic identities}

\begin{lemma}[Gauss-Bonnet]
    \label{lem:gauss bonnet}
    Let $E\sb\R^2$ be a simply connected Lipschitz $\phi$-regular set. Then \begin{equation}
        \label{eq:gauss bonnet}
        \int_{\partial^* E} \kappa^\phi_E dP_\phi = 2|W_\phi| = P_\phi(W_\phi).
    \end{equation}
\end{lemma}
\begin{proof}
    Let $I_1, \dots, I_m$ be the regular edges of $\partial E$. Then \[ \int_{\partial^*E}\kappa^\phi_E dP_\phi = \sum_{j=1}^m \int_{I_j} \kappa^\phi_E(I_j) dP_\phi = \sum_{j=1}^m \alpha_{I_j} \phi(\nu_{I_j}). \]
    Since the vectors $\nu_{i_j}$ are adjacent with respect to $j$, a simple combinatorial argument yields $\sum_{i_j = i} \sigma_{I_j} = 1$ for each $i$, and hence \[ \int_{\partial^* E} \kappa^\phi_E dP_\phi = \sum_{i=1}^{N} L_i \phi(\nu_i) = \sum_{i=1}^{N} |p_i - p_{i-1}| \phi(\nu_i) = 2|W_\phi| \]
    where the last inequality follows from recalling $W_\phi = \{x: x\cdot \nu_i \leq \phi(\nu_i)\ \forall i\}$ and that $[p_{i-1}, p_i]$ is the edge of $W_\phi$ with outer normal $\nu_i$.
\end{proof}

\begin{lemma}[Anisotropic coarea formula]
    Let $u$ be Lipschitz on an open domain $A\sb\R^2$. Then
    \begin{equation}
        \label{eq:anisotropic coarea}
        \int_A \phi(\nabla u(x))dx = \int_{-\infty}^\infty P_\phi(\{u\leq t\};A) dt.
    \end{equation}
    Moreover, for an integrable function $g:A\to\R$, 
    \begin{equation}
        \label{eq:weighted coarea}
        \int_A g(x) \phi(\nabla u(x))dx = \int_{-\infty}^\infty \int_{\{u=t\}\cap A} g\, dP_\phi dt
    \end{equation}
    where $\{u=t\}$ is given the orientation of $\{u\leq t\}$.
\end{lemma}

\begin{lemma}
\label{lem:eikonal}
Let $F\sb\R^2$ be a closed set of finite perimeter. Then $\sdp_F$ satisfies the anisotropic Eikonal equation $\phi(\nabla \sdp_F) = 1$ a.e. In particular, for any Borel set $E\sb\R^2$, 
\begin{equation}
    \label{eq:sdf coarea}
    |E| = \int_{-\infty}^\infty P_\phi(\{\sdp_F\leq t\};E)dt.
\end{equation}
\end{lemma}
\begin{proof}
    Since $\sdp_F$ is Lipschitz, it suffices to show the result at points $x$ where $\sdp_F$ is differentiable, and since $|\partial F| = 0$ we may further assume $x\not\in \partial F$. 
    
    Assume that $\sdp_F(x) < 0$, in which case $\sdp_F(x) = -\sup\{r>0: W_\phi(x,r) > 0\}$. For any $v\in\R^2$ such that $\phic(v)=1$ and $0 < t < r < |\sdp_F(x)|$, we have $W_\phi(x+tv, r-t) \sb W_\phi(x, r)$ and thus 
    \[ |\sdp_F(x+tv)| \geq |\sdp_F(x)| - t \qquad\implies\qquad \sdp_F(x+tv) \leq \sdp_F(x) + t. \]
    Sending $t\to0$, it follows that $\nabla \sdp_F(x) \cdot v \leq 1$. In particular, 
    \[ \phi(\nabla \sdp_F(x)) = \sup_{\phic(v)=1} \nabla\sdp_F(x) \cdot v \leq 1. \]
    Moreover, there exists $y\in \partial F$ such that $y \in \partial W_\phi(x,|\sdp_F(x)|)$. Setting $v := \frac{y-x}{|\sdp_F(x)|}$, then we have $\phic(v) =1$ and $\sdp_F(x+tv) = \sdp_F(x) + t$, so $\sdp_F(x)\cdot v = 1$, and hence $\phic(\nabla \sdp_F(x))=1$ as desired.

    One argues similarly in the case $\sdp_F(x) > 0$ using the definition $\sdp_F(x) = \sup\{r>0: -W_\phi(x,r) \sb F^c\}$. The identity \eqref{eq:sdf coarea} follows from \eqref{eq:anisotropic coarea}.
\end{proof}

\begin{lemma}
    \label{lem:translate perimeter}
    Let $F\sb\R^2$ be Lipschitz $\phi$-regular, and let $F(h)$ be a translate of $F$ as defined in \cref{sec:ODE flow}.
    Then 
    \begin{equation}
        \label{eq:translate perimeter A}
        P_\phi(F(h)) = P_\phi(F) + \sum_{I\sb\partial F \text{ regular}} \alpha_I \phi(\nu_I) h_I.
    \end{equation}
\end{lemma}
\begin{proof}
    It is straightforward to check by $\phi$-minimality that $P_\phi(F(h))$ is constant with respect to $h_J$ for any fuzzy edge $J\sb\partial F$. Thus, it makes no harm to assume that $F$ has no type I fuzzy edges, and that type II fuzzy edges are line segments with outer normal in $\cN_\phi$. Moreover, it suffices to show \eqref{eq:translate perimeter A} holds for sufficiently small $h$. 

    Fix a regular edge $I\sb\partial F$ and suppose without loss of generality that $I$ is positive and $\nu_I = \nu_1$. For small enough $t > 0$, the set $U := F(te_I)\setminus F$ is a trapezoidal region with parallel sidelengths $I$ and $I(te_I)$. There exists a vector field $X\in C^1_c(\R^2;\R^2)$ such that $X = p_1$ along $I_1(te_I)\cap \partial U$, $X = p_0$ along $I_2(te_I)\cap \partial U$, and $X$ interpolates linearly along line segments in $U$ parallel to $I$. By a similar application of the divergence theorem as in \eqref{eq:sharp minimality div}, we find 
    \begin{equation*}
        P_\phi(F(te_I)) - P_\phi(F) = P_\phi(U\setminus I) - P_\phi(I) = \int_U \div X\,dx = |p_1 - p_0| \phi(\nu_1) t = L_I \phi(\nu_I)t.
    \end{equation*}
    By reversing the roles of $F$ and $F(te_I)$ in the case $t<0$, we find that for all $t$, 
    \begin{equation*}
        P_\phi(F(te_I)) = P_\phi(F) + L_I \phi(\nu_I) t. 
    \end{equation*}
    A similar argument holds in the case where $I$ is negative, yielding 
    \begin{equation*}
        P_\phi(F(te_I)) = P_\phi(F) + \alpha_I \phi(\nu_I) t.
    \end{equation*}
    By iterating the prior argument over all regular edges, we obtain \eqref{eq:translate perimeter A}.
\end{proof}

\begin{lemma}
    \label{lem:translated Wulff}
    Let $E = W_\phi(h)$ be a translate of the Wulff shape as defined in \cref{sec:ODE flow} such that $|h| < \inv{2}$. Then 
    \begin{equation}
        \label{eq:translated Wulff}
        |E| - |W_\phi| = P_\phi(E) - P_\phi(W_\phi) + O(|h|^2)
    \end{equation}
    where the implicit constant depends only on $\phi$.
\end{lemma}
\begin{proof}
    It is straightforward to check the asymptotic 
    \begin{equation*}
        |E| = |W_\phi| + \sum_{I\sb\partial W_\phi} L_I \phi(\nu_I)h_I + O(|h|^2).
    \end{equation*}
    Then \eqref{eq:translated Wulff} follows from \cref{lem:translate perimeter}.
\end{proof}

\begin{lemma}[Localized projection]
    \label{lem:projection}
    Let $E\sb\R^n$ be a set of finite perimeter, $\phi$ an anisotropic surface energy, and $\Omega\sb\R^{n-1}$ an open domain. Let also $f,g:\Omega\to\R$ be bounded continuous functions such that $f<g$, and consider the domain $U := \{(z,y): z\in \Omega, y\in(f(z),g(z))\}$. For $z\in\Omega$, define the \emph{slice} $E_z := \{y\in(f(z),g(z)): (z,y)\in E\}$. Suppose there exists $\delta>0$ such that $E_z$ is an interval whenever $d(z,\partial\Omega) < \delta$, and such that $(f(z),f(z)+\delta) \sbq E_z \sbq (f(z), g(z)-\delta)$ for all $z\in\Omega$. Define the projection $\tilde{E}\sb\R^2$ such that $\tilde{E}$ and $E$ coincide outside of $U$, and for all $z\in\Omega$, $\tilde{E}_z = (f(z),f(z)+\cH^1(E_z))$. Then $P_\phi(\tilde{E}) \leq P_\phi(E)$.
\end{lemma}
\begin{proof}
    Since $\tilde{E}\Delta E \Subset U$, it suffices to show $P_\phi(\tilde{E};U) \leq P_\phi(E;U)$. By a standard approximation argument (see \cite[Remark 13.13]{Maggi2012}) and Reshetnyak's continuity theorem, it suffices to show the result for polyhedral $E$ such that $\nu_E \cdot e_n \neq 0$ everywhere. It also makes no harm to assume that $f$ is piecewise affine.

    We may decompose $\Omega$ into relatively polyhedral regions $\Omega_1, \dots, \Omega_m \sb\Omega$ such that over each $\Omega_k$, there are piecewise affine functions $f = f_1 < g_1 < \cdots < f_{N_k} < g_{N_k}$ such that 
    \begin{equation*}
        E \cap U = \bigcup_{k=1}^m \bigg\{
            (z,y): z\in \Omega_k, y\in \bigcup_{j=1}^{N_k} (f_j(z), g_j(z))
        \bigg\}.
    \end{equation*}
    Then for $z\in\Omega_k$, $\cH^1(E_z) = \sum_{j=1}^{N_k} (f_j(z) - g_j(z))$, and by convexity of $\phi$, we may bound
    \begin{align*}
        P_\phi(E; U) &= \sum_{k=1}^m \int_{\Omega_k} \bigg(\phi((-\nabla g_1,1)) + \sum_{j=2}^{N_k} \phi((-\nabla f_j,1)) + \phi((\nabla g_j,-1))\bigg)dz \\
        &\geq \sum_{k=1}^m \int_{\Omega_k} \phi\bigg(-\nabla \Big(g_1 + \sum_{j=2}^{N_k} (f_j - g_j)\Big), 1 \bigg) dz\\
        &= \sum_{k=1}^m \int_{\Omega_k} \phi\big(-\nabla (f(z) + \cH^1(E_z)), 1\big)dz\\
        &= P_\phi(\tilde{E};U). \qedhere
    \end{align*}
\end{proof}

\bibliographystyle{alpha}
\bibliography{main}

\end{document}